\documentclass[11pt,letterpaper,reqno]{amsart}
\usepackage{libertine}

\usepackage{amsmath,amsthm,amssymb,amsfonts,epsfig,color,graphicx,enumerate,enumitem,accents}
\usepackage{amsbsy,setspace,yfonts} 
\usepackage{rotating,multirow,color,url,mathrsfs}
\usepackage{arydshln}
\usepackage{bbm}
\usepackage{cases}
\usepackage{esint}
\usepackage{gensymb}
\usepackage{epstopdf}
\usepackage{comment}
\usepackage{leftidx}
\usepackage{graphicx}
\usepackage[normalem]{ulem}
\usepackage{stmaryrd}
\usepackage{mathtools}
\usepackage{subfig}
\usepackage{amsaddr}
\usepackage[rightcaption]{sidecap}

\usepackage[utf8]{inputenc}
\usepackage[T1]{fontenc}

\usepackage{scalerel}
\usepackage{bbm}
\usepackage{tikz}
\usepackage{IEEEtrantools}
\usepackage[hidelinks]{hyperref}

\usepackage{thmtools}
\usepackage{thm-restate}
\usepackage[dvipsnames]{xcolor}

\newtheorem{theorem}{Theorem}[section]
\newtheorem{corollary}[theorem]{Corollary}
\newtheorem{lemma}[theorem]{Lemma}
\newtheorem{proposition}[theorem]{Proposition}

\theoremstyle{definition}
\newtheorem{definition}[theorem]{Definition}
\newtheorem{remark}[theorem]{Remark}
\newtheorem{example}[theorem]{Example}

\theoremstyle{definition}

\newtheorem*{remark*}{Remark}
\newtheorem*{theorem*}{Theorem}
\newtheorem*{example*}{Example}

\numberwithin{equation}{section}

\newcommand{\weakstar}{\overset{*}\rightharpoonup}

\definecolor{ARbrown}{RGB}{120,60,10}	 

\def\R{{\mathbb R}}
\def\NN{{\mathbb N}}
\def\Rd{\R^d}
\def\N{{\mathcal N}} 
\def\V{{\mathcal V}} 
\def\W{{\mathcal W}}
\def\F{{\mathcal F}}
 
\def\A{{\mathcal A}}
\def\G{{\mathcal G}}

\def\Ld{{\mathcal L}^d}
\def\M{{\mathcal M}} 
\def\Md{\M(\Rd;\Rd)}				 
\newcommand{\Mk}{\M(\Rd)^k}	 
\newcommand{\Grass}{\mathrm{Gr}(d)} 
 
\def\H{{\mathcal H}} 
\def\L{{\mathcal L}} 
\def\B{{\mathcal B}}
\def\PP{{\mathcal P}}
\def\D{{\mathcal D}}
\def\S{{\mathcal S}}

\def\T{{\mathcal T}}

\def\res{\mathop{\hbox{\vrule height 7pt width.5pt depth 0pt\vrule height .5pt width 6pt depth 0pt}}\nolimits}
\def\div{\mathop {\rm div}\nolimits}
\def\spt{\mathop {\rm spt}\nolimits}
\def\eps{\varepsilon}

\def\f{\varphi}

\newcommand{\sig}{\sigma}
\newcommand{\dist}{\mathrm{dist}}

\newcommand{\ov}{\overline}

\newcommand{\la}{\lambda}
\newcommand{\AEd}{\text{\bf\AE}(\Rd)}

\newcommand{\Mod}{\M_{0,1}(\Rd)}
\newcommand{\Fdp}{\mathbf {F^\perp_1}(\Rd)}
\newcommand{\Fd}{\mathbf{F_1}(\Rd)}
\newcommand{\Nd}{\mathbf {N_1}(\Rd)}
\newcommand{\set}[2]{\left\{\, #1 \ \textup{\textbf{:}}\ #2 \,\right\}}
\newcommand{\Sdd}{{\mathcal{S}^{d \times d}}}
\newcommand{\KR}{\mathrm{KR}}
\newcommand{\st}{\ : \ }
\newcommand{\diffquot}[2]{D_{#2}^{#1}\mu}

\newcommand{\pairing}[1]{{\left \langle #1 \right \rangle}}
\newcommand{\one}{{{\bf 1}
		\kern-0,28em \rm l}} 

\newcommand{\Lip}{{\mathrm{Lip}}}

\def\x0s{{\bf X}_0^\sharp (\Omega)}

\def\ds{\displaystyle}
\def\med{\medskip\noindent}

\definecolor{projmag}{RGB}{190,0,120}

\setlist[itemize,1]{label=$-$}
\setlist[itemize,2]{label=$-$}
\setlist[itemize,3]{label=$-$}

\setlist[enumerate]{label={(\roman*)},labelsep=6pt}

\title[{The Monge-Kantorovich tangent bundle}]{The Monge-Kantorovich tangent bundle to a measure}

\author[A. Arroyo-Rabasa, A. Aussedat, G. Bouchitt\'e]{%
	Adolfo Arroyo-Rabasa\textsuperscript{\dag}, \
	Averil Aussedat\textsuperscript{\dag}, \
	Guy Bouchitt\'e\textsuperscript{*}}

\address{\textsuperscript{\dag}\,Dipartimento di Matematica, Universit\`a di Pisa, Largo Pontecorvo 5, 56127 Pisa, Italy}
\email{adolfo.rabasa@unipi.it}
\email{averil.aussedat@dm.unipi.it}

\address{\textsuperscript{*}\,Laboratoire IMATH, Universit\'e de Toulon, BP 20132, 83957 La Garde Cedex, France}
\email{bouchitte@univ-tln.fr}
\date{\today}

\begin{document}
	
\begin{abstract}
	We study the tangent bundle $T_\mu$ to a Radon measure $\mu$ in Euclidean space, introduced by Bouchitt\'e, Champion and Jimenez (2005). Its construction, inspired by Monge-Kantorovich optimal transport theory, involves the duality between Lipschitz functions and the so-called Arens-Eells space, a Banach subspace of distributions obtained by completing the set of balanced signed measures.
	Precisely, a velocity field $\sigma$ is $\mu$-tangent when the divergence of $\sigma \mu$ lies in the Arens-Eells space. The tangent bundle $T_\mu$ defined $\mu$-almost everywhere provides a local projection that allows to construct a $\mu$-tangential gradient operator on Lipschitz functions that is weakly continuous and enjoys integration by parts.
	In this paper, we introduce a new quantitative estimate involving the tangential and normal components of a given velocity vector field $\sig \in L^1_\mu(\Rd)$. Specifically, we show that the tangential condition $\sig \in T_\mu$ holding $\mu$ a.e. is equivalent to each of the following two conditions: the convergence in the Arens-Eells space of $h^{-1} \left((id+h \sig)_{\#} \mu - \mu\right)$ as $h \to 0$, and the differentiability of Lipschitz functions along $\sig$.
	Moreover, given a field of non-tangent directions, we construct a Lipschitz function that is not differentiable on a large set, confirming that $T_{\mu}$ agrees with the decomposability bundle of Alberti and Marchese. We finally relate $T_\mu$ to the tangent measures of Preiss and survey further properties of the tangential differential calculus.
	
	\bigskip
	\begin{flushleft}
	
	
	\bigskip
	\textbf{AI statement.\,} AI tools were used exclusively for minor language assistance. No AI tools were used for the derivation, proof, calculation, or logical reasoning of the mathematical results presented in this paper.
\end{flushleft}
\end{abstract}

\maketitle


\tableofcontents

\section{Introduction}

The study of calculus on singular or nonsmooth structures is a central theme at the intersection of real analysis, geometric measure theory, and partial differential equations. When working with a smooth manifold, the tangent bundle provides a local frame for defining gradient vector fields and integration by parts. However, when the underlying domain is represented by a general Radon measure $\mu$ in $\mathbb{R}^d$, establishing an intrinsic differential calculus becomes significantly more delicate. A fundamental difficulty in this setting is to extend classical differential operators, such as the gradient or divergence, to $L^p_\mu$ spaces. In standard Sobolev spaces, the underlying measure $\mu$ is the Lebesgue measure and integration by parts ensures that smooth derivations are closable, thus enabling the definition of weak derivatives. This is not the case for a singular measure $\mu$, for which closability can only be restored by pointwise projecting the classical gradient onto a suitable tangent bundle associated with $\mu$.
This pointwise tangent space that cannot be recovered from Preiss' tangent measures: in general the tangent cone $\mathrm{Tan}(\mu,x)$ may contain non unique tangent planes or none at all, and still contain arbitrary sets of measures (see \cite{Preiss87,Neil}).

A natural alternative for constructing such a bundle comes from the calculus of variations: it consists in applying the \emph{De Giorgi} relaxation procedure to integral functionals of the form $\int |\nabla u|^p \, d\mu$. In this manner a tangent bundle was introduced in \cite{BBS} and further developed in a series of papers \cite{B-B-S, BF, rev} (see also \cite{FraMant,Zhikov} and the connection with Sobolev spaces in metric measure spaces in \cite{lucicCharacterisationUpperGradients2021}). It is defined $\mu$ a.e., with a dimension depending on the local geometry of $\mu$ as well as on the exponent $p\in [1,\infty)$. Its good behavior with respect to duality and integration by parts has allowed to treat successfully various applications in applied physics, such as the homogenization of thin structures \cite{BF-2scale}, dimension reduction, and optimal design problems in which the measure $\mu$ itself is an unknown minimizing a certain energy (see for example \cite{B-B, bouchitte2007, BFS} and \cite{KBGB}).

Despite their generality and flexibility, these variational constructions have had little impact on geometric measure theory.
In this paper, we show that the variant based on optimal transport yields a tangent bundle $T_\mu(x)$ which retains the variational and duality properties mentioned above and, in addition, is intimately connected with the theory of one-dimensional flat chains and with the differentiability of Lipschitz functions. First introduced in 2005 (see \cite{BCJ2005}), the bundle $T_\mu(x)$ ultimately agrees with the decomposability bundle $V_\mu(x)$ of Alberti and Marchese \cite{AM}, who proved that $V_\mu$ is also the largest \emph{differentiability bundle}: for any $u \in \Lip(\R^d)$, there exists a linear function $d_V u(x): V_\mu(x) \to \R$ at $\mu$-almost every point $x$ for which
\begin{equation*}
	 u(x + v) = u(x) + d_Vu(x)v + o(|v|) \qquad \text{for $v \in V_\mu(x)$}.
\end{equation*}

\paragraph{\bf Construction of the bundle $T_\mu(x)$ } Given a nonnegative Radon measure $\mu$ on $\Rd$, we begin by defining
a subspace of vector fields $\sig\in (L^1_\mu)^d$ which satisfy a suitable tangency condition. As will become clear, this condition encodes the first-order stability of the Monge-Kantorovich distance between $\mu$ and its infinitesimal pushforward along the velocity field $\sig$:
\begin{align*}
	\mu_h^\sig := (\mathrm{id} + h\sig)_\sharp \mu.
\end{align*}
In turn, the tangent velocity fields form a subspace $\T_\mu$ of $(L^1_\mu)^d$ characterized by a divergence condition: $\sig$ belongs to $\T_\mu$ if and only if, in the sense of distributions,
\begin{equation}\label{eq:regularity}
	\mathrm{div} (\sigma \mu) \in \AEd.
\end{equation}
Here $\AEd$, the Arens-Eells space, denotes the completion, with respect to the Kantorovich-Rubinstein norm \cite{CR} (denoted $\Vert \cdot\Vert_{\KR}$), of the space of signed measures with finite first moment and zero average. This Banach space, known as the canonical pre-dual of $\mathrm{Lip}_0(\mathbb{R}^d)$, has been rediscovered as many times as there are viewpoints on it: introduced by Arens and Eells \cite{AE} in 1956, it reappears as a distinguished space of first-order distributions in the completion problem for the Kantorovich norm \cite{Hanin1997,BCJ2005,BBDP2010} (see its characterization in \eqref{criterion}), while in non linear Banach space geometry, it becomes the Lipschitz-free space \cite{GK03, WeaverBook}.

The interest of allowing the right hand-side of \eqref{eq:regularity} to be a first-order distribution in $\AEd$, instead of a measure, is that the resulting subspace $\T_\mu$ is closed in $(L^1_\mu)^d$ and enjoys the key locality property that, if $\sigma\in \T_\mu$, then $ \theta \sigma\in \T_\mu $ for every scalar function $\theta\in L_\mu^\infty$. These two properties imply the existence of a $\mu$-measurable map $T_\mu: \Rd \to \mathcal{G}_d$ (the Grassmannian of linear subspaces of $\Rd$) such that
\begin{equation}\label{Tmu-local}
	\sig\in \T_\mu \iff \sig(x)\in T_\mu(x) \ \mu\ {\rm a.e.}
\end{equation}
Summarizing the Monge-Kantorovich (MK) tangent bundle $T_\mu(x)$ is defined through the PDE condition \eqref{eq:regularity}, which allows the local property \eqref{Tmu-local}. An immediate consequence of this construction is that 
any Borel vector measure $\la\in \M(\Rd;\Rd)$ can be decomposed as the sum $\la=\la_T + \la_N$ of its tangential part $\la_T = (P_\mu \sig) \, \mu$ and of its normal part $\la_N = (P_\mu^\perp \sig) \, \mu$, where $\mu=|\la|$, $\sig= d\la/d|\la|$ and $P_\mu(x), P_\mu^\perp$ denote the orthogonal projections onto $T_\mu(x)$ and $T_\mu^\perp(x)$ respectively.

At this level, we can foresee an important connection with the geometric measure theory: the divergence condition \eqref{eq:regularity} is equivalent to the fact that $\sigma \mu$ being a flat chain of finite mass in the classical sense of Whitney and Federer \cite{Fed,FF60}: 
\[	\T_\mu = \set{ \sigma\in (L^1_\mu)^d }{\sigma \mu \in \Fd} \ , \]
an equivalence which has been used recently to obtain a simple proof of the Ambrosio-Kirchheim Flat Chain Conjecture \cite{AK00} in the context of one dimensional metric currents \cite{ARGB}. 
A consequence of the equality \eqref{Tmu-local} is that the set of flat chains $\Fd$ forms a complemented subspace of $\Md$, a fact that we have not found stated in the literature (see Theorem \ref{projthm}). More precisely, it holds a direct topological sum
$$ \Md = \Fd \oplus \ker {\mathcal P}_{\mathbf F} \ ,$$
where the linear projection map ${\mathcal P}_{\mathbf F}: \la \in \Md \to \la_T \in \Fd$ is defined such that $\la_T$ uniquely solves the minimal distance problem
\[
	\dist(\la, \Fd) = \inf \set{|\la - \eta|(\Rd)}{\eta \in \Fd} \ .
\]
Two features of this statement go beyond the pointwise characterisation recalled above: it is metric, in the sense that $\la_T$ is the \emph{nearest} flat chain and that the distance is computed explicitly as $|\la_N|(\Rd)$; and it is linear, the map $\la \mapsto \la_T$ being a $1$-Lipschitz projection.

Before presenting our main results, we point out that the MK-bundle $T_\mu$ defined above is local (see Subsection \ref{sec:locality}) and enjoys all the properties one would expect: decomposition of $\mu$ into pieces of constant dimension, chain and product rules, stability under Lipschitz push-forward, and the definition of a generalized varifold and of a mean curvature of $\mu$ as an element of $(\AEd)^d$. Some of these properties already appear in \cite{BCJ2005} or are adapted from \cite{BBF1,BBF2}. They are summarized in Section~\ref{sec:further}, together with the equivalence $\mu \ll \L^d$ if and only if $T_\mu(x)=\Rd$, $\mu$ a.e. which we obtain by extending to flat chains the structure theorem of De Philippis and Rindler \cite{PR2016,ADHR}.

\medskip
\paragraph{\bf Main results}

Our main contributions are presented in Section \ref{tangential calculus}. First, following the closability argument of \cite{BCJ2005}, we recall the construction of the tangential gradient operator $u\in \Lip(\Rd)\to \nabla_\mu u \in (L^\infty_\mu(\Rd))^d$ which is a weak-star continuous operator and satisfies the integration by parts formula:
$$ \pairing {-\div \sig\, \mu, u} \ = \ \int \pairing {\sig, \nabla_\mu u}\, d\mu \quad \forall (\sig,u) \in \T_\mu\times \Lip(\Rd).$$
Then we show that, for every pair $(\sig,u) \in \T_\mu\times \Lip(\Rd)$, we have:
\begin{enumerate}[label={(\roman*)}, labelsep=10pt]
	\item\label{item:limIs0} $\ds \lim_{h\to 0^+} \Vert \diffquot{\sig}{h} + \div(\sig\, \mu) \Vert_{\KR} = 0$,  where we set $\diffquot{\sig}{h} := \frac{\mu_h^\sig - \mu}{h}$;
	\item\label{item:formul} $\pairing{\ds \nabla_\mu u(x),\sigma(x)} = \lim_{h\to 0} \frac{u(x+ h\sig(x))- u(x)}{h}$ \ $\mu$ a.e.
\end{enumerate}

Assertion \ref{item:limIs0} is a stability result: for $\sig\in\T_\mu$, the difference quotients $\diffquot{\sig}{h}$ converge in $\AEd$, and their limit is the distributional limit $-\div(\sig\mu)$. The assertion \ref{item:formul} identifies the weak tangential gradient $\nabla_\mu u$ of any Lipschitz function $u\in \Lip(\Rd)$ as the pointwise directional derivative of $u$ in the direction $\sig$.

The two assertions above, contained in our main result (Theorem \ref{MK-infinitesimal}), are proved by starting with measures $\mu$ associated with Lipschitz curves, then by extending to pairs $(\mu,\sigma)$ with $\sig\in \T_\mu$, by using the Smirnov decomposition of normal currents \cite{Smi93} combined with an approximation argument allowing to pass from $\sig\in \T_\mu$ to $\sig \mu$ being a normal current.
Since this limit is linear in $\sig(x)$ and $T_\mu(x)$ is already a linear subspace, we can lift~\ref{item:formul} to differentiability with respect to the whole subspace $T_\mu(x)$:
\begin{equation}\label{eq:fulldiff}
	u(x+v) \ =\ u(x) + \pairing{\nabla_\mu u(x),\, v} + o(|v|)
	\qquad \text{as } T_\mu(x)\ni v \to 0 .
\end{equation}

Next we improve upon the pointwise view of differentiability by testing the previous convergence on vector fields $\sig$ which belong to $(L^1_\mu)^d$ and do not belong to $\T_\mu$. The convergence stated in \ref{item:limIs0} extends as follows: 
\begin{equation}\label{eq:equi}
	\lim_{h \to 0^+} F_h(\mu,\sig) = 2 \int |\sigma - P_\mu \sigma| \, d\mu,
\end{equation}
where $ F_h(\mu,\sig)$, the defect of uniform differentiability over $C^1$ $1$-Lipschitz maps, is given by
$$ F_h(\mu,\sig) \ \coloneqq \sup_{u \in \mathrm{Lip}_1\cap C^1} \int \left( \frac{u(x + h \sigma(x)) - u(x)}{h} - \pairing{\sigma,\nabla u} \right)\, d\mu(x).$$

Since this limit vanishes exactly when $\sig$ has no component in $\N_\mu$, the bundle $\T_\mu$ is precisely the set of velocity fields along which the Kantorovich-Rubinstein norm is stable -- equivalently, along which the difference quotients $\diffquot{\sig}{h_n}$ form a Cauchy sequence in $\AEd$ for every $h_n\to 0^+$. Corollary~\ref{cauchy} quantifies the case when $\sig\notin\T_\mu$: the difference quotients remain asymptotically at distance exactly $\int|\sig_N|\,d\mu$ from $-\div(\sig_T\mu)$. Formula \eqref{eq:equi} moreover quantifies the defect created by the non-tangential part of the velocity.  

In Section~\ref{sec:otherbundles} we compare $T_\mu$ with other bundles. We first show that $T_\mu$ coincides with the decomposability bundle $V_\mu$ of \cite{AM} (assertion \ref{item:VT} of Proposition~\ref{AM=MK}); this proof is simple and measure-theoretic. The differentiability results described above provide a second, independent route to the same identification, based on optimal transport and duality. 
Exploiting the dual formulation of the normal bundle $\N_\mu=(\T_\mu)^\perp$ (see \eqref{orthocrit}), we provide
a simple construction (Proposition~\ref{nodiff}) of a Lipschitz function whose directional derivatives fail to exist
a.e. along any vector field which is not tangent to $T_\mu$, thereby recovering the converse of Rademacher's
Theorem obtained in \cite{AM}. Combined with \eqref{eq:fulldiff}, this identifies $T_\mu$ as the maximal differentiability bundle for Lipschitz maps (Theorem~\ref{maximal}), independently of the results of \cite{AM}.
We finally relate $T_\mu$ to the Preiss tangent measures of $\mu$ (Proposition~\ref{MK-Preiss}): at $\mu$-a.e. point $x$, every tangent measure is invariant under translations along $T_\mu(x)$, so that the MK bundle is always contained in the Preiss tangent plane of $\mu$ when the latter exists.

\section{Preliminaries}

Let $(X,d_X)$ and $(Y,d_Y)$ be metric spaces. The complement of a set $A \subset X$ is denoted $A^c$, and its indicator by $\one_A$, i.e. $\one_A(x) = 1$ if $x \in A$, and 0 otherwise. $C_0(X;Y)$ is the set of continuous functions from $X$ to $Y$. The set of Lipschitz functions from $X$ to $Y$ is denoted $\Lip(X;Y)$, with $\Lip(u)$ being the Lipschitz constant of an element $u \in \Lip(X,Y)$. If $X$ is a Banach space, we define the following subsets of $\Lip(X;Y)$:
\begin{align*}
	\Lip_0(X;Y) := \left\{ u \ : \ u(0) = 0 \right\}, \qquad
	\Lip_1(X;Y) := \left\{ u \ : \ \Lip(u) \leqslant 1 \right\}.
\end{align*}
The support of a function $f$ is denoted $\spt f$. The set of compactly supported functions of class $C^{\infty}$ is $\D(\Rd)$, with its dual $\D'$ being the set of distributions.

Let $\B(\Rd)$ be the Borel $\sigma$-algebra of $\Rd$. The Lebesgue measure is denoted $\Ld$, the Hausdorff measure of dimension $\alpha$ by $\mathcal{H}^{\alpha}$, and the Dirac mass at a point $x \in \Rd$ by $\delta_x$. If $A \subset \Rd$ is measurable, the restriction of a measure $\mu$ to $A$ is given by $\mu \res A$, with $(\mu \res A)(B) \coloneqq \mu (A \cap B)$. Two measures are mutually singular, which we write as $\mu \perp \nu$, if there exists $A \in \B(\Rd)$ such that $\mu(A^c) = \nu(A) = 0$. At the opposite, the notation $\mu \ll \nu$ means that $\mu$ is absolutely continuous with respect to $\nu$, i.e. $\nu(A) = 0$ implies $\mu(A) = 0$ for $A \in \B$.

The set $\M(\Rd)$ collects all Borel signed measures, with $\M_+(\Rd)$ being the subset of nonnegative measures.
We also write $\Mk = \M(\R^d;\R^k)$ to denote the set of $\R^k$-valued measures with $\|\la\|_{\Mk} < +\infty$, where the total variation norm is defined as
\begin{align*}
	\|\la\|_{\Mk} \coloneqq \sup \{ \set{ \pairing{\la, \f}}{\f \in C_0(\Rd;\R^k),\ |\f(x)| \le 1}.
\end{align*}
Any $\la \in \Mk$ writes as $\la = \sigma |\la|$, where $|\la| \in \M_+(\Rd)$ is a nonnegative measure with $\|\la \res A\|_{\Mk} = |\la|(A)$ for any $A \in \B(\Rd)$, and $\sigma \in L^1_{|\la|}(\Rd;\R^k)$ valued in the unit ball of $\R^k$. We refer to $\sigma$ as the polar of $\la$. 

\subsection{Convex analysis}

Let $X$ be a normed space and let $h: X\to \R\cup \{+\infty\}$ be a convex function. Recall that the Moreau-Fenchel conjugate of $h$ is defined on the dual space $X^*$ by 
$$ h^{*} (x^*) := \sup_{x\in X} \big\{ \pairing {x,x^*} - h(x) \big\} \qquad \forall\, x^*\in X^*.$$ 
Clearly, $h^*$ is convex and lower semi-continuous with respect to the weak-$*$ topology on $X^*$. The existence of $x_0^* \in X^*$ such that $h^*(x_0^*)<+\infty $ ($h^*$ is said to be \emph{proper}) means that $h$ admits an affine minorant. In that case, the supremum of all affine continuous minorants is given by the biconjugate of $h$ defined on $X$ by
$$ h^{**} (x) := \sup_{x^*\in X^*} \big\{ \pairing {x,x^*} - h^*(x^*) \big\}\qquad \forall \,x\in X.$$
Besides we define the lower semicontinuous envelope of $h$ at $x\in X$ by $$ \ov h(x) =\inf_{x_n\to x} \liminf_n h(x_n).$$

We will use two classical results. The first one is a consequence of Hahn-Banach theorem and can be found in \cite{EK76}; the second one is due to J.J. Moreau \cite{Moreau1966} in the infinite dimensional case, see also \cite{ency}. 
\begin{proposition}\label{relax}
	Assume that $h^*$ is proper. Then $h^{**} =\ov h$.
\end{proposition}

\begin{proposition}\label{duality_classical}
	Assume that there exists $r>0$ such that $\sup \{ h(x) \, :\, \|x\| \le r \} < \infty$. Then:
	\begin{enumerate}
		\item $h$ is continuous at $0$, while $h^*$ is coercive and attains its minimum on $X^*$;
		\item we have the equalities: $ h(0) = h^{**}(0) = - \min h^{*}.$
	\end{enumerate}
\end{proposition}

\subsection{Measurable multifunctions}\label{prel:multif}

Fix a Radon measure $\mu\in \M_+(\Rd)$. To every family $\F$ of Borel functions from $\Rd$ to a separable Banach space $Y$, we may associate a closed valued and $\mu$-measurable multifunction $\Gamma: \Rd \to 2^Y$ such that:
\begin{enumerate}
	\item\label{appx:G} $u(x) \in \Gamma(x) \ \mu$-a.e. for any $u\in \F$,
	\item Any $\mu$-measurable multifunction $\Gamma'$ satisfying \ref{appx:G} is such that $\Gamma(x)\subset\Gamma'(x) \ \mu$ a.e..
\end{enumerate}
The existence of such $\Gamma$ can be proved in an even more general context (see \cite{Bou-Val}) and it is unique up to the equality $\mu$-a.e.. The usual notation used for $\Gamma$ is $\mu-\mathrm{ess} \; \bigcup \{u, u\in \F \}$. In the case where $\F $ is a subset of $L^1_\mu(\Rd, Y)$ (being $Y$ for instance a finite dimensional normed space), a very simple characterization of $\Gamma=\mu-\mathrm{ess} \; \bigcup \{u, u\in \F \}$ is obtained by considering a countable dense subset $\{u_n, n \in \mathbb{N}\}$ of $\F$ and then clearly $\Gamma(x) = \overline{\{ u_n(x), n \in \mathbb{N}\}}$. Another straightforward observation is that:
\begin{itemize}
	\item if $\overline{\F}$ denotes the closure of $\F$ in $L^1_\mu$, then
	$$ \mu-\mathrm{ess} \; \bigcup \{u ;\ u\in \F \} = \mu-\mathrm{ess} \; \bigcup \{u ;\ u\in \overline{\F} \}.$$
	\item If $\F$ is a vector subspace, then so is $\mu-\mathrm{ess} \; \bigcup \{u, u\in \F \}$.
\end{itemize} 

\medskip
We now particularize to $\F$ being a $\D$-stable vector subspace of $L^1_\mu(\Rd; Y)$ in the following sense:
\begin{equation}\label{Dstable}
	u \in \F, \qquad \f \in \D(\Rd) \quad \implies \ \f \, u \in \F.
\end{equation}

\begin{proposition}\label{local}
	Let $\F$ be a $\D$-stable vector subspace of $L^1_\mu(\Rd;Y)$ and let $F_\mu := \mu-\mathrm{ess} \; \bigcup \{u ;\ u\in \F \}$ be the associate bundle. Then we have the equality
	\begin{equation}\label{HU} 
		\overline{\F} = \{ u\in L^1_\mu(\Rd;Y) \ :\ u(x) \in F_\mu(x) \ \mu-\text{a.e.} \}.
	\end{equation}
\end{proposition}

Note that, the inclusion $\subset$ holds true without assuming \eqref{Dstable}. 

\begin{proof} 
	By approximation, the condition \eqref{Dstable} implies that $u \, \one_A \in \overline{\mathcal{F}}$ whenever $u \in \overline{\mathcal{F}}$ and $A$ is $\mu$-measurable. The wished equality then follows from Hiai-Umegaki \cite[Th.~3.1 p.158]{Hiai}. For the sake of completeness, we give a short proof of the non trivial inclusion $\mathcal{G} \subset \overline{\mathcal{F}}$ being $\mathcal{G}$ the right hand side of \eqref{HU} assuming that $Y$ is separable. Since $\mathcal{G}$ is closed, by passing to the orthogonal in $(L^\infty_\mu)^d$, the wished inclusion is equivalent to $\mathcal{F}^\perp \subset \mathcal{G}^\perp = \set{v\in (L^\infty_\mu)^d}{v(x) \in F_\mu^\perp(x) \ \mu-\text{a.e.}}.$
	Let $v\in \mathcal{F}^\perp$ and $\{u_n, n \in \mathbb{N}\}$ be a countable dense subset of $\mathcal{F}$ so that $F_\mu(x) = \overline{\{ u_n(x), n \in \mathbb{N}\}}$. Then, for every $n$ and any Borel subset $A$, we have that $u_n \one_A$ belongs to $\mathcal{F}$, hence $ \int_A \pairing{v, u_n}\, d\mu=0$. By localizing, we infer that $\pairing{v, u_n}= 0$ holds $\mu$ a.e, hence $v(x) \in F_\mu^\perp(x)$ as wished.
\end{proof}

\subsection{Monge-Kantorovich distance}

The Monge-Kantorovich distance between two probability measures $\mu, \nu$ with finite first moment, i.e. $\int |x| d\mu < \infty$, is given by 
\begin{equation}\label{def:W1}
	W_1(\mu,\nu) \coloneqq \inf_{\gamma} \int_{\Rd\times\Rd}|x-y|\,d\gamma(x,y),
\end{equation}
where the infimum is taken among all admissible transport plans $\gamma$ such that $\pi^\sharp_1 \gamma = \mu$ and $\pi^\sharp_2 \gamma = \nu$. Here $\pi^\sharp_1$ and $\pi^\sharp_2$ are the usual push-forward operators associated to the projections $\pi_1$ and $\pi_2$ from $\Rd\times\Rd$ on the first and second factors respectively. By homogeneity the definition above extends naturally to finite measures $\mu, \nu \in \M_+(\Rd)$ such that $\int \mu=\int \nu$. The core of this theory is the Kantorovich-Rubinstein dual representation
\begin{align*}
	W_1(\mu,\nu) = \sup \left\{\int u \, d\nu - \int u \, d\mu\ :\ u \in \Lip_1(\Rd) \right\},
\end{align*}
where $\Lip_1(\Rd)$ stands for the set of Lipschitz functions whose Lipschitz constant $\Lip(u)$ is smaller than $1$. It is easy to check that the supremum above is achieved provided $\mu$ and $\nu$ have a finite first order moment and share the same mass. On the other hand, it is clear from the dual formulation that $W_1(\mu,\nu)$ only depends on the difference $f = \nu-\mu$.
Accordingly we define the following vector subspace of signed measures on $\Rd$:
\begin{equation}\label{def:Mod}
	\Mod(\Rd) :=\left\{ f\in \M(\Rd) \ :\ \int f =0 \quad,\quad \int |x| \, d|f|(x) < +\infty \right\} 
\end{equation}
endowed with the Kantorovich-Rubinstein (KR) norm:
\begin{align*}
	\| f\|_{\KR} \ = W_1(f_+,f_-) := \sup \left\{\int u\, df\ :\ u \in \Lip_1(\Rd) \right\}.
\end{align*}
In particular, $\delta_x - \delta_0 \in \Mod$ for any $x \in \Rd$. In turn $ \Mod(\Rd)$ is not complete as a normed space. Indeed there exists Cauchy sequences $(f_n)$ which are not bounded in total variation, as for instance:
\begin{align*}
	f_n = \sum_{j=1}^n (\delta_{b_j} -\delta_{a_j}) 
	\quad \text{where} \quad
	\sum_j |b_j-a_j| < +\infty,
\end{align*}
which converges in the sense of $\D'$ to a first order distribution $f$ given by the sum of the series.

\subsection{Completion of $\Mod$}

Owing to its relation with the Arens-Eells (free Lipschitz fields) theory, in which more general metric spaces than $\Rd$ are considered, we will denote by $\AEd$ the completion of $\Mod$ with respect to the KR norm. In our context, $\AEd$ can be identified as a subset of distributions $f \in \D_1'(\Rd)$ of order $\le 1$ (see for instance \cite{BCJ2005}). 

\begin{definition} 
	A distribution $f\in \D'(\Rd)$ belongs to $\AEd$ if and only if, for every $(\f_n)$ in $\D(\Rd)$, the following implication holds:
	\begin{equation}\label{criterion}
		\f_n \to c \in \R, \quad \sup |\nabla \f_n| \le C \ \implies \ \pairing{f,\f_n} \to 0 \ . 
	\end{equation}
\end{definition} 

\begin{proposition}\label{completion} 
	\begin{enumerate}[label={(\roman*)}]
		\item Every $f\in \AEd$ can be extended in a unique way into a linear form on $\Lip(\R^d)$ (still denoted $f$) such that
		$$ \pairing{f, u} \le C \, \Lip(u) \quad,\quad \pairing{ f, 1}=0 \ .$$
		The best constant $C$ is given by $C = \| f\|_{\KR}$ where we have set
		$$ \| f\|_{\KR} := \sup \set{\pairing{ f, u}}{u \in \Lip_1(\R^d)}.$$
		\item For every $f\in \AEd$, there exists a potential $u\in \Lip_1(\R^d)$ such that $$ \pairing{ f, u} = \| f\|_{\KR}.$$
		\item Equipped with the KR-norm, $\AEd$ is a Banach space which contains $\Mod$ as a dense subspace.
	\end{enumerate}
\end{proposition}

The proof of Proposition~\ref{completion} can be found in \cite{BCJ2005, BBDP2010} where the analysis is restricted to measures compactly supported in a convex compact subset of $\Rd$, further generalized in metric spaces in \cite{ARGB}. Actually, in light of the Arens-Eells theory (see \cite{AE, Hanin1997, Olu2004, WeaverBook}), it is well known that the dual of $\AEd$ can be identified as
\begin{align*}
	\AEd^* \sim \Lip_0(\Rd):=\{ u \in \Lip(\Rd) : u(0)=0 \}
\end{align*}
by considering the canonical map $\Psi : u\in \Lip_0(\Rd) \mapsto \pairing{\cdot,u} \in \AEd^*$. This map is an isometry, i.e. $\Vert \Psi\Vert_{\AEd^*} =\Lip(u)$, and its inverse is given by $\Psi^{-1}: f\in \AEd \mapsto u$ where $u(x) = \pairing{f, \delta_x-\delta_0}$.
As a consequence, the weak-star convergence of an equi-Lipschitz sequence in $\Lip(\Rd)$ is equivalent to its pointwise convergence.
\begin{definition}\label{weakstarLip} A sequence $(u_n)$ in $\Lip(\Rd)$ is said to be weakly star convergent to $u$ if 
	$$ \sup_n \Lip(u_n) <+\infty \quad \text{and}\quad u_n(x) \to u(x) \quad \text{for every $x\in \Rd$}. $$
	In this case, we shall write shortly $u_n \weakstar u$.
\end{definition}

Note that the convergence above implies the uniform convergence of $(u_n)$ on compact subsets of $\Rd$. Moreover,
in virtue to Banach-Steinhaus theorem, it holds:
\begin{equation}\label{weak-star.crit}
	u_n\weakstar u \iff \pairing{f,u_n} \to \pairing{f,u} \quad \text{for every $f\in \AEd$}.
\end{equation}

Therefore every equi-Lipschitz sequence $(u_n)$ in $\Lip_0(\Rd)$ admits a weak-star converging subsequence.
For the reader's convenience, we give a short proof of the dense inclusion $\Mod \subset \AEd$
stated in assertion (iii) of Proposition~\ref{completion}. 

\begin{proof} 
	Let $f\in \Mod$ and let us check that it satisfies \eqref{criterion}. As $f$ has integral 0, we can assume that $c=u_n(0)=0$.
	Therefore $|u_n|(x)\le |x|$ and, since $u_n\to 0$ uniformly on compact subsets, we get that for every $R > 0$,
	\begin{align*}
		\limsup_{n\to \infty} |\pairing{f,u_n}| 
		\le \limsup_{n\to \infty} \int |u_n|\, d|f|(x) 
		\le \int_{\{|x|\ge R\}} |x|\, d|f|(x).
	\end{align*}
	The conclusion follows by letting $R\to \infty$ in virtue of the integrability condition in \eqref{def:Mod}.
	The density is deduced from $\AEd^* \sim\Lip_0(\Rd)$: indeed, a linear form in $\AEd^*$ vanishing on $\Mod$
	is represented by a Lipschitz function $u$ with $u(0)=0$, which satisfies in particular $\pairing{u,\delta_x-\delta_0} = 0$ for any $x \in \Rd$. 
	It follows that $u \equiv 0$, and $\Mod$ is dense in $\AEd$.
\end{proof}

\begin{remark}\label{Lipdual}
	In view of Definition~\ref{weakstarLip}, by using convolution by smooth kernels and truncations arguments, it is easy to check that $\D(\Rd)$ is weakly-star dense
	in $\Lip(\Rd)$. Accordingly, we have:
	\begin{equation}\label{smoothing}
		\|f\|_{\KR} = \sup \set{\pairing{f,\f}}{\f\in \D\cap \Lip_1(\Rd)} \quad \text{for every } f\in{\AEd}. 
	\end{equation}
	It is tempting to use this equality to extend the definition of the KR norm to all elements of the dual of $\Lip(\Rd)$.
	However, as a linear form $T \in (\Lip(\Rd))'$ is not uniquely determined by its restriction to smooth functions 
	(the closure of $\D(\Rd)$ is a subspace of $C^1$), to be consistent, we would have to chose specific Hahn-Banach extensions $T$ of elements $f\in \AEd$ (which are not unique).
\end{remark}

\subsection{Transport measures versus currents}

A transport measure associated with a source term $f$ is a vector measure $\la \in \Md$ such that $- \div \la = f$ in the sense of distributions. We are interested in two cases: $f\in \Mod$ or $f\in \AEd$. By identifying such a $\la$ to a one dimensional currents with finite mass $\|\la\|$, we fall into the terminology of the Federer-Flemming theory of currents \cite{Fed,Morgan} from which we borrow the notation.

\begin{definition} 
	We will denote
	\begin{align}
		\Nd &:= \set {\la \in \Md} {-\div \la \in \Mod} &&\text{(normal 1-currents),} \label{def:Nd} \\
		\Fd &:= \set {\la \in \Md} {-\div \la \in \AEd} && \text{(flat chains of finite mass).} \label{def:Fd}
	\end{align} 
\end{definition}

In the language of currents, $-\div \la$ represents the boundary $\partial T$ of the $1$-current $T$ associated to $\la$ by $\left<T,\f\right> = \int \f \, d\lambda$. A key result is the following: 
\begin{proposition}\label{fundamental}
	\begin{enumerate}
		\item\label{fundamental:surj} The linear map $ \la \in \Fd\subset \Md \mapsto -\div \la \in \AEd$ is surjective and satisfies:
		\begin{equation}\label{divlip} 
			\| -\div \la \|_{\KR} \ \le \ \|\la\|_{\Md} \qquad \text{for every } \la\in \Fd.
		\end{equation}
		Moreover, for every $f\in \AEd$, we have the equality
		\begin{equation}\label{Beck}
			\|f\|_{\KR} \ =\ \min \set{\|\la\|_{\Md}}{ -\div\la = f,\ \la\in \Fd}.
		\end{equation}
		\item\label{fundamental:clos} $\Fd$ is the closure of $\Nd$ in the Banach space $\Md$;
		\item\label{fundamental:locp} $\Fd$ has the following local property:
		\begin{equation}\label{Fdlocal}
			\la \in \Fd \ \implies\ \theta \, \la \in \Fd \quad \text{for any } \theta \in L^1_{|\la|}(\Rd). 
		\end{equation}
	\end{enumerate}
\end{proposition}

The first assertion already appeared in \cite{BCJ2005}. The density property stated in assertion \ref{fundamental:clos} plays a central role in the proof of the Flat Chain Conjecture in general 1 metric currents (see for instance \cite{ARGB}).

\begin{proof} 
	\ref{fundamental:surj}\quad By \eqref{smoothing} applied to $f = -\div \la$, recalling that $\pairing{f,\f}=\pairing{\la,\nabla\f}$, 
	\begin{align*}
		\|-\div \la\|_{\KR} 
		= \sup \set{\pairing{\la,\nabla\f}}{\f\in \D(\Rd), |\nabla \f|\le 1}
		\le \|\la\|_{\Md}.
	\end{align*}
	If $f\in \Mod$, the equality \eqref{Beck} is the classical Beckmann dual representation of the Monge-Kantorovich distance $W_1(f_+,f_-)$. The existence of an admissible competitor $\la$ and the fact that the infimum is attained is ensured once $W_1(f_+,f_-)=\|f\|_{\KR}<+\infty $. This is always the case for $f\in \Mod$ since, by the triangular inequality, $W_1(f_+,f_-)\le \int |x| |f|(dx)$. To extend \eqref{Beck} to general source terms $f\in \AEd$, consider an approximating sequence $(f_n)$ in $\Mod$ such that $f_n\to f$ in $\AEd$. Then there exists $\la_n$ a minimizer of \eqref{Beck} for $f_n$, such that $\int |\la_n| = \|f_n\|_{\KR}$ and $- \div \la_n = f_n$. Then the sequence $(\la_n)$ is bounded in $\Md$ and admits at least one weak-star cluster point $\la$, which satisfies:
	$$ \int |\la| \le \| f\|_{\KR} \quad,\quad -\div \la = f.$$
	Here we used the lower semicontinuity of the total variation in $\Md$ with respect to the weak-star convergence (which implies the convergence in the sense of distributions) and also the fact that $\|f_n- f\|_{\KR} \to 0$ implies that $f_n\to f$ in $\D'(\Rd)$. Then, by applying \eqref{divlip}, we deduce that $\int |\la| =\| f\|_{\KR}$, whence the optimality of $\la$ in \eqref{Beck}.
	
	\med
	\ref{fundamental:clos}\quad We prove that $\Fd$ is closed in $\Md$. Let $(\la_n)$ be a sequence in $\Fd$ and $\la\in \Md$ such that $\Vert \la_n-\la\Vert_{\Md} \to 0$. Then, from \eqref{divlip}, we deduce that $f_n: =-\div\la_n$ is a Cauchy sequence in $\AEd$; hence there exists $f\in \AEd$ such that $\|f_n- f\|_{\KR} \to 0$ and passing to the limit in $\D'(\Rd)$, we get $f=-\div \la$ thus $\la\in \Fd$.
	
	\medskip
	Next we consider an element $\la\in \Fd$ and prove the existence of a sequence $(\la_n)$ in $\Nd$ such that $\|\la_n-\la\|_{\Md}\to 0$. Denoting $f:=-\div\la$ (then $f\in\AEd$), we observe that it is enough to find a sequence $(\ov\la_n)$ in $\Nd$ such that $\|\ov\la_n-\ov\la\|_{\Md}$ where $\ov \la$ is any element $\ov\la \in\Md$ such that $-\div \ov\la=f$. Indeed, if it is the case, then $\la-\ov\la$ is divergence free and therefore belongs to $\Nd$. Then the sequence $\la_n =\ov{\la_n} + (\la-\ov \la)$ will satisfy the desired requirement.
	
	\med
	\emph{Construction of $\ov\la_n\to \ov \la$:}\ Since $f\in \AEd$, there exists a sequence $(f_n)$ in $\Mod$ such that $\|f_n- f\|_{\KR} \to 0$. Possibly after extracting a subsequence, we can assume that $\eps_n:= \|f_{n+1}-f_n\|_{\KR}$ satisfies $\sum_1^\infty \eps_n<+\infty$. Setting $f_0=0$ and $g_k:=f_{k+1} - f_k$ we have $f_n= \sum_{k=0}^n g_k$ and we can express $f$ as the sum of the absolutely convergent series:
	$$ f = \sum_{k=0}^\infty g_k \quad,\quad g_k \in \Mod \quad,\quad \| g_k\|_{\KR} = \eps_k.$$
	Then letting $\chi_k\in \Md$ be optimal in \eqref{Beck} for $g_k$, the equality $\|\chi_k\|= \eps_k$ allows to construct a series converging absolutely in total variation. We set:
	$$ \ov\la_n:= \sum_{k=0}^{n-1} \chi_k \quad, \quad \ov\la\ :=\ \sum_{k=0}^\infty \chi_k.$$ 
	Since the convergences $\ov\la_n \to \ov\la$ in $\Md$ and $f_n\to f$ in $\AEd$ imply the convergence in the sense of distributions, by passing to the limit in the equation $-\div \ov\la_n = f_n$, we deduce that $\ov \la$ satisfies $-\div \ov \la = f$ in $\D'(\Rd)$. 
	
	\med
	\ref{fundamental:locp}\quad Let $\la \in \Fd$ and consider a smooth scalar function $\theta\in \D(\Rd)$. Then, since for every $u\in\D(\Rd)$, we have 
	$ \pairing{ - \div(\theta\, \la), u} = \pairing{\la, \nabla (\theta u)- u \nabla \theta}, $
	the criterion \eqref{criterion} under which $\div(\theta \la)$ belongs to $\AEd$ is clearly satisfied. Accordingly we have $\theta \, \la \in \Fd$. Given now $\theta \in L^1_{|\la|}(\Rd)$, we can take a sequence of smooth functions $(\theta_n)$ such that $\theta_n\to \theta$ in $L^1_{|\la|}(\Rd)$. Then $\lambda_n=\theta_n \, \la$ converges to $\theta\, \la$ in $\Md$. Hence we can conclude that $\theta\la \in \Fd$ since, by the assertion \ref{fundamental:clos}, we know that $\Fd$ is a closed subspace of $\Md$.
\end{proof}

\begin{remark}\label{localFd} 
	The locality property of assertion \ref{fundamental:locp} turns out to be crucial, although it a straightforward consequence of \ref{fundamental:clos}. We stress that the implication \eqref{Fdlocal} is false if we replace $\Fd$ by $\Nd$. Indeed, $\div(\theta \la)$ may not have finite total variation if the scalar function $\theta$ is discontinuous (take for instance $\theta = \one_B$ where $B$ has infinite perimeter).
\end{remark}

\begin{remark}\label{tangential} 
	In \cite{BCJ2005}, $\Fd$ was introduced as the set of \emph{tangential measures}. The reason behind this name is that in smooth cases, such measures $\la$ are tangential to the geometric support of $|\la|$. The forthcoming construction of the MK tangent bundle will result in a precise formulation of this property.
\end{remark}

\section{The MK tangent bundle to a measure}

\subsection{Construction and locality}\label{sec:locality}

Let $\mu \in \M_+(\Rd)$ be a Radon measure. We consider the following subspace of $(L^1_{\mu}(\Rd))^d$:
\begin{align}\label{def:Tmu}
	\T_\mu \coloneqq \left\{ \sigma \in (L^1_\mu(\Rd))^d \st - \div(\sigma \mu) \in \AEd \right\}
\end{align}
By \eqref{def:Fd}, $\T_\mu$ coincides with the subspace $ \set{ \sigma \in (L^1_\mu(\Rd))^d }{\sigma \, \mu \in \Fd}$. Therefore
it inherits the properties of $\Fd$ given in the assertions \ref{fundamental:clos} and \ref{fundamental:locp} of Proposition~\ref{fundamental}, namely:
\begin{itemize}
	\item $\T_\mu$ is a closed subspace of $(L^1_\mu(\Rd))^d$;
	\item If $\sigma\in \T_\mu$, we have $\sigma\, \theta \in \T_\mu$ for every $\theta\in L^\infty_\mu$.
\end{itemize}
It follows that $\T_\mu$ is determined by a local multivalued application $T_{\mu}$. 

\begin{proposition}\label{res:firstprop}
	There exists a $\mu$-measurable multivalued application $T_\mu : \Rd \rightrightarrows \Rd$ such that $T_\mu(x)$ is a vector subspace of $\Rd$, possibly reduced to $\{0\}$, and 
	\begin{align}\label{Tmu=}
		\sigma \in \T_\mu 
		\quad \iff \quad
		\sigma(x) \in T_\mu(x) \quad \mu\text{-a.e.} \qquad (\iff \sigma\, \mu \in \Fd).
	\end{align}
\end{proposition}

\begin{proof}
	As in Proposition~\ref{local}, the existence of $T_\mu(\cdot)$ follows from \cite[Th.~3.1 p.158]{Hiai}.
\end{proof}

\begin{definition}[MK-tangent bundle]
	The multivalued application $T_{\mu}$ characterized in \eqref{def:Tmu} is the MK-tangent bundle to the measure $\mu$. Its orthogonal bundle is denoted
	\begin{align*}
		N_{\mu}(x) \coloneqq (T_{\mu}(x))^{\perp}.
	\end{align*}
\end{definition}

A main interest of the construction of $\T_\mu$ lies in the following property.

\begin{lemma}[Locality]
	Let $\mu, \nu$ be two Radon measures and $B\in \B(\Rd)$. Then
	\begin{align}\label{locality}
		\mu\res B = \nu \res B\ \implies \ T_\mu(x)= T_\nu(x) \quad \text{a.e. on } B. 
	\end{align}	
\end{lemma}

\begin{proof}
	Denoting $\mu_B = \mu\res B$, the implication above is clear if we can show that $T_\mu(x) = T_{\mu_B}(x)$ holds $\mu$ a.e. in $B$. Let $\sigma \in \T_{\mu_B}$, then $\sigma \mu_B= \widetilde{\sigma} \mu$ belongs to $\Fd$, where $\widetilde{\sigma}$ is the extension of $\sigma$ by 0 outside of $B$. Therefore $\widetilde{\sigma}(x) \in \T_{\mu}$. Accordingly we have $T_{\mu_B}(x) \subset T_{\mu}(x)$ for $\mu$-a.e. $x \in B$. In the opposite direction, if $\sig\in \T_\mu$, then $\sig \, \one_B \mu= \sig\, \mu_B$ belongs to $\Fd$ by the locality property \eqref{local}. It follows that $\sig \in T_{\mu_B}$, whence $\sig(x) \in T_{\mu_B}(x)$ for $\mu$-a.e. $x \in B$.
\end{proof}

\medskip 
As a corollary of the property \eqref{locality}, we get the following.
\begin{lemma}\label{varieloc}
	The following implications hold:
	\begin{align}\label{disjunction}
		& \mu = \mu_1 + \mu_2 \ \text{ and}\ \mu_1 \perp \mu_2 \implies
		T_\mu(x)= 
		\begin{cases} 
			T_{\mu_1} (x) & \mu_1 \text{ a.e.},\\ 
			T_{\mu_2} (x) & \mu_2 \text{ a.e.}.
		\end{cases} \\
		& \nu \ll \mu \implies T_\nu(x) = T_\mu(x) \quad \nu\ \text{a.e.} \label{restriction}
	\end{align} 
\end{lemma}

\begin{proof} 
	\eqref{disjunction} follows from \eqref{locality} by considering a Borel subset $B$ such that $\mu\res B=\mu_1$ and $\mu\res B^c= \mu_2$. To show \eqref{restriction}, we write $\nu = \theta \mu$ where $\theta$ is the Radon-Nikodym derivative of $\nu$ with respect to $\mu$. Then $\sigma \in \T_\nu$ implies that $-\div(\sigma \theta \mu) \in \AEd$, hence $\sigma(x)\in T_\mu(x)$ holds $\mu$ a.e. on $\{\theta > 0\}$. This means that $T_\nu(x) \subset T_\mu(x)$ holds $\nu$ almost everywhere. In the opposite direction, assume that $\sig$ belongs to $\T_\mu$. Then, since $\sig\, \mu \in \Fd$, we also have that $\sig\, \nu= \theta \sig \mu \in \Fd$ by the locality property \eqref{Fdlocal}. Therefore $\sig\in \T_\nu$ and $\sig(x)\in T_\nu(x)$ holds $\nu$ a.e..
\end{proof}

\subsection{Projection theorem and orthogonal bundle}

The next result extends estimates obtained in \cite{BBDP2010} allowing us to show that $\Fd$ is a complemented space of $\Md$. To the best of our knowledge, this seems to be a new result. We refer to~\cite[Section 3]{MM} for a similar idea of decomposition to study purely non-flat currents. 
\begin{theorem}\label{projthm}
	Let $\la\in \Md$. Then there exists a unique solution $\la_T \in \Fd$ to the minimal distance problem
	\begin{align*}
		\|\la-\la_T\|_{\Md} \ =\ \inf \set{\|\eta-\la\|_{\Md}}{ \eta\in \Fd} \ =: \dist(\la, \Fd).
	\end{align*}
	Moreover the projection map ${\mathcal P}_{\mathbf F}: \la\in \Md \mapsto \la_T\in \Fd$ is linear and 1-Lipschitz. Accordingly, we have the topological direct sum
	\begin{equation*}\label{directsum}
		\Md \ =\ \Fd \oplus \Fdp \ , 
	\end{equation*}
	where $\Fdp := \ker {\mathcal P}_{\mathbf F} = \set{\la\in\Md}{\frac{d\lambda}{d|\lambda|} \in N_{|\lambda|} \ \ |\lambda|-a.e.}$.
\end{theorem}

\begin{proof} 
	Let $\lambda = \sigma \mu$ be the polar decomposition of $\lambda$, with $|\sigma| = 1$ $\mu$ a.e.. Denoting $P_{\mu}$ the orthogonal projection on $T_{\mu}$ and $P_\mu^\perp$ the orthogonal projection on $T_\mu^\perp$, we can decompose $\sigma$ as $\sig=\sigma_T +\sigma_N$, where
	\begin{align*}
		\sigma_T(x):= P_\mu(x)(\sigma(x)), \qquad \sigma_N (x):= P_\mu^\perp(x) (\sigma(x)).
	\end{align*}
	By \eqref{Tmu=}, $\la_T:= \sigma_T\, \mu$ belongs to $\Fd$ and therefore:
	$$d(\la, \Fd)\ \le\ \|\la-\la_T\|_{\Md}\ = \ \int |\sigma_N|\, d\mu.$$
	On the other hand, any $\eta\in \Fd$ admits a Lebesgue-Nikodym decomposition $\eta = w\, \mu + \eta_s$ where $|\eta_s| \perp \mu$. Using \eqref{disjunction}, we infer that $w(x) \in T_\mu(x)$ at $\mu$ a.e. $x \in \Rd$. It follows that
	\begin{align*}
		\|\la- \eta\|_{\Md} \ =\ \int \sqrt{ |\sigma_T - w|^2+ |\sigma_N|^2}\, d\mu + \int |\eta_s|\ \ge\ \int |\sigma_N|\, d\mu.
	\end{align*}
	Hence the minimal distance of $\la$ to $\Fd$ is reached at $\lambda_T = \eta$ if and only if $w=\sigma_T$ and $\eta^s=0$, that is $\lambda_T = P_{\mu}(\sigma) \mu$. 
	
	\medskip

	Next, for every $\eta = \xi \, m$ with $m\in \M_+(\Rd)$ and $\xi \in (L^1_m)^d$, we set 
	\begin{equation}\label{def:PF} 
		{\mathcal P}_{\mathbf F} (\eta) := P_m(\xi)\, m. 
	\end{equation}
	Thanks to the homogeneity and local property (see \eqref{restriction}) of the MK tangent bundle, it is easy to check that the latter definition does not depend on the decomposition used for $\eta$. In particular, using the polar decomposition of a given $\la\in \Md$ as before, we obtain the relation $\la_T= {\mathcal P}_{\mathbf F}(\la)$. Now, let us demonstrate that ${\mathcal P}_{\mathbf F}$ is a linear operator. The homogeneity follows directly from \eqref{def:PF}. To check the additivity, we consider arbitrary $\la_1,\la_2\in\Md$ and set $m := |\la_1|+|\la_2|$. Then $\la_1,\la_2$ and $\la_1+\la_2$ are absolutely continuous with respect to $m$. Using \eqref{def:PF} and exploiting the linearity of the local projection $P_m$, we find that
	\begin{align*}{\mathcal P}_{\mathbf F}(\la_1+\la_2) & = P_m \left(\frac{d \la_1}{dm} +\frac{d \la_2}{dm} \right) \, m \\
	& = P_m \left(\frac{d \la_1}{dm}\right) \, m + P_m \left(\frac{d \la_2}{dm}\right) m = {\mathcal P}_{\mathbf F}(\la_1) + {\mathcal P}_{\mathbf F}(\la_2).
	\end{align*} Finally, the inequality $ \| {\mathcal P}_{\mathbf F}(\la)\|_{\Md} \le \|\la\|_{\Md}$ entails that ${\mathcal P}_{\mathbf F}$ is 1-Lipschitz.
\end{proof}

\begin{example}
	Let $\la =\sigma \, \mu$ in $\R^2$, where $\mu = \H^1\res A$ with $A := [0,1]\times \{0\}$, and $\sigma(x,0) = (h(x),k(x))$. 
	Assuming that $h,k$ are smooth functions on $[0,1]$, it is easy to check that $f=-\div \la$ is given by
	\begin{align*}
		f = h(1) \delta_{(1,0)}- h(0) \delta_{(0,0)} - h'(x) \H^1\res A - \div \left((0,k(x)) \H^1\res A\right).
	\end{align*}
	Accordingly $\la\in \mathbf{N_1}(\R^2)$ if and only if $k\equiv 0$. Indeed, by testing with $\f_n(x,y)= k(x) \theta (ny)$ 
	where $\theta\in \D(\R, [-1,1]) $ is such that $\theta(0)=0$ and $\theta'(0)=1$, we infer that the total variation of $f$ is minorized by $n \int_0^1 k^2(x) \ dx$, and therefore blows up to infinity as $n\to \infty$ unless $k$ vanishes. Therefore, the projection of $\la$ onto $\mathbf{F_1}(\R^2)$ is given by $\la_T=(h(x),0) \mu$, achieving the distance $\int |\la_N| = \int_0^1 |k(x)|\, dx$. 
	Eventually the reader can check easily that, for general $(h,k)\in (L^1_\mu)^2$, we have $ \la \in \mathbf{F_1}(\R^2) $ if and only if $k=0$ holds $\mu$-a.e., while requiring that $ \la \in \mathbf{N_1}$ implies the additional conditions that $h$ and $k$ belong to $\mathrm{BV}(0,1)$. 
\end{example}

A consequence of Theorem \ref{projthm} is that a vector field $\sigma\in (L^1_\mu)^d$ satisfies $\sigma(x)\in T_\mu^\perp(x) \ \text{ $\mu$ a.e}$ if and only if $\sigma\, \mu \in \Fdp$, or equivalently if $\dist(\sigma\, \mu, \Fd)= \int |\sigma|\, d\mu$. Actually it is possible to give a direct dual characterization of such elements $\sigma$ that are in $(L^\infty_\mu)^d$, that is of the subspace
\begin{equation}\label{def:Nmu}
	\N_\mu := \T_\mu^\perp = \set{\sigma\in (L^\infty_\mu)^d}{\sigma(x)\in T_\mu^\perp(x) \ \mu \text{ a.e.}},
\end{equation}
where the second identity is a direct consequence of Proposition~\ref{res:firstprop}.

As we will show in Proposition~\ref{Nmu} below, one has
\begin{equation} \label{orthocrit}
	\sigma\in \N_\mu\quad \iff\quad \exists u_n\in \D(\Rd)\ \ :\ u_n\weakstar 0 \;\;\text{in $\Lip(\R^d)$} \ , \ \nabla u_n \to \sigma \;\; \mu \text{ a.e..}
\end{equation} 
This characterization follows from the computation of the Moreau-Fenchel conjugate of the convex functional $G_{\sigma} : \Lip(\mathbb{R}^d) \to \mathbb{R} \cup \{+\infty\}$ defined by
\begin{equation}\label{def:Gsigma}
	G_\sigma(u) := 
	\begin{cases}
		\int \pairing{\sigma,\nabla u} \, d \mu & \mbox{if} \ u \in (\D \cap \Lip_1)(\Rd), \\
		+\infty & \mbox{otherwise}.
	\end{cases}
\end{equation}

Note that $G_{\sigma}$ is invariant by addition of a constant, so that it can be identified with a functional on $\Lip_0(\mathbb{R}^d)$.

\begin{lemma}\label{dual-sigma} 
	Let $\sigma\in (L^1_\mu)^d$. Then, for every $g\in \AEd$, the conjugate of $G_\sigma$ in the duality between $\Lip_0(\Rd)$ and $\AEd$ is given by
	\begin{equation}\label{Gsigma*}
		G_\sigma^* (g) = \min_{\la\in \M^d} \left\{\int_{\Rd} |\la-\sigma\, \mu | \ : \ -\div \la = g \right\}.
	\end{equation}
\end{lemma}

\begin{proof} 
	We consider the perturbation functional $h : (C_0(\Rd))^d \to \R$ defined by:
	$$ h(p) := \inf \set{ \int \pairing{\sigma, \nabla u + p} \, d\mu - \pairing{g,u}}{ u \in \D(\Rd),\ |\nabla u + p| \le 1}.$$
	One checks that
	\begin{enumerate}
		\item $h(0)= - G_\sigma^*(g)$;
		\item $h$ is convex as it can be written as $h(p)= \inf \set{ H(u,p)}{u\in \Lip(\Rd)}$ where $H$ is convex in both variables on $\Lip(\Rd) \times (C_0(\Rd))^d$ (see \cite[Lemma~2.1 p.50]{EK76}); 
		\item $h(p) \le \int \pairing{\sigma, p}\, d\mu$ if $|p|\le 1$ (simply taking $u=0$ as a competitor).
	\end{enumerate}
	Therefore $h$ is continuous at $p=0$ and, by Proposition~\ref{duality_classical}, its Fenchel conjugate $h^*: \Md \to \R \cup \{+\infty\}$ is coercive, achieves its minimum and satisfies the equalities:
	\begin{equation}\label{h-duality}
		h(0)= h^{**}(0) = - \min \set{h^* (\la)}{\la \in \Md}.
	\end{equation}
	We deduce that
	\begin{align*} 
		h^*(\la) 
		&= \sup \set{\pairing{\la, p} - h(p)}{p\in (C_0(\Rd))^d} \\
		&= \sup \set{ \pairing{\la, p} +\pairing{g, u} - \int \pairing{\sigma, \nabla u + p} \, d\mu}
		{(u,p)\in \D \times (C_0)^d,\ |\nabla u+p| \le 1 }\\
		&= \sup \set{ \pairing{\la -\sigma\, \mu, q} - \pairing{\la, \nabla u} +\pairing{g, u}}{(u,q) \in \D \times (C_0)^d, \ |q| \le 1 } \\
		&= \sup \set { \pairing{\la -\sigma\, \mu, q}}{q \in (C_0)^d, |q|\le 1} + \sup \set{-\pairing{\la, \nabla u} +\pairing{g, u}}{u\in \D}\\
		&= 
		\begin{cases} 
			\int |\la-\sigma\mu | & \text{if } -\div \la = g,\\
			+\infty & \text{otherwise}.
		\end{cases}
	\end{align*}
	Note that, by Proposition \ref{fundamental}, the domain of $h^*$ is not empty since the map $\la\in \Fd \to -\div\la \in \AEd$ is surjective (in particular, we have $h^*(\la) \le \|\la\|_{\Md} + \|\sigma\|_{L^1_\mu}$ for every $\la \in \Fd$).
	The equality \eqref{Gsigma*} and the existence of a minimal $\la$ then follow from \eqref{h-duality}.
\end{proof}

\begin{proposition}\label{Nmu} 
	Let $\sigma\in (L^1_\mu)^d$ and denote $(\sigma_T, \sigma_N):= (P_\mu \sigma, P_\mu^\perp \sigma)$. Then there exists a sequence $(u_n) \subset \D\cap \Lip_1 (\Rd)$ such that:
	\begin{equation}\label{normal-approx}
		u_n \weakstar 0 \; \; \text{in $\Lip(\R^d)$} \ ,\qquad \nabla u_n\to \frac{\sigma_N}{|\sigma_N|} \, \one_{\{\sigma_N \neq 0\}} \quad \mu \text{ a.e. }
	\end{equation}
	As a consequence, the equivalence \eqref{orthocrit} holds and also the equality: 
	\begin{equation}\label{Nmu=}
		\N_\mu = \set{\sigma\in (L^\infty_\mu)^d}{ \exists u_n \in \D(\Rd) : u_n \weakstar 0 \ \text{in } \Lip(\Rd),
			\ \nabla u_n\weakstar \sigma\ \text{in } (L^\infty_\mu)^d}
	\end{equation}
\end{proposition}

\begin{remark}\label{normaltrick} 
	In \cite{BCJ2005}, the construction of the tangent bundle $T_\mu$ was obtained by orthogonality, starting from the condition \eqref{orthocrit} to define the normal space $\N_\mu$ and proving that $\N_\mu^\perp$ satisfies the conditions of Proposition~\ref{local}. 
\end{remark}

\begin{proof}
	By Proposition \ref{relax}, since $G_\sigma$ is convex and proper on $\Lip(\Rd)$, its lower semicontinuous envelope with respect to the weak-star topology given by 
	\begin{equation}\label{def:relax}
		\ov G_\sigma (v):= \inf_{v_n\weakstar v} \liminf_n G_\sigma (v_n)
	\end{equation}
	coincides with the biconjugate $G_\sigma^{**}$, namely:
	$$ \ov G_\sigma(v)= G_\sigma^{**}(v) := \sup \set{\pairing{g,v} - G_\sigma^*(g)}{g\in \AEd}.$$
	In particular, by taking $v=0$, we obtain 
	$$ - \ov G_\sigma(0) = \inf \set{G_\sigma^*(g)}{g\in \AEd}= d(\sigma \mu, \Fd),$$
	where the last equality follows from \eqref{Gsigma*}. In view of Theorem \ref{projthm}, we know that 
	$d(\sigma \mu, \Fd)= \int |\sigma_N| \, d\mu$.
	Therefore, by applying \eqref{def:relax} to $v=0$, we obtain a sequence $(v_n)$ in $\D\cap \Lip_1$ such that $v_n\weakstar 0 $ and 
	\begin{align*}
		\int \pairing {\sigma, \nabla v_n} \, d\mu
		= G_{\sigma}(v_n)
		\underset{n \to \infty}{\longrightarrow}
		\ov G_{\sigma}(0)
		= - \int |\sigma_N| \, d\mu.
	\end{align*}
	Setting $u_n:=-v_n$, we claim that the sequence $(u_n)$ satisfies \eqref{normal-approx}. Indeed, since $-\div(\sigma_T \mu)$ belongs to $\AEd$ and $u_n\weakstar 0$, \eqref{criterion} implies that $\int \pairing {\sigma_T, \nabla u_n} \, d\mu \to 0$. It follows that $|\sigma_N|- \pairing {\sigma_N, \nabla u_n}$, which is non negative, converges strongly to $0$ in $L^1_\mu$.
	Possibly after extracting a subsequence, we can then assume that $\pairing {\sigma_N, \nabla u_n} \to |\sigma_N|$ at $\mu$ a.e. point. Since $|\nabla u_n|\le 1$, this entails that 
	$\nabla u_n\to \frac{\sigma_N}{|\sigma_N|} \, \one_{\{ \sigma_N \not=0\}}$ as claimed in
	\eqref{normal-approx}.
	
	\medskip
	We now show the equality \eqref{Nmu=} and the equivalence \eqref{orthocrit}. Let us denote
	\begin{align*}
		\A:= &\set{\sigma\in (L^\infty_\mu)^d}{\exists u_n \in \D(\Rd) : u_n \weakstar 0 \;\; \text{in $\Lip(\Rd)$},\ \nabla u_n\to \sigma \;\; 
			\text{$\mu$ a.e.}} \\ 
		\widetilde{\A} :=& \set{ \sigma\in (L^\infty_\mu)^d }{ \exists u_n \in \D(\Rd) : u_n \weakstar 0 \;\; \mathrm{in}\ \Lip(\Rd),\ \nabla u_n\weakstar \sigma \;\; \mathrm{in}\ (L^\infty_\mu)^d}.
	\end{align*}
	Clearly $\A \subset \widetilde{\A}$. On the other hand, one checks that $\widetilde{\A} \subset \N_\mu$. Indeed let $(u_n)$ be a sequence in $\D(\Rd)$ such that $u_n \weakstar 0$ and $\nabla u_n \weakstar \sigma$ in $(L^\infty_\mu)^d$. Then, $-\div(\sig_T\mu)$ belongs to $\AEd$, by applying \eqref{weak-star.crit}, we get $\int \pairing{ \nabla u_n, \sigma_T}\, d\mu = \pairing{- \div (\sigma_T \mu), u_n} \to 0$. It follows that
	$\int |\sigma_T|^2\, d\mu= \int \pairing{ \sigma, \sigma_T}\, d\mu =\lim_{n\to\infty} \int \pairing{ \nabla u_n, \sigma_T}\, d\mu = 0$, hence $\sigma_T=0$ holds $\mu$ a.e. and therefore $\sigma\in \N_\mu$.
	
	\medskip 
	The converse inclusion $\N_\mu \subset \A$ needed for \eqref{orthocrit} is more involved. In order to conclude, we are reduced to show the following claims:
	\begin{align}
		& \A^\perp \subset \T_\mu, \label{claim1} \\
		& \A = \widetilde{\A} \quad \text{and}\quad \widetilde{\A} \ \text{is a weakly-star closed subspace of $(L^\infty_\mu)^d$.} \label{claim2}
	\end{align}
	Indeed, by passing to the orthogonal in \eqref{claim1}, we get the inclusion $\N_\mu = \T_\mu^\perp \subset \A^{\perp\perp}$ while, by \eqref{claim2}, the identities $ \A= \widetilde{\A}= \A^{\perp\perp}\ $ all hold true.
	
	\subsubsection*{Proof of \eqref{claim1}} 
	Let $\zeta\in (L^1_\mu)^d$ such that $\zeta\in \A^\perp$. Denoting $\zeta_N= P_\mu^\perp(\zeta)$, we consider the normal vector function of $(L^\infty_\mu)^d$ defined by 
	$$\sigma:= \frac{\zeta_N}{|\zeta_N|}\, \one_{{\zeta_N\not=0}}\ ,$$
	for which we may retrieve a sequence satisfying \eqref{normal-approx}: there exists $(u_n)_n$ in $\Lip_1(\Rd)$ such that $u_n\weakstar 0$ and (since $\mathcal A \subset \mathcal A$) $\nabla u_n \weakstar \sigma$ in $(L^\infty_\mu)^d$.
	Therefore $\sig$ belongs to $\A$ and in particular $\int |\zeta_N| \, d\mu = \pairing{\zeta, \sigma} =0$. This proves $\zeta\in \T_\mu$, as desired.
	
	\subsubsection*{Proof of \eqref{claim2}}
	
	Consider the vector space $V := \set { (u, \nabla u)}{u \in \D(\Rd)\ ,\ u(0)=0}$, which we identify isometrically with a subspace of $ \Lip_0(\Rd) \times (L^\infty_\mu)^d$, the dual of $\AEd\times (L^1_\mu)^d $. Denoting by $ {\ov V}^{*\sigma}$ the weak-star closure of $V$, we may look at $\widetilde{ \A}$ as the intersection of ${\ov V}^{*\sigma}$ with the closed subspace $\{0\} \times (L^\infty_\mu)^d$. It follows that $\widetilde{\A}$ is weakly-star closed. It remains to show the equality $\A= \widetilde{\A}$. The inclusion $\A\subset \widetilde{ \A}$ is true since the convergence $\mu$ a.e. of a uniformly bounded approximating sequence $(\nabla u_n)$ implies its weak-star convergence in $(L^\infty_\mu)^d$. 
	The other set contention $\widetilde{\mathcal A} \subset \mathcal A$ is a consequence of Mazur's Lemma: there exists a convex combination $v_j$ of the elements of $(u_n)$ satisfying $v_j \weakstar 0$ in $\Lip(\R^d)$ and $\nabla v_j \to \sigma$ in $(L^1_\mu)^d$. By convexity of the Lipschitz seminorm also $\Lip(v_j) \le C$ where $C$ is the uniform Lipschitz constant of $(u_n)_n$. Moreover, up to extracting a subsequence, we may assume that $\nabla v_j \to \sigma \ \text{$\mu$ a.e.}$. This shows that $\sigma \in \mathcal A$.
\end{proof}

\subsection{Closability, weak tangential gradient, integration by parts}\label{sec:closability}

In this section, $\mu$ is a Radon measure in $\M_+(\Rd)$. Based on the construction of the tangent bundle $T_\mu$, we obtain a notion of tangential derivation which allows to extend classical operations of differential calculus to Lipschitz functions. 

Let us consider the operator $A_\mu: \Lip(\Rd) \to (L^\infty_\mu)^d$ with domain $D(A)= \D(\Rd)$ obtained by setting
\begin{align}\label{def:Amu}
	(A_\mu u)(x)\ =\ P_\mu(x) (\nabla u(x))\quad\text{for every $x\in \Rd$}.	
\end{align}
Note that $D(A)$ is weakly-star dense in $\Lip(\Rd)$ (see Remark \ref{Lipdual}). 
\begin{lemma}\label{closable} 
	$A_\mu$ is weakly-star closable in the sense that:
	\begin{align*}
		u_n \in \D(\Rd),\quad 
		u_n \weakstar 0 \ \text{in } \Lip(\Rd),\quad 
		P_\mu(\nabla u_n) \weakstar \xi \ \text{in } (L^\infty_\mu)^d
		\ \implies \ 
		\xi = 0.
	\end{align*}
\end{lemma}

\begin{proof} 
	For every $\sigma\in (L^1_\mu)^d$, we have $P_\mu \sigma \in \T_\mu$, hence $f:=- \div((P_\mu \sigma)\, \mu)$ belongs to $\AEd$. By applying \eqref{weak-star.crit}, it follows that:
	\begin{align*} 
		\int \pairing{\xi, \sigma}\, d\mu 
		=\lim_n \int \pairing{P_\mu(\nabla u_n), \sigma}\, d\mu 
		= \lim_n \int\pairing{\nabla u_n, P_\mu \sigma}\, d\mu 
		= \lim_n \pairing{ u_n, f} = 0.
	\end{align*} 
	Since $\sigma$ is arbitrary in $(L^1_\mu)^d$, we conclude that $\xi =0$ as wished.
\end{proof}

In virtue of Lemma \ref{closable}, we can extend $A_\mu$ in a unique way to all $u\in \Lip(\Rd)$, defining a weakly-star linear continuous operator
\begin{equation}\label{def:nabla_mu}
	\nabla_\mu: \ \Lip(\Rd) \mapsto (L^\infty_\mu)^d.
\end{equation}
Section~\ref{sec:weaver} below shows that $\nabla_{\mu}$ is a derivation in the sense of Weaver.
More precisely, we have the following characterization: 
\begin{theorem}\label{weakgrad}
	\begin{enumerate}
		\item\label{weakgrad:xi} Let $u\in \Lip(\Rd)$ and $(u_n)$ be any sequence in $C^1\cap\Lip(\Rd)$ such that
		$u_n \weakstar u$. Then the sequence $ (P_\mu (\nabla u_n))$ admits a unique weak-star cluster point $\xi\in (L^\infty_\mu)^d$. Moreover $\xi$ is independent of the sequence $(u_n)$ and, by setting $\nabla_\mu u=\xi$ in \eqref{def:nabla_mu}, we define a linear weakly-star continuous operator. 
		
		\item\label{weakgrad:un} Conversely, for any pair $(u,\zeta)\in \Lip(\Rd)\times\N_\mu$, there exists a sequence $(u_n)$ in $\D(\Rd)$ such that $ u_n \weakstar u$ while the convergence $\nabla u_n \to \nabla_\mu u + \zeta$ holds $\mu$ a.e. and weakly-star in $(L^\infty_\mu)^d$. In particular, we have the convergences 
		$\mu$ a.e.: 
		$$P_\mu(\nabla u_n) \to \nabla_\mu u\quad,\quad P_\mu^\perp (\nabla u_n) \to \zeta.$$ 
	\end{enumerate}
\end{theorem}

\begin{proof} 
	\ref{weakgrad:xi} The sequence $(P_\mu (\nabla u_n))$ is bounded in $ (L^\infty_\mu)^d$ and admits at least one weak-star cluster point $\xi$. By Lemma~\ref{closable}, $\xi$ is unique and does not depend of the approximating sequence $(u_n)$. Note that such a sequence exists since $\D(\Rd)$ is weakly-star dense in $\Lip(\Rd)$ (see Remark~\ref{Lipdual}). The fact that $\nabla_\mu$ is linear, weakly-star continuous and coincides with $P_\mu \circ \nabla$ on smooth functions is straightforward.
	
	\med \ref{weakgrad:un} Take any sequence $(\f_n)$ in $\D(\Rd)$ such that $\f_n\weakstar u$. Then, $(\nabla \f_n)$ is bounded in $(L^\infty_\mu)^d$ and, up to extracting a subsequence, converges weakly-star to some $\eta$. Observing that the latter convergence holds also weakly in $(L^2_{{\rm loc},\mu})^d$, by Mazur's lemma, we can pick $v_n$ in the convex hull of $\{\f_m, m\ge n\}$ such that $P_\mu (\nabla v_n)\to \eta$ strongly in $(L^2_{\mu,{\rm loc}})^d$. Clearly the Lipschitz constant of $v_n$ remains smaller than the uniform Lipschitz bound of $(\f_n)$ and $v_n\to u$ pointwise as well. Therefore $v_n \weakstar u$ and, possibly after the extraction of a new subsequence, we arrive to the convergence $\nabla v_n\to \eta $ holding $\mu$ a.e..
	From the assertion \ref{weakgrad:xi}, we know that $P_\mu (\nabla v_n) \weakstar \nabla_\mu u$, hence $P_\mu (\eta)=\nabla_\mu u$ so that $\nabla v_n \to \nabla_\mu u + P_\mu^\perp(\eta)$. Then, observing that $\sigma := \zeta- P_\mu^\perp (\eta)$ belongs to $\N_\mu$, it is enough to consider $u_n:= v_n + w_n$, where $(w_n)$ is the sequence in $\D(\Rd)$ obtained by applying \eqref{orthocrit}. Indeed this sequence satisfies $w_n\weakstar 0$ and $\nabla w_n \weakstar \sigma$ pointwise. Accordingly $ u_n \weakstar u$ and $\nabla u_n \to \nabla_\mu u + \zeta$ pointwise as wished.
\end{proof}

An important feature of the \emph{weak tangential gradient operator} $\nabla_\mu$ is that it satisfies the following integration by parts formula:
\begin{equation}\label{byparts}
	\int \pairing{\sigma, \nabla_\mu u}\, d\mu = \pairing{- \div(\sigma\mu), u} \quad \text{for every } (u,\sigma) \in \Lip(\Rd) \times\T_\mu 
\end{equation}
where the duality bracket is intended in the sense of the duality between $\AEd$ and $\Lip(\Rd)$. 
This follows from the next lemma where we show that \eqref{byparts} actually characterizes $\nabla_\mu u$. 
\begin{lemma}\label{parts} 
	Let $u \in \Lip(\Rd)$ and $\xi \in (L^\infty_\mu)^d$ be such that $\xi(x)\in T_\mu(x)$, $\mu$ a.e.. Then
	\begin{equation}\label{xi=grad}
		\xi = \nabla_\mu u \ \iff\ \int \pairing{\sigma, \xi}\, d\mu = \pairing{- \div(\sigma\mu), u} \quad \text{for every } \sigma\in \T_\mu. 
	\end{equation}
\end{lemma}

\begin{proof} 
	Let $\sigma\in \T_\mu$. Then, by \eqref{def:Fd} and \eqref{criterion}, the distribution $- \div(\sigma\mu)$ belongs to $\AEd$. By Theorem~\ref{weakgrad}, for any smooth sequence in $\D(\Rd)$ such that $u_n\weakstar u$, we have $\nabla_\mu u_n \weakstar \nabla_\mu u$ in $(L^\infty_\mu)^d$. Therefore:
	\begin{align*} 
		\pairing{- \div(\sigma\mu), u} 
		&= \lim_{n\to \infty} \pairing{- \div(\sigma\mu), u_n} 
		= \lim_{n\to \infty} \int \pairing{\sigma, \nabla u_n}\, d\mu \\
		&= \lim_{n\to \infty} \int \pairing{\sigma, \nabla_\mu u_n}\, d\mu = \int \pairing{\sigma, \nabla_\mu u}\, d\mu \ ,
	\end{align*} 
	where, in the last line, we used the fact that $\sigma \in \T_\mu$. It follows that the right hand side of \eqref{xi=grad} holds if and only if $\xi - \nabla_\mu u$ belongs to $\N_\mu = \T_\mu^\perp$. As $\xi(x) -\nabla_\mu u(x)\in T_\mu(x)$, this implies that $\xi = \nabla_\mu u$.
\end{proof}

As a consequence of Lemma \ref{parts}, we get a useful localization property of the tangential gradient which supplements the locality property \eqref{restriction}.
\begin{lemma}[Locality of the tangential gradient]\label{localgrad} 
	Let $u,v \in \Lip(\Rd)$ and $\mu, \nu$ two Radon measures in $\Rd$. Then
	\begin{align}
		u = v \quad \mu\text{ a.e.} \qquad &\implies \qquad \nabla_\mu u = \nabla_\mu v \quad \mu\text{ a.e.,} \label{localgrad:statementu} \\
		\nu \ll \mu \qquad &\implies \qquad \nabla_\mu u = \nabla_\nu u \quad \nu \text{ a.e..} \label{localgrad:statement}
	\end{align}
\end{lemma}

\begin{proof} 
	Since $\nabla_\mu$ is linear, does not see constants and $u = u^+ - u^-$ with $u^\pm\ge0$ Lipschitz and vanishing $\mu$-a.e., we may assume $u\ge0$. The proof of \eqref{localgrad:statementu} uses the pointwise differentiability along tangent fields established in the next section (assertion \ref{MKi:diff} of Theorem~\ref{MK-infinitesimal}).
	Let $\sig\in\T_\mu$. At $\mu$-a.e. $x$ we have $u(x) = 0$, hence
	\begin{align*}
		\pairing{\nabla_\mu u(x),\sig(x)}
		= \lim_{h\to0^+}\frac{u(x+h\sig(x))}{h} \ \ge\ 0 .
	\end{align*}
	Applying this to $-\sig\in\T_\mu$ gives the reverse inequality, so that $\pairing{\nabla_\mu u,\sig} = 0$ $\mu$-a.e. for every $\sig\in\T_\mu$, i.e. $\nabla_\mu u\in\N_\mu$. Since $\nabla_\mu u(x)\in T_\mu(x)$, \eqref{localgrad:statementu} follows.

	In order to prove \eqref{localgrad:statement}, consider $\sig \in \T_\nu$. Then $\tilde \sig \coloneqq \sigma \frac{d\nu}{d\mu}$ belongs to $\T_\mu$ by the local tangency criterion~\eqref{Tmu=}. Since ${\tilde\sig} \,\mu = \sigma\, \nu$, by using twice \eqref{byparts}, we get
	\begin{align*}
		\int \pairing{\sig, \nabla_\mu u}\, d\nu 
		= \int \pairing{\tilde{\sig}, \nabla_\mu u}\, d\mu 
		= \pairing{- \div(\tilde{\sig}\mu), u} 
		= \pairing{- \div(\sig\nu), u}.
	\end{align*}
	The equality \eqref{localgrad:statement} follows by applying \eqref{xi=grad} to the measure $\nu$ and $\xi= \nabla_\mu u$.
\end{proof}

\begin{remark}\label{rem:eqpp}
	Identity \eqref{localgrad:statementu} extends the classical locality of the weak derivative for Sobolev functions with $\mu = \mathcal L^d$: if $u,v \in W^{1,p}(\R^d)$, then $\nabla u \equiv \nabla v$ $\mathcal L^d$ almost everywhere on $\set{x \in \Rd}{u(x) = v(x)}$. As an example, one can consider $\mu$ as the 1-Hausdorff measure on the segment $[0,1] \times \{0\} \subset \R^2$, with $u(x,y) = y$. Then $u = 0$ at $\mu$ a.e. point, and $\nabla_\mu u = 0$, even though $\nabla u \neq 0$.
\end{remark}

\begin{remark}\label{rem:closability}
	The closability result appearing in Lemma \ref{closable} (see also in \cite{BCJ2005}) has been recently extended in \cite{ABM} to more general differential operators and function spaces. In particular, Theorem 1.3 and Section 1.7 therein yield another proof of the fact that $T_{\mu} = V_{\mu}$, relying on~\cite{AM}.
\end{remark}

\section{Stability of the MK distance}\label{tangential calculus}

In this section, we investigate the first-order stability of the Monge-Kantorovich distance under infinitesimal displacements of a general measure $\mu$ along velocity fields $\sigma \in (L^1_\mu)^d$, given by $\mu_h^\sigma = (\text{id} + h\sigma)_\sharp \mu$. Our main result shows that the Monge-Kantorovich tangent bundle $\T_\mu$ characterizes the subset of stable directions. As an application, we provide a proof that Lipschitz functions are differentiable almost everywhere in the directions of the MK tangent bundle, recovering a result previously established in \cite{AM}.

\subsection{Statements}

Let $\mu \in \M_+(\Rd)$ be a Radon measure and $h > 0$. For any velocity field $\sigma\in (L^1_\mu)^d$, and any $u\in C^1$, we set:
\begin{align}
	\Delta_h^\sigma(u)(x) &:= \frac{u(x+h \sig(x))-u(x)}{h} - \pairing{\nabla u(x),\sig(x)}, \label{Delta-hsig} \\
	F_h(\mu,\sigma) &:= \sup_{ u\in \Lip_1 \cap C^1 (\Rd)} \left\{ \int \Delta_h^\sig(u) (x)\, d\mu \right\}, \label{def:Fh}
\end{align}
and 
\begin{equation}\label{def:F0}
	F_0^+(\mu,\sigma) := \limsup_{h\to 0} F_h(\mu,\sigma),\qquad 
	F_0^-(\mu,\sigma) := \liminf_{h\to 0} F_h(\mu,\sigma).
\end{equation}
If $\sigma \in \T_\mu$, then according to \eqref{def:nabla_mu}, the expression \eqref{Delta-hsig} is extended to all $u\in \Lip(\Rd)$ by 
\begin{equation} \label{Delta-hsig*}
	\Delta_h^\sigma(u)(x) := \frac{u(x+h \sig(x))-u(x)}{h} - \pairing{\nabla_\mu u(x),\sig(x)}.
\end{equation}
In this case, by Remark~\ref{Lipdual}, the supremum in \eqref{def:Fh} is the KR norm of an element of $\AEd$, which is expressed through the difference quotients $\diffquot{\sig}{h} = \frac{\mu_h^\sig-\mu}{h}$ with $\mu_h^{\sig} \coloneqq (\mathrm{id}+h\sigma)_{\sharp} \mu$: 
\begin{equation}\label{defKR:Fh}
	F_h(\mu,\sigma) = \left\Vert \diffquot{\sig}{h} + \div (\sigma\, \mu) \right\Vert_{\KR} \qquad \text{for } \sig\in\T_\mu.
\end{equation}
Note that $\diffquot{\sig}{h}$ belongs to $\AEd$ for every $\sig\in(L^1_\mu)^d$ since it is a balanced measure, with $\|\diffquot{\sig}{h}\|_{\KR} \le \int|\sig|\,d\mu$.
The main result of this section is the following stability result. 

\begin{theorem}\label{MK-infinitesimal} 
	Let $\mu \in \M_+(\Rd)$ be a Radon measure and let $T_\mu$ be its MK-tangent bundle. Then:
	\begin{enumerate}
		\item\label{MKi:form} for any vector field $\sigma\in (L^1_\mu(\Rd))^d$, we have the equalities
		\begin{equation}\label{CVFh}
			F_0^+(\mu, \sigma)\, =\, F_0^-(\mu,\sigma) \, = \, F(\mu,\sigma) := \ 2\int |\sigma_N|\, d\mu,
		\end{equation}
		where $\sigma_N = P_\mu^\perp \sigma$ denotes the normal component of $\sigma$. As a consequence, we have the following equivalence:
		\begin{equation}\label{main-crit}
			\sig \in \T_\mu \ \iff\ \lim_{h\to 0} F_h(\mu,\sigma)=0.
		\end{equation}
		
		\med
		\item\label{MKi:diff} Let $\sig\in \T_\mu$. Then $\Delta_h^\sig(u)\to 0$ in $L^1_\mu$, for every $u\in\Lip(\Rd)$. Consequently, any Lipschitz function $u$ is differentiable in the direction $\sig$, namely:
		\begin{align*}
			\frac{u(x+h \sig(x))-u(x)}{h}
			\to \pairing{\nabla_\mu u(x),\sig(x)} \quad
			\mu \text{ a.e. and in } L^1_\mu(\Rd).
		\end{align*}

		\med
		\item\label{MKi:full} Every $u\in\Lip(\Rd)$ is differentiable at $\mu$-a.e. $x$ with respect to the whole subspace $T_\mu(x)$, with differential $\pairing{\nabla_\mu u(x),\cdot}$:
		\begin{equation}\label{eq:MKifull}
			u(x+v) \ =\ u(x) + \pairing{\nabla_\mu u(x),\, v} + o(|v|) \qquad \text{as } T_\mu(x)\ni v\to 0 .
		\end{equation}
	\end{enumerate}
\end{theorem}
The proof of Theorem~\ref{MK-infinitesimal} is contained in the subsection~\ref{proofmain}, after some comments and preliminary results.

\begin{remark}\label{Gamm-conv} 
	Since the family of functionals $(F_h(\mu,\cdot))$ is equi-Lipschitz in $(L^1_\mu)^d$ (see Lemma~\ref{pFh} below), the pointwise convergence \eqref{CVFh} implies that $F_h(\mu,\cdot)$ $\Gamma$-converges to $F(\mu,\cdot)$
	as $h\to 0$ in the Banach space $(L^1_\mu)^d$. 
\end{remark}

\begin{remark}\label{1-2} 
	The factor $2$ appearing in the limit $F(\mu,\sigma)$ might look counterintuitive since we expect the two terms 
	$\int \left(\frac{u(x+h\sigma)-u(x)}{h}\right) \, d\mu$ and $\int\pairing{\nabla u,\sigma}\, d\mu$
	in \eqref{def:Fh} to share the same sign while both of them have absolute values not larger than $\int |\sigma|\, d\mu$.
	The simple example below illustrates that it is not the case as far as one considers the supremum over $\Lip_1(\Rd)$.
\end{remark}

\begin{example*}
	Let $\mu = \H^1\res([0,1]\times \{0\})$ in $\R^2$ and $\sigma = (0,1)$. Then $\int |\sigma_N|\, d\mu =1$ and by taking the 1-Lipschitz function $u_h(x_1,x_2) := |x_2-h^2|$ for $h\le1$ (which we may mollify near the line $\{x_2=h^2\}$, at no cost for the values used below, so as to meet the requirement $u\in C^1$ of \eqref{def:Fh}), we obtain the lower bound:
	\begin{align*} 
		F_h(\mu,\sigma) &\ge \int_0^1 \left( \frac{u_h(x_1,h) -u_h(x_1,0)}{h} - \pairing{\nabla u_h (x_1,0), (0,1)}\right)\, dx_1 \\
		&= \int_0^1 \left( \frac{h-h^2-h^2}{h} +1\right) \, dx_1 = 2-2h.
	\end{align*}
\end{example*}

\begin{remark}\label{uniform-diff} 
	The convergence $F_h(\mu,\sigma)\to 0$ obtained when $\sigma\in \T_\mu$ cannot be applied if alternatively, we consider $\widetilde{F_h}$ defined by:
	$$\widetilde{F_h} (\mu,\sigma):= \sup_{ u\in \Lip_1 \cap C^1 (\Rd)} \left\{ \int \left|\Delta_h^\sig(u)\right|\, d\mu
	\right\}. $$
	Despite the fact that, for any Lipschitz function $u$, the integral above vanishes as $h\to 0$ thanks to the $\mu$-a.e. differentiability in the direction $\sigma$ (see assertion \ref{MKi:diff}), we cannot expect this convergence to be uniform with respect to $u\in \Lip_1(\Rd)$. This is confirmed by the example below.
\end{remark}

\begin{example*}\label{nonuniform} 
	Let $d=1$ and consider $\mu = \mathcal{L}_{[0,1]}$, and $\sigma(x) = 1$ for any $x \in [0,1]$. Then $\sigma \in L^1_{\mu}$ belongs to $\T_{\mu}$. However, we claim that
	\begin{align}\label{lim>0}
		\liminf_{h \searrow 0} \sup_{u \in \Lip_1} \int_{x \in [0,1]} \left|\frac{u(x + h) - u(x)}{h} - u'(x)\right| dx \ \ge \frac{2}{\pi}.
	\end{align}
	Indeed, for each $h > 0$, the function $u : x \mapsto \frac{h}{2\pi} \sin(2 \pi \frac{x}{h} )$ is 1-Lipschitz and vanishes at $x = 0$. There holds
	\begin{align*}
		\left|\frac{u(x + h) - u(x)}{h} - u'(x)\right|
		= \left|\frac{ \sin(2 \pi \frac{x}{h} + 2 \pi) - \sin(2 \pi \frac{x}{h})}{2\pi} - \cos(2 \pi x/h)\right|
		= |\cos(2 \pi x/h)|.
	\end{align*}
	Since the function $|\cos(2 \pi y)|$ is $1$-periodic, the integral $\int_{[0,1]} |\cos(2 \pi \frac{x}{h})| dx$ converges to the mean value $\int_{[0,1]} |\cos(2 \pi y)| dy= \frac{2}{\pi}$, hence the lower bound \eqref{lim>0}.
\end{example*}

\begin{remark}
	A natural variant of $F_h(\mu,\sigma)$ is obtained by averaging two opposite velocity fields.
	Setting $\widehat{\mu^{\sig}_h} := \frac12\big(\mu^{\sig}_h + \mu^{-\sig}_h\big)$, one has
	$$ \sup_{u\in \Lip_1\cap C^1(\Rd)} \int \frac{\Delta_h^\sig(u)+ \Delta_h^{-\sig}(u)}{2}\, d\mu
	\ =\ \left\Vert\frac{\widehat{\mu^{\sig}_h} -\mu}{h}\right\Vert _{\KR}
	\ =\ \frac1{h}\, W_1\big(\widehat{\mu^{\sig}_h},\mu\big) , $$
	which, in contrast with \eqref{defKR:Fh}, involves no divergence term. As this quantity is bounded by
	$\frac12\big(F_h(\mu,\sig)+F_h(\mu,-\sig)\big)$, the criterion \eqref{main-crit} gives
	$\T_\mu\subset\T_\mu^{sym}$, where $\T_\mu^{sym}$ denotes the set of those $\sig\in(L^1_\mu)^d$ for
	which $h^{-1}W_1(\widehat{\mu^{\sig}_h},\mu)\to0$. The inclusion is strict in general. This is demonstrated in \cite{A26}, where the analogous $W_2$ question is studied for $\mu\in\PP_2(\Rd)$, provides uniform
	measures on Cantor subsets of $[0,1]$ with $\T_\mu=\{0\}$ for which $\sig\equiv1$ belongs to
	$\T_\mu^{sym}$. A characterization of $\T_\mu^{sym}$ seems to be still out of reach.
\end{remark}
Next we derive a relation involving a given velocity field $\sig\in (L^1_\mu)^d$ and the asymptotic behavior as $h \to 0_+$ of the sequence ${\diffquot{\sig}{h}}$. Although the total variation of this sequence blows up in general, it admits an uniform upper bound in the Banach space $\AEd$, namely $\|{\diffquot{\sig}{h}}\|_{\KR} \le \int |\sig|\, d\mu$.
Then, even if ${\diffquot{\sig}{h}}$ fails to converge in $\AEd$, we can define its \emph{asymptotic radius} $\rho:= \inf_{f\in \AEd} H(\sig,f)$ where
\begin{equation}\label{def:asymptotic}
	H(\sig, g) := \limsup_{h\to 0_+} \| {\diffquot{\sig}{h}} - g\|_{\KR}.
\end{equation} 

Accordingly the \emph{asymptotic center} of the sequence $\{\diffquot{\sig}{h}\}$ is the (possibly empty) set of minimizers $\set{g\in \AEd}{H(\sig,g)=\rho}$. Clearly $\{\diffquot{\sig}{h}\}$ is a Cauchy sequence in $\AEd$ if and only if $\rho=0$. In this case, the asymptotic center reduces to the singleton $\{f\}$, where $f$ is the limit of $\{\diffquot{\sig}{h}\}$, which exists since $\AEd$ is complete. If $\rho > 0$, despite the fact that the functional $H(\sig, \cdot): \AEd \to \R_+$ defined in \eqref{def:asymptotic} is convex Lipschitz and coercive, it is not clear that the asymptotic center is non-empty since $\AEd$ is not reflexive.

\begin{corollary}\label{cauchy} 
	Let $\sig\in (L^1_\mu)^d$ be a velocity field and denote $\sig_T = P_\mu \sig$ and $\sig_N = P_\mu^\perp \sig$ its tangential and normal components respectively; then we define
	\begin{equation}\label{def:rovf}
		f_T :=- \div (\sig_T \mu)\quad,\ r:= \int |\sig_N|\, d\mu \ .
	\end{equation}
	\begin{enumerate}[label={(\roman*)}]
		\item\label{cauchy:estim} $f_T \in \AEd$ and the functional $H(\sig,\cdot)$ defined in \eqref{def:asymptotic} satisfies the inequalities
		\begin{align}\label{centerfh}
			r \ \le\ H(\sig, g) \ \le\ r + \|g-f_T\|_{\KR} \qquad \text{for every $g\in\AEd$}.
		\end{align}
		\item\label{cauchy:iffseq} the sequence $\left\{ \diffquot{\sig}{h}\right\}$ is a Cauchy sequence in $\AEd$ if and only if $\sig\in\T_\mu$. In this case, we have
		\begin{align}\label{MK:AEconv}
			\lim_{h\to 0} \left\Vert \diffquot{\sig}{h} + \div (\sigma\, \mu) \right\Vert_{\KR} = 0.
		\end{align}
		\item\label{cauchy:redsing} If $\sig\notin\T_\mu$, then the sequence $\left\{ \diffquot{\sig}{h}\right\}$ has positive asymptotic radius $\rho = r = \int |\sig_N| \, d\mu$ and asymptotic center reducing to the singleton $\{f_T\}$. In particular we have
		\begin{align}\label{defect}
			\lim_{h\to 0} \left\Vert \diffquot{\sig}{h} + \div (\sigma_T\, \mu) \right\Vert_{\KR} = \int |\sig_N| \, d\mu.
		\end{align}
	\end{enumerate}
\end{corollary}

\begin{proof} 
	By \eqref{def:Tmu}-\eqref{Tmu=}, we already know that $f_T$ belongs to $\AEd$. Since $\sig_T\in\T_\mu$, it follows from \eqref{defKR:Fh} and \eqref{main-crit} that 
	$$ \lim_{h\to 0_+} \|\diffquot{\sig_T}{h} - f_T\|_{\KR} \ =\ 0.$$
	On the other hand, we have the upper bound:
	$$ \|\diffquot{\sigma_T}{h} - \diffquot{\sig}{h}\|_{\KR} \le\ \int |\sig - \sig_T| \, d\mu = \int |\sig_N| \, d\mu = r.$$ 
	Therefore, for every $g\in \AEd$, we infer that:
	\begin{align*} 
		H(\sig, g) &= \limsup_{h\to 0_+} \| g- \diffquot{\sig}{h}\|_{\KR} \\
		&\le \| g - f_T\|_{\KR} + \limsup_{h\to 0_+} \|f_T - \diffquot{\sig_T}{h} \|_{\KR} + \limsup_{h\to 0_+} \| \diffquot{\sig_T}{h} - \diffquot{\sig}{h} \|_{\KR} \\
		&\le \| g - f_T\|_{\KR} + 0 + r.
	\end{align*} 
	Let us now establish the lower bound $H (\sig, g) \ge r=\int |\sig_N|\, d\mu$. For every $u\in \D\cap \Lip_1(\Rd)$, one has 
	\begin{align}\label{H>G*}
		H(\sig, g) \ge \liminf_{h\to 0_+} \| \diffquot{\sig}{h} -g\|_{\KR} \ge \liminf_{h\to 0_+} \pairing{\diffquot{\sig}{h}-g,u} = \int \pairing{\sigma,\nabla u}\, d\mu - \pairing {g,u}
	\end{align}
	where to pass to the last equality, we used the the distributional convergence $\diffquot{\sig}{h} \to - \div(\sig\, \mu)$. Let us apply the inequality above replacing $u$ by the sequence $(u_n)$ in $\D\cap \Lip_1(\Rd)$ given in \eqref{normal-approx} (see Proposition~\ref{Nmu}) (i.e. $ u_n \weakstar 0$ and $ \nabla u_n \to \frac{\sigma_N}{|\sigma_N|} \, \one_{\{\sigma_N \neq 0\}}$ \ $\mu$ a.e.), we conclude that 
	$$ H (\sig, g) \ge \liminf_n \int \pairing{\sigma,\nabla u_n}\, d\mu - \pairing {g,u_n} \ \ge \int |\sig_N|\, d\mu,$$
	whence \eqref{centerfh} and the assertion \ref{cauchy:estim}. 
	
	\med \ref{cauchy:iffseq}\ Assume that $\diffquot{\sig}{h}$ is a Cauchy sequence in $\AEd$. Then, since $\AEd$ is complete, we have that $\diffquot{\sig}{h} \to g$ for a suitable $g\in \AEd$. Since this convergence holds also in the distributional sense, the only possibility is that $g= -\div (\sig\, \mu)$. Therefore, by \eqref{def:Tmu}, $\sig\in \T_\mu$. Conversely, if $\sig\in \T_\mu$, then $\sig=\sig_T$, $\sig_N = 0$ and by applying \eqref{centerfh} to $g= -\div (\sig\, \mu)$, we get $H(\sig,g)=0 $, hence the strong convergence $\diffquot{\sig}{h}\to g$ in $\AEd$.
	
	\med \ref{cauchy:redsing}\ The equality $\rho=r$ and \eqref{defect} follow by applying \eqref{centerfh} to $g=f_T$. There remains to prove that no other element of $\AEd$ is an asymptotic center. Assume that $H(\sig,g) = r$ for some $g\in\AEd$. By taking the supremum with respect to $u$ (or $-u$) in $\D\cap\Lip_1(\Rd)$ in the right hand side of \eqref{H>G*}, we obtain 
	\begin{align*}
		H(\sig,g) \ge G_\sig^*(g) = \min_{\la\in \M^d} \left\{\int_{\Rd} |\la-\sigma\, \mu | \ : \ -\div \la = g \right\},
	\end{align*}
	where $G_\sig$ is defined in \eqref{def:Gsigma} while the second equality follows from \eqref{Gsigma*}. Let $\ov \la= w\, \mu +\la^s$ be optimal in the minimum above, being $w \mu$ the absolutely continuous part and $\la^s\perp \mu$. Since $g \in \AEd$ and $T_{|\la|}=\T_\mu$ holds $\mu$ a.e. (in virtue of \eqref{restriction}), we infer that $w\in \T_\mu$. Therefore 
	\begin{align*}
		H(\sig,g) \ge G_\sig^*(g) = \int \sqrt{ |w-\sig_T|^2 + |\sig_N|^2}\, d\mu + \int |\la^s|.
	\end{align*}
	Then we observe that the equality $H(\sig,g)=r= \int |\sig_N|\, d\mu$ implies that
	\begin{align*}
		\int \left(\sqrt{ |w-\sig_T|^2 + |\sig_N|^2}- |\sig_N|\right)\, d\mu = \int |\la^s| = 0,
	\end{align*}
	hence $\ov\la= \sig_T\, \mu$ and $g= -\div (\sig_T\, \mu)$, as wished.
\end{proof}

\begin{remark}
	The right hand-side of \eqref{defect} has a coefficient 1, instead of the coefficient 2 in \eqref{CVFh}. This comes from the fact that \eqref{CVFh} compares the differential quotient with $\div(\sigma \mu) = \div(\sigma_N \mu) + \div(\sigma_T \mu)$, while \eqref{defect} only contains the latter term $\div(\sigma_T \mu)$. The term $\div(\sigma_N \mu)$ accounts for the distance between $\diffquot{\sigma}{0}$ and $\AEd$ in the bidual of $\AEd$, measured with a suitable extension of the Kantorovich-Rubinstein norm.
\end{remark}

\subsection{Preliminary results}

In the proof of the assertion \ref{MKi:form} of Theorem \ref{MK-infinitesimal}, we will need some properties of 
$F_h(\mu,\cdot)$ and $F_0^\pm(\mu,\cdot)$. For convenience, we define counterparts on every vector measure $\la\in \Md$ by 
\begin{equation}\label{def:Phih}
	\Phi_h(\la) := F_h(\mu,\sigma), \quad \text{where}\quad \mu= |\la|,\ \sigma= \frac{d\la}{d |\la|}.
\end{equation}
Accordingly we let $\Phi_0^\pm(\la)= F_0^\pm(\mu,\sigma)$. 
Similarly, in the proof of the assertion \ref{MKi:diff} of Theorem \ref{MK-infinitesimal}, we need to associate the following functional on $\Lip(\Rd)$:
\begin{equation}\label{def:Dhu}
	\Psi_h(\la)(u) \coloneqq \int \widehat{\Delta^{\sigma}_h} (u) \ d\mu \qquad \text{where} \quad \mu= |\la|,\ \sigma= \frac{d\la}{d |\la|},
\end{equation}
where the function $\widehat{\Delta^{\sigma}_h} (u)$ is defined by
\begin{equation}\label{Delta-hsup} 
	\widehat{\Delta^{\sigma}_h} (u)\coloneqq \sup_{s \in (0,h]} |\Delta^{\sigma}_s (u)|.
\end{equation}
Note that the condition $\la\in\Fd$ implies that $\sig\in \T_\mu$ which is necessary in order to extend the definition of $\Delta_s^\sig(u)$ given by \eqref{Delta-hsig*} to Lipschitz functions.
	
\begin{lemma}\label{pFh} 
	Let $\mu\in \M_+(\Rd)$. Then: 
	\begin{enumerate}
		\item\label{pFh:resB} For every $\sigma\in (L^1_\mu)^d$ and every Borel subset $B\in \B(\Rd)$, we have
		$$F_h(\mu, \sigma \one_B) = F_h(\mu\res B, \sigma).$$
		
		\item\label{pFh:subadd} The map $B \mapsto F_h(\mu,\sigma \one_B)$ is subadditive on disjoint Borel subsets. By taking the limit sup, one deduces that $F_0^+$ is also subadditive.
		
		\item\label{pFh:posh} $\Phi_h$ and $F_h(\mu,\cdot)$ are positively homogeneous, in that the equalities
		\begin{align*}
			\Phi_h(t \lambda) = t \Phi_h(\lambda), \quad
			F_h(t \mu, \sigma) = t F_h(\mu,\sigma)
		\end{align*}
		hold for every $t \in \R_+$. Hence $\Phi_0^{\pm}$ and $F_0^{\pm}$ are positively homogeneous. Moreover $F_h(\mu,t\sig) = t\,F_{th}(\mu,\sig)$, so that $F_0^\pm(\mu,\cdot)$ is positively homogeneous in $\sig$ as well.
		
		\item\label{pFh:lip} The functionals $F_h(\mu,\cdot)$, $F_0^\pm(\mu,\cdot)$ are $2$-Lipschitz in $(L^1_\mu)^d$, while $\Phi_h$, $\Phi_0^{\pm}$ and $\Psi_h$ are $3$-Lipschitz in $\Md$. Specifically, for any $\lambda_1,\lambda_2 \in \M^d$ and $u \in \Lip_1$,
		\begin{align}
			\left|\Phi_h(\lambda_2) - \Phi_h(\lambda_1)\right| \ &\le \ 3 \|\lambda_1 - \lambda_2\|, \label{PhiLip} \\
			\left|\Psi_h(\lambda_2)(u) - \Psi_h(\lambda_1)(u)\right| \ &\le \ 3 \|\lambda_1 - \lambda_2\|. \label{PsiLip}
		\end{align}
	\end{enumerate}
\end{lemma}

\begin{proof}
	The assertion \ref{pFh:resB} follows from $\Delta_h^{\sig\one_B} = \one_B \Delta_h^\sig$. The subadditivity \ref{pFh:subadd} is direct from the definition as a supremum of integrals. The assertion \ref{pFh:posh} is straightforward once one notices that $\frac{d \lambda}{d|\lambda|} = \frac{d(t \lambda)}{d|t\lambda|}$ for any $t > 0$. For the assertion \ref{pFh:lip}, let $u\in \Lip_1(\Rd)$. For any $\sig_1,\sig_2\in (L^1_\mu)^d$,
	\begin{equation}\label{pFh:ineq2}
		\left|\Delta_h^{\sig_2}(u)-\Delta_h^{\sig_1}(u)\right|
		\ \le\ \frac{\left|u(x+h\sig_2)-u(x+h\sig_1)\right|}{h} + \left|\pairing{\nabla u,\sig_2-\sig_1}\right|
		\ \le\ 2 \left|\sig_2-\sig_1\right| ,
	\end{equation}
	so that, for every $\mu$, $F_h(\mu,\cdot)$ and $F_0^\pm(\mu,\cdot)$ are $2$-Lipschitz. For $\Phi_h$ and $\Psi_h$ the measure varies as well, and we compare the two integrands on $m:=|\la_1|+|\la_2|$. Set $p_i:=\frac{d\la_i}{dm}$, $a:=|p_1|$, $b:=|p_2|$, $\omega_i:=p_i/|p_i|$, $\theta:=|\omega_1-\omega_2|$ and $V_i:=\Delta_h^{\omega_i}(u)(x)$, so that $\int\Delta_h^{\sig_i}(u)\,d|\la_i| = \int |p_i|\,V_i\,dm$ while $|V_i|\le2$ and, by \eqref{pFh:ineq2}, $|V_1-V_2|\le2\theta$. Decomposing $aV_1-bV_2$ as $a(V_1-V_2)+(a-b)V_2$ and as $b(V_1-V_2)+(a-b)V_1$, and using $|p_1-p_2|^2=(a-b)^2+ab\,\theta^2$ together with $ab\ge\min\{a,b\}^2$, we get the pointwise bound
	\begin{equation}\label{jLip}
		\left|a V_1 - b V_2\right| \ \le\ 2\big(\min\{a,b\}\,\theta + |a-b|\big) \ \le\ 2\sqrt2\,|p_1-p_2| \ \le\ 3\,|p_1-p_2|.
	\end{equation}
	Integrating \eqref{jLip} against $m$ and taking the supremum over $u$ gives \eqref{PhiLip}, and \eqref{PsiLip} follows from the same computation with $V_i:=\widehat{\Delta_h^{\omega_i}}(u)(x)$. The tangential gradients adapt well to the argument, since $|\la_i|\ll m$ forces $\nabla_{|\la_i|}u=\nabla_m u$ at $|\la_i|$-a.e. point by \eqref{localgrad:statement}, and taking absolute values and the supremum over $s\in(0,h]$ are $1$-Lipschitz operations.
\end{proof}

Next we consider a measure $Q$ on the Banach space $\Md$, endowed with the Borel $\sigma$-algebra of the weak-$*$ topology, 
such that $\int \Vert\la\Vert \, dQ <+\infty$. Then we denote $\la_Q = \int \la\, dQ(\la)$ where the integral is intended in the weak sense as follows:
\begin{align*}
	\pairing{\la_Q,\psi}\, =\, \int \pairing{ \la, \psi}\, dQ(\la) \quad \text{for every } \psi\in (C_0(\Rd))^d.
\end{align*}
It is not restrictive for us to assume that $Q$ is supported on a closed ball $B_R:= \{\la \in \Md : \|\la\|\le R\}$
(for $R$ suitably large) on which the weak-$*$ topology is metrizable and compact.
\begin{lemma}\label{propQ} 
	Let us assume that the measure $Q$ above satisfies the condition:
	\begin{equation}\label{Qcomplete}
		\Vert \la_Q\Vert = \int \Vert\la\Vert \, dQ(\la) <+\infty. 
	\end{equation}
	Then:
	\begin{enumerate}
		\item\label{propQ:loc} $|\la_Q| = \int |\la| \, dQ(\la)$ and, for $Q$ almost all $\la$, it holds
		$\frac{d\la_Q}{d|\la_Q|} = \frac{d\la}{d|\la|}$, $|\la|$-a.e.;
		\item\label{propQ:sub} For every $h>0$, the Lipschitz functionals $\Phi_h$ satisfies the subadditivity property:\footnote{That \eqref{subPhih} is well defined deserves an argument, that is, why $\Phi_h$ is $Q$-measurable. The supremum of $\Delta_h^{\sig_\la}(u)$ in $\Phi_h(\lambda)$ runs over $\Lip_1\cap C^1$, which is separable for the topology of local uniform convergence of the function and of its gradient. Along a sequence $u_n\to u$ in that topology the integrands $\Delta^{\sig_\la}_h$ converge pointwise while staying bounded by $2$, so that the supremum may be restricted to a countable dense family. This reduces the question to the $Q$-measurability of $\lambda \mapsto \Delta_h^{\sigma_\lambda}(u)$. This is where the assertion \ref{propQ:loc} is needed: it gives $\sig_\la(x)=\sig_Q(x)$ at $|\la|$ a.e. point $x$, for $Q$ a.e. $\la$. This yields for each $u \in \Lip(\R^d)\cap C^1$ that 
		\begin{align*}
			\int \Delta^{\sigma_\lambda}_h(u) \, d|\lambda| & = \int \Delta^{\sigma_Q}_h(u) \, d|\lambda| \qquad \text{$Q$ a.e..}
		\end{align*}
		The measurability of the right-hand side is then handled by the other half of \ref{propQ:loc}: the identity $|\la_Q|=\int|\la|\,dQ$ makes $\la\mapsto\int f\,d|\la|$ a $Q$-measurable map for every bounded Borel $f$, which we apply to $f=\Delta_h^{\sig_Q}(u)$. Thus $\Phi_h$ is a countable supremum of $Q$-measurable maps.} 
		\begin{equation}\label{subPhih}
			\Phi_h( \la_Q) \ \le\ \int \Phi_h(\lambda) \, dQ(\lambda).
		\end{equation}
		\item\label{propQ:Psi} Assume in addition that $\la \in \Fd$ for $Q$ almost all $\la$. 
		Then $\la_Q\in \Fd$ and for every $h>0$, we have the equality:
		\begin{equation} \label{subPsih}
			\Psi_h(\la_Q)(u) \ =\ \int \Psi_h(\lambda)(u) \, dQ(\lambda) \quad\text{for every } u\in \Lip(\Rd).
		\end{equation}
	\end{enumerate}
\end{lemma} 

\begin{proof}
	\ref{propQ:loc} For any open subset $A\subset \Rd$ and any $\f\in (C_0(\Rd))^d$ such that $\spt(\f)\subset A$ and $|\f|\le 1$, it holds 
	$$\pairing{\f, \la_Q} = \int \pairing{\f, \la}\, dQ(\la) \le \int |\la|(A) \, dQ(\la).$$
	By passing to the supremum over such $\f$, we get the inequality $|\la_Q| \le \int |\la| \, dQ(\la)$ as measures in $\M_+(\Rd)$.
	The assumption implies that these two measures have the same total mass, hence coincide.
	Setting $\sigma_Q:=\frac{d\la_Q}{d|\la_Q|}$ and $\sigma_\la:= \frac{d\la}{d|\la|}$, which are unit vectors, we deduce that:
	$$ |\la_Q| = \pairing{\sigma_Q, \la_Q} = \int \pairing{\sigma_Q, \sigma_\la} |\la| \, dQ(\la) \le \int |\la| \, dQ(\la)= |\la_Q|,$$
	whence the equalities $\pairing{\sigma_Q,\sigma_\la}=1$ and $\sigma_Q=\sigma_\la$ holding $|\la|$-a.e. and for $Q$-almost all $\la \in \Md$.
	
	\med
	\ref{propQ:sub} Denote $\mu:=|\la_Q|$ and $\mu_\la:= |\la|$. Since $\la_Q= \sigma_Q\, \mu$ is the polar decomposition of $\la_Q$, in virtue of \eqref{def:Fh}, \eqref{def:Phih} and
	recalling the notation \eqref{Delta-hsig}, we have:
	\begin{equation}\label{PhilaQ}
		\Phi_h(\la_Q) = F_h(\mu, \sigma_Q) = \sup_{u\in \Lip_1\cap C^1} \int \Delta_h^{\sigma_Q}(u)\, d\mu.
	\end{equation}
	
	By the assertion \ref{propQ:loc}, we have that $\mu= \int \mu_\la \; dQ(\lambda)$ while the unit vectors $\sigma_Q$ and $\sigma_\la$ agree $\mu_\la$ -a.e.
	Accordingly, for every $u\in \Lip_1\cap C^1$, one has:
	\begin{align*}
		\int \Delta_h^{\sigma_Q}(u)\, d\mu &= \int \left( \int \Delta_h^{\sigma_\la}(u)\, d\mu_\la \right) \, dQ(\lambda)\\
		&\le \int F_h (\mu_\la, \sigma_\la)\, dQ(\la) = \int \Phi_h (\la)\, dQ(\la).
	\end{align*}
	In virtue of \eqref{PhilaQ}, the inequality \eqref{subPhih} of the assertion \ref{propQ:sub} follows by taking the supremum with respect to $u$ in the left hand side above.
	
	\med \ref{propQ:Psi} We show first that $\la_Q= \int \la\, dQ(\la)$ belongs to $\Fd$. Let $(\f_n)$ be a sequence in $\D(\Rd)$ such that $\f_n \weakstar 0$. By testing the former integral against $\nabla\f_n$, we are led to 
	$$ \pairing{- \div \la_Q, \f_n} = \pairing{\la_Q, \nabla\f_n}= \int \pairing{\la, \nabla\f_n}\, dQ(\la).$$
	Since $Q$-almost every $\la$ belongs to $\Fd$, we have that $\pairing{\la, \nabla\f_n}$ converges pointwise to zero while $|\pairing{\la, \nabla\f_n}|\le C \| \la\|$, where $C= \sup \Lip(\f_n)$. 
	The dominated convergence can be then applied since $\int \|\la\| \, dQ(\la) < +\infty$ and we conclude
	that $ \pairing{- \div \la_Q, \f_n} \to 0$, whence $\la_Q\in \Fd$ in virtue of \eqref{criterion}. 
	
	\medskip Next, let us fix $u\in \Lip(\Rd)$. We know that, for $Q$ almost all $\la$, 
	the equality $\sig_\la= \sig_Q$ is valid $\mu_\la$ a.e. while, in virtue of Lemma \ref{localgrad} that we apply to $\mu_\la\ll \mu$, we also have $ \nabla_\mu u = \nabla_{\mu_\la} u$. 
	Accordingly the equality $\widehat{\Delta_h^{\sigma_Q}}(u)=\widehat{ \Delta_h^{\sigma_\la}}(u)$ holds $\mu_\la$ a.e.
	and we conclude that:
	\begin{align*}
		\Psi_h(\la_Q)(u) =\int \widehat{\Delta_h^{\sigma_Q}}(u) \, d\mu = \int \left( \int \widehat{\Delta_h^{\sigma_\la}}(u)\, d\mu_\la \right) \, dQ(\lambda)
		= \int \Psi_h(\la)(u)\, dQ(\la). &\qedhere
	\end{align*}
\end{proof} 

We will be using Lemma~\ref{propQ} when the measure $Q$ is concentrated on vector measures $\la \in \Nd$ which are associated to a family of curves. More precisely, being $T>0$, we associate to any Lipschitz curve $\gamma\in \Lip_1([0, T];\Rd)$ the vector measure $\la_\gamma$ defined by
\begin{equation}\label{def:gala}
	\pairing {\la_\gamma, \f} := \int_0^T \pairing {\f(\gamma(t)), \dot \gamma(t)}\, dt 
\end{equation}
where the derivative $ \dot \gamma$ is Lebesgue a.e. defined in $[0,T]$. The length of the curve $L(\gamma):=\int_0^T |\dot \gamma (t)| dt $ is such that: $ \| \la_\gamma \| \le L(\gamma) \le T$ with possibly a strict inequality; an equality prevents $\gamma(t)$ to go back and forth on a part of $\gamma([0,T])$. Besides we note that the distributional equality $-\div \la_\gamma= \delta_{\gamma(T)} - \delta_{\gamma(0)}$ ensures that $\la_\gamma\in \Nd$. 

\medskip
Next we consider the map $\Lambda: \gamma\in \Lip_1([0,T];\Rd) \mapsto \la_\gamma \in \Md$. It is continuous when $\Lip_1([0,T];\Rd)$ is equipped with the uniform norm of $(C_0(\Rd))^d$ and $\Md$ with the weak-$*$ topology. However it is not injective due to the fact that the curve integral \eqref{def:gala} is invariant under reparametrization. Then we will consider measures $Q$ of the form $Q = \Lambda^\sharp (\nu)$, where $\nu$ is a suitable Borel measure on $\Lip_1([0,T];\Rd)$ (a closed convex subset of the Banach space $C_0([0,T];\Rd)$).

The possibility of decomposing a general $\la\in \Nd$ as $\la= \int \la_\gamma d\nu(\gamma)$, in which $\nu$ fulfils the mass preserving condition \eqref{Qcomplete} of Lemma \ref{propQ}, is given by the celebrated Smirnov decomposition theorem \cite[Thm C]{Smi93}. Here we use the precise formulation of \cite{Smi24}:

\begin{theorem}[Smirnov decomposition]\label{Smirnov}
	Let $T>0$ and $\la\in \Nd$ be given. Then, there exists a Borel measure $\nu\in \M_+(\Lip_1([0,T];\Rd))$ such that:
	\begin{enumerate}
		\item\label{Smirnov:inc} Inclusion: $ \|\la_\gamma\|= L(\gamma) $ and $\gamma([0,T])\subset \spt(\la)$ for $\nu$-almost every $\gamma$;
		\item\label{Smirnov:dec} Decomposition: $\la= \int \lambda_\gamma \, d\nu(\gamma)$, where the integral is intended in the weak sense $ \pairing{\la, \psi}= \int \pairing{\la_\gamma, \psi}\, d\nu(\gamma), \ \forall \psi\in C_0(\Rd)^d$;
		\item\label{Smirnov:mas} Mass preserving: $|\la| = \int |\lambda_\gamma| \, d\nu(\gamma)$ as elements of $\M_+(\Rd)$.
	\end{enumerate}
\end{theorem}

\medskip
The following result is true for the above-mentioned subfamily of curves:
\begin{lemma}\label{Phigala} 
	Let $\gamma\in \Lip_1([0, T];\Rd)$ be such that $ \| \la_\gamma \| = L(\gamma)$.
	Then we have
	\begin{equation}\label{Phigala=0}
		\Phi_h(\la_\gamma) \le 2\, L(\gamma)\quad,\quad \Phi_0^+(\la_\gamma)=\lim_{h\to 0} \Phi_h(\la_\gamma)=0.
	\end{equation}
	Furthermore, for every $u\in \Lip_1(\Rd)$, it holds
	\begin{equation}\label{Psigala=0}
		\Psi_h(\la_\gamma)(u) \le 2\, L(\gamma)\quad,\quad \lim_{h\to 0} \Psi_h(\la_\gamma)(u)= 0.
	\end{equation}
\end{lemma}

\begin{proof}
	After reparametrization, we may assume that $|\dot\gamma(t)|=1$ a.e. $t\in [0, L(\gamma)]$
	while $\gamma(t)=\gamma(T)$ for $t\ge L(\gamma)$. The condition $ \|\la_\gamma \| = L(\gamma)$ implies that 
	$|\la_\gamma|= \gamma^\sharp \L_{[0,L(\gamma)]}$ while $\la_\gamma = \sigma_\gamma |\la_\gamma|$ where $\sigma_\gamma (\gamma(t))= \dot\gamma(t)$. Accordingly, we have
	\begin{equation}\label{Philag}
		\Phi_h(\la_\gamma) = F_h(|\la_\gamma|, \sigma_\gamma)= \sup_{u\in \Lip_1\cap C^1} 
		\int_0^T \Delta_h^\gamma(u)(t) \, dt, 
	\end{equation}
	where 
	\begin{equation}\label{Delta-h-gamma}
		\Delta_h^\gamma(u)(t) := \frac{u(\gamma(t)+h \dot\gamma(t)) -u(\gamma(t))}{h}- \pairing{\nabla u(\gamma(t)), \dot\gamma(t)}.
	\end{equation}
	The first inequality in \eqref{Phigala=0} follows since $|\Delta_h^\gamma(u)(t)| \le 2$. To prove the convergence of $\Phi_h(\la_\gamma)$ to zero, we decompose
	\begin{equation*}\label{splitAh}
		\Delta_h^\gamma(u)(t) = A_h^\gamma(u)(t) + R_h^\gamma(u)(t) 
	\end{equation*}
	where we have set:
	\begin{align*} 
		A_h^\gamma(u)(t) &:= \frac{u(\gamma(t+h)) - u(\gamma(t))}{h} - \pairing{\nabla u(\gamma(t)),\dot\gamma(t)} \\
		R_h^\gamma(u)(t) &:=\frac{u (\gamma(t) +h \dot\gamma(t)) - u(\gamma(t+h))}{h} 
	\end{align*}
	Then, since $u$ is $1$-Lipschitz and recalling that $\gamma(t+h)=\gamma(T)$ whenever $t+h>T$, we get:
	$$ \int_0^T \Delta_h^\gamma(u)\, dt \le \int_0^T A_h^\gamma(u)\, dt + \int_0^T \left|\frac{\gamma(t+h)-\gamma(t)}{h} - \dot\gamma(t)\right|\, dt.$$
	Since $\gamma$ is Lipschitz and a.e. differentiable, the second integral vanishes by dominated convergence and we are reduced to show that $ \int_0^T A_h^\gamma(u)\, dt \to 0$. Note that this second integral does not depend on $u$, so that the estimate obtained below is uniform with respect to $u$. We notice that 
	\begin{align*} 
		\int_0^T \pairing{\nabla u(\gamma(t)), \dot\gamma(t)}\, dt &= u(\gamma(T)) - u(\gamma(0)) \ ,\\
		\int_0^T \frac{u(\gamma(t+h))-u(\gamma(t))}{h}\, dt&= u(\gamma(T)) - \frac1{h} \int_0^h u(\gamma(t))\, dt.
	\end{align*}
	Thus, since $\gamma$ is also $1$-Lipschitz, we get:
	$$ \int_0 ^T A_h^\gamma(u)(t) \, dt = u(\gamma(0)) - \frac1{h} \int_0^h u(\gamma(t))\, dt \le \frac1{h} \int_0^h t\, dt=\frac{h}{2},$$
	whence \eqref{Phigala=0} by sending $h\to 0$.
	
	\medskip Let us show \eqref{Psigala=0}, keeping the same parametrization for the curve $\gamma$. Then, for every $u\in \Lip_1(\Rd)$, the tangential gradient of $u$ with respect to $\mu_\gamma:= |\la_\gamma|$ satisfies, for a.e. $t\in [0, L(\gamma)]$,
	\begin{equation}\label{gradcurve}
		\pairing{\nabla_{\mu_\gamma} u (\gamma(t)), \dot\gamma(t)} \ = (u\circ \gamma)'(t).
	\end{equation}
	Indeed, at a.e. $t$ both $\gamma$ and $u\circ\gamma$ are differentiable and, $u$ being Lipschitz, $|u(\gamma(t+h))-u(\gamma(t)+h\dot\gamma(t))|=o(h)$; hence $(u\circ\gamma)'(t)=D_{\sig_\gamma}u(\gamma(t))$, where $D_{\sig_\gamma}u(y):=\lim_{h\to0}\frac{u(y+h\sig_\gamma(y))-u(y)}{h}$ depends on the point $y=\gamma(t)$ only and is defined $\mu_\gamma$-a.e.. On the other hand, for $\theta\in L^\infty_{\mu_\gamma}$, $\theta\sig_\gamma\in\T_{\mu_\gamma}$ by locality, and \eqref{byparts} gives $\int\theta\pairing{\sig_\gamma,\nabla_{\mu_\gamma}u}\,d\mu_\gamma = \pairing{-\div(\theta\la_\gamma),u}$, whose right-hand side equals $\int_0^T\theta(\gamma(t))\,(u\circ\gamma)'(t)\,dt=\int\theta\,D_{\sig_\gamma}u\,d\mu_\gamma$ for smooth $u$ by \eqref{def:gala}. 
	Since $\theta$ is arbitrary, $\pairing{\sig_\gamma,\nabla_{\mu_\gamma}u}=D_{\sig_\gamma}u$ $\mu_\gamma$-a.e., which is \eqref{gradcurve}. Accordingly \eqref{Delta-h-gamma} can be written as
	$$ \Delta_h^\gamma(u)(t) = \frac{u(\gamma(t) + h \dot\gamma(t)) -u(\gamma(t))}{h}- (u\circ \gamma)'(t).$$
	Next, according to \eqref{Delta-hsup}, we set: 
	$$\widehat{\Delta_h^\gamma}(u)(t):= \sup_{s\in(0,h]} |\Delta_s^\gamma(u)(t)|.$$
	Clearly $\widehat{\Delta_h^\gamma}(u) \le 2$ and, by the a.e. differentiability of the Lipschitz function
	$u\circ \gamma$, we have the pointwise convergence 
	$ \widehat{\Delta_h^\gamma}(u)(t) \to 0$ as $h\to 0$ holding a.e. 
	Accordingly we conclude by dominated convergence that 
	$\Psi_h(\la_\gamma)(u)=\int \widehat{\Delta_h^\gamma}(u)(t) \, dt\to 0$.	
\end{proof}

\subsection{Proof of Theorem \ref{MK-infinitesimal}}\label{proofmain}

\begin{proof}[Proof of the assertion \ref{MKi:form} of Theorem \ref{MK-infinitesimal}] 
	We proceed in three steps, successively showing that:
	\begin{enumerate}[label={(\arabic*)}]
		\item $\Phi_0^+(\la)=0$ when $\lambda\in \mathbf{F}_1(\Rd)$;
		\item $F_0^+( \mu,\sigma) \le 2 \int |\sigma_N|\, d\mu$ for any pair $(\mu,\sigma)\in \M_+ \times (L^1_\mu)^d$;
		\item $F_0^-( \mu,\sigma) \ge 2\int |\sigma_N|\, d\mu$.
	\end{enumerate}
	
	\med {\bf Step 1} \ We use Smirnov and Lebesgue's dominated convergence. Let $\lambda$ be an element of $\mathbf{F}_1(\Rd)$. Then we know from Proposition~\ref{fundamental} (assertion~\ref{fundamental:clos}) that there exists a sequence $\la_n \in \mathbf{N}_1(\Rd)$
	such that $\| \la_n-\la\| \to 0$. Using the Lipschitz property of the assertion~\ref{pFh:lip} of Lemma~\ref{pFh}, we have that $\Phi_0^+(\lambda) \le \Phi_0^+(\la_n)+ 3 \| \la_n-\la\|$. Thus we only need to show that $\Phi_0^+(\la)=0$ for every $\la\in \mathbf{N}_1(\Rd)$.
	Since the distributional divergence of such a vector measure $\la$ is a signed measure of bounded total variation, we can use the Smirnov decomposition (Theorem~\ref{Smirnov}). We obtain a decomposition of the form $\la= \int \la_\gamma \, d\nu(\gamma)$ where $\nu \in \M_+(\Lip_1([0,T];\Rd))$.
	Let $Q= \Lambda^\sharp (\nu)$ where the map $\Lambda: \gamma\in \Lip_1([0,T];\Rd) \mapsto \la_\gamma \in \Md$
	is continuous when $\Lip_1([0,T];\Rd)$ is equipped with the uniform norm of $(C_0(\Rd))^d$ and $\Md$ with the weak-$*$ topology.
	
	From the mass preserving condition \ref{Smirnov:mas} of Theorem \ref{Smirnov}, we infer that $Q$ fulfils the condition \eqref{Qcomplete} of Lemma \ref{propQ}, while $\la= \int \zeta\, dQ(\zeta)$ with $\int \|\zeta\|\, dQ(\zeta) = \int L(\gamma)\, d\nu(\gamma) <+\infty$.
	Rewriting the subadditivity property \eqref{subPhih} in terms of $\nu$ yields
	\begin{align*}
		\Phi_h (\la)\ \le \ \int \Phi_h(\la_\gamma) \, d\nu(\gamma).
	\end{align*}
	Since the equality $ \| \la_\gamma \| = L(\gamma)$ holds for $\nu$-a.e. $\gamma$, we can apply Lemma \ref{Phigala}. Then \eqref{Phigala=0} implies the pointwise convergence $\Phi_h(\la_\gamma) \to 0$\ $\nu$-a.e. whereas $\Phi_h(\la_\gamma) \le 2 L(\gamma)$. Recalling that $\int L(\gamma)\, d\nu(\gamma) <+\infty$, we conclude that $\Phi_h (\la) \to 0$ by using dominated convergence. Therefore $\Phi_0^+(\la)=0$ for every $\la\in \Fd$ and the proof of Step 1 is complete.

	\med {\bf Step 2} \ Let $\mu\in\M_+(\Rd)$ be given and denote again
	$$\T_\mu:= \{ \sigma\in (L^1_\mu)^d \, :\, \sigma(x)\in T_\mu(x) \ \mu\ \text{a.e}\}.$$
	We are going to show the inclusion
	\begin{equation}\label{TmuFmu}
		\T_\mu \subset \F_\mu:=\{ \sigma\in (L^1_{\mu})^d : F_0^+(\mu,\sigma)=0\}.
	\end{equation}
	Since $F_0^+(\mu,\cdot)$ is Lipschitz continuous (see Lemma \ref{pFh}.\ref{pFh:lip}), $\F_\mu$ is a closed subset of $(L^1_\mu)^d$.
	From Step 1, $\F_\mu$ contains every $\sig\in\T_\mu$ such that $|\sig| \equiv t > 0$ is constant. Indeed $\la := \sig\mu\in\Fd$ has polar decomposition $\la = \frac{\sig}{t}\,(t\mu)$, so that, by the positive homogeneity of Lemma~\ref{pFh}.\ref{pFh:posh} (in both arguments),
	$0 = \Phi_0^+(\la) = F_0^+(t\mu,\sig/t) = F_0^+(\mu,\sig)$.
	Next, let $\sig\in\T_\mu$ be such that $|\sig|$ takes finitely many values $t_1,\dots,t_m$, and set $B_j := \{|\sig| = t_j\}$. By the locality property \eqref{Tmu=}, $\sig\one_{B_j}\in\T_\mu$, hence $\sig\in\T_{\mu\res B_j}$ with $|\sig| \equiv t_j$ $\mu\res B_j$-a.e. (if $t_j=0$, then $F^+_0(\mu\res B_j,\sig)=0$ trivially). Using the subadditivity and the localization properties \ref{pFh:subadd} and \ref{pFh:resB} of Lemma~\ref{pFh}, we get
	\begin{align*}
		F_0^+(\mu,\sigma)
		\le \sum_j F_0^+(\mu,\sigma \, \one_{B_j})
		= \sum_j F_0^+(\mu\res B_j,\sigma) = 0.
	\end{align*}
	Finally, an arbitrary $\sig\in\T_\mu$ is the $L^1_\mu$-limit of such fields: for $n\in\NN$, let $\theta_n := \frac{1}{|\sig|}\min\big\{n, \tfrac{1}{n}\lfloor n|\sig|\rfloor\big\}\one_{\{\sig\neq0\}}$. Then $0\le\theta_n\le 1$, so that $\sig_n := \theta_n\sig\in\T_\mu$ by \eqref{Tmu=}, $|\sig_n|$ takes values in $\{0,\frac1n,\frac2n,\dots,n\}$, and $\sig_n\to\sig$ in $L^1_\mu$ by dominated convergence. Since $\F_\mu$ is closed, \eqref{TmuFmu} follows.

	Eventually let us consider a general vector field $\sigma\in (L^1_\mu)^d$ and its decomposition $\sigma= \sigma_T + \sigma_N$ into its tangential part $\sigma_T= P_\mu \sigma$ and its orthogonal part $\sigma_N = P_\mu^\perp \sigma$. 
	Then, recalling that $F_0^+(\mu,\cdot)$ is 2-Lipschitz, we deduce that $ F_0^+(\mu,\sigma)\le F_0^+(\mu,\sigma_T) + 2 \int |\sigma_N| \,d\mu $, whence 
	$F_0^+(\mu,\sigma)\ \le\ 2 \int |\sigma_N| \, d\mu \ ,$
	since we already proved that $F_0^+(\mu,\sigma_T)=0$.
	
	\med
	{\bf Step 3} \ We now show that $ F_0^-(\mu,\sigma) \ge 2 \int |\sigma_N| \,d\mu$.
	To that aim, we use the variational representation \eqref{Gsigma*} of the conjugate of the functional $G_\sigma$ (defined in \eqref{def:Gsigma})
	that we apply to the measure $\diffquot{\sig}{h} = \frac{\mu_h^\sigma- \mu}{h}$. Noticing that
	\begin{align*}
		F_h(\mu,\sigma)
		= \sup_{u \in C^1 \cap \Lip_1} \int \Delta_h^\sig(u)\, d\mu 
		\ \ge\ \sup_{u \in \D\cap\Lip_1} \left\{\pairing{\diffquot{\sig}{h}, u} - G_{\sigma}(u)\right\}\ =\ G_{\sigma}^{*}(\diffquot{\sig}{h}).
	\end{align*}
	We deduce that
	\begin{align} 
		\nonumber F_h(\mu,\sigma)\ &\ge\ \min \left\{\int_{\Rd} |\la-\sigma\, \mu | \ :\ -\div \la = \diffquot{\sig}{h} \right\} \\ 
		\label{FhD} & =\ \min \left\{\int |w-\sigma| \, d \mu + \int|\la^s |\ : \ (w, \la^s)\in \S_h \right\},
	\end{align}
	where $\S_h:= \set{ (w,\la^s)\in (L^1_\mu)^d\times \M^d} {-\div (w \mu+\la^s) = \diffquot{\sig}{h} \ ,\ \la^s \perp \mu}$. 
	Since the measure $\diffquot{\sig}{h}$ belongs to $\AEd$, the divergence condition in $\S_h$ implies that $w \mu+\la^s$ belongs to $\Fd$. By the locality property \eqref{Fdlocal}, we infer that
	$w \in \T_\mu$, while $\la^s$ is an element of $\Fd$ which satisfies $-\div \la^s= \diffquot{\sig}{h} + \div w\mu$. As such, \eqref{Beck} yields the lower bound
	$$\|\la^s\|_{\Md} \ \ge\ \Vert \diffquot{\sig}{h} + \div w\mu\Vert_{\KR}.$$
	Accordingly, for any given $u\in \Lip_1(\Rd) \cap \D(\Rd)$, we get the following inequalities:
	\begin{align*}
		\int |\la^s|
		&\ge \int \left(\frac{u(x+ h \sigma)-u(x)}{h} -\pairing{\nabla u,w}\right)\, d\mu \\
		&= \int \left(\frac{u(x+ h \sigma)-u(x)}{h} -\pairing{\nabla u,\sigma}\right)\, d\mu + \int \pairing{\nabla u,\sigma-w} \,d\mu\\
		& = \int \left( \int_0^1 \pairing {\nabla u(x+ s\, h\sigma(x))-\nabla u(x), \sigma(x)}\, ds\right) \, \mu(dx) + \int \pairing{\nabla u,\sigma-w} \,d\mu \\
		&\ge \int \pairing{\nabla u,\sigma-w} \,d\mu - \ \omega_u(h), \qquad \omega_u(h) := \int \min\Big\{2,\ \frac{h}{2}\,\Lip(\nabla u)\,|\sig|\Big\}\,|\sig|\,d\mu ,
	\end{align*}
	where we used that $|\nabla u(x+s\,h\,\sig)-\nabla u(x)|\le\min\{2, \Lip(\nabla u)\,sh|\sig|\}$. Note that $\omega_u(h)\to0$ as $h\to0$ by dominated convergence, since $\sig\in L^1_\mu$.
	Inserting the last inequality in \eqref{FhD}, we deduce that for every fixed $u\in \Lip_1\cap \D(\Rd)$, we have
	$$ F_h(\mu,\sigma) \ge \inf_{w\in \T_\mu} \int \left(|w-\sigma|+ \pairing{\nabla u,\sigma-w}\right) \, d \mu - \omega_u(h),$$
	where $\omega_u(h)$ does not depend on $w$. Thus by passing to the limit inf in $h\to 0$, we are led to the following lower bound:
	$$F_0^-(\mu,\sigma) \ge \inf_{w\in \T_\mu} \int \left(|w-\sigma|+ \pairing{\nabla u,\sigma-w}\right) \, d \mu.$$
	Since the constraint $w\in \T_\mu$ is local, namely $w(x)\in T_\mu(x)$ $\mu$ a.e., every competitor
	$w$ satisfies the pointwise inequality $|w(x)-\sig(x)|+\pairing{\nabla u(x),\sig(x)-w(x)} \ge
	m(\sig(x),\nabla u(x),T_\mu(x))$, where for every $(\sigma, z, T) \in \Rd\times\Rd \times \mathcal{G}_d$, we denote the infimum value function 
	$$ m(\sigma, z, T) := \inf_{w\in T} \{|w-\sigma|+ \pairing{z,\sigma-w} \}.$$
	Integrating and taking the infimum in $w$ we get
	$$ F_0^-(\mu,\sigma) \ge \int m(\sigma(x),\nabla u(x), T_\mu(x)) \, d\mu.$$
	
	Let us decompose $z= z_T + z_N$ where $z_T$, $z_N$ denote the orthogonal projections of $z$ on the subspace $T$ and its orthogonal $N=T^\perp$
	respectively. In the same way, we write $\sigma= \sigma_T+ \sigma_N$. It is convenient to rewrite the infimum problem in term of $\zeta:= w -\sigma_T$.
	After easy manipulations we obtain:
	$$ m(\sigma, z, T) = \inf_{\zeta\in T} \{ |\zeta-\sigma_N|- \pairing{z_T,\zeta}\}+ \pairing{z_N, \sigma_N}.$$
	Then, handling this finite dimensional convex program, we conclude that
	\begin{align*}
		m(\sigma, z, T) =
		\begin{cases} 
			|\sigma_N| \sqrt{1- |z_T|^2} + \pairing{z_N, \sigma_N} & \text{if } |z_T|\le 1, \\
			- \infty & \text{otherwise}. 
		\end{cases}
	\end{align*}
	We end up with the following integral inequality holding for every $u\in \Lip_1\cap \D(\Rd)$: 
	\begin{equation}\label{keylb}
		F_0^-(\mu,\sigma) \ge \int \left(|\sigma_N| \sqrt{1- |\nabla_\mu u|^2} + \pairing{\nabla u, \sigma_N}\right) \, d\mu. 
	\end{equation}
	To conclude this final step, we appeal now to \eqref{normal-approx} to get a sequence $(u_n)_n$ in
	$\Lip_1\cap \D(\Rd)$ such that 
	\begin{align*}
		u_n \weakstar 0,\qquad 
		\nabla u_n\to \frac{\sigma_N}{|\sigma_N|} \qquad
		\mu \text{ a.e. on } \{\sigma_N \neq 0\}.
	\end{align*}
	In particular, $\nabla_\mu u_n = P_\mu(\nabla u_n)$ converges to 0 $\mu$ a.e. on $\{\sigma_N \neq 0\}$.
	By applying \eqref{keylb} to $u=u_n$ and after passing to the limit by dominated convergence, we conclude that $ F_0^-(\mu,\sigma) \ge 2 \int |\sigma_N| \,d\mu$.
	The proof of the assertion \ref{MKi:form} of Theorem \ref{MK-infinitesimal} is now complete.
\end{proof}

\begin{proof}[Proof of the assertion \ref{MKi:diff} of Theorem \ref{MK-infinitesimal}]
	Let us fix $u\in\Lip(\Rd)$. Without loss of generality, we can assume that $\Lip(u)\le 1$. We begin by proving that $\Delta_h^\sig(u)\to 0$ in $L^1_\mu$ for every $\sig\in\T_\mu$ of the form $|\sig| = \one_B$, where $B$ is a Borel set of finite $\mu$-measure. Then, setting $\ov \la= \sig\, \mu$ and recalling \eqref{def:Dhu}, we are reduced to show that $\limsup_{h\to 0_+} \Psi_h(\ov \la) (u)= 0$. The condition $\sigma\in \T_\mu$ implies that
	$\ov \la \in \Fd$. Then, by the density of $\Nd$ in $\Fd$, we can assume that $\ov\la\in \Nd$. Indeed by \eqref{PsiLip}, we have for every $ \la\in \Nd$
	$$\limsup_{h\to 0_+} \Psi_h(\ov\la)(u) \le \limsup_{h\to 0_+} \Psi_h(\la)(u) + 3 \| \la- \ov\la\|.$$
	Thus, assuming now that $\ov \la\in \Nd$, we can apply Smirnov (Theorem~\ref{Smirnov}) in the same way as we did in the proof of the assertion \ref{MKi:form}: we get a decomposition of the form
	$$ \ov\la\ =\ \int \la_\gamma \, d\nu(\gamma)\ = \int \la \, dQ(\la) \ ,$$ 
	where $Q= \Lambda^\sharp(\nu)$ is a Borel measure concentrated on $\Nd$ which satisfies the conditions of the assertion \ref{propQ:Psi} of Lemma \ref{propQ}. In view of \eqref{subPsih},
	we deduce that:
	$$ \limsup_{h\to 0} \Psi_h(\ov\la)(u)\ = \limsup_{h\to 0} \int \Psi_h(\la_\gamma)(u) \, d\nu(\gamma).$$
	By \eqref{Psigala=0}, we know that $ \Psi_h(\la_\gamma)(u) \le 2 \, L(\gamma)$ while
	$ \Psi_h(\la_\gamma)(u) \to 0$. Therefore, recalling that $\int L(\gamma)\, d\nu(\gamma) <+\infty$, by dominated convergence, we conclude that
	$$ \Psi_h(\ov\la)(u)\ =\ \int \widehat{\Delta_h^\sig} (u) \, d\mu \ \to 0.$$
	By applying Fatou's Lemma, it follows that $\liminf_{h\to 0} \widehat{\Delta_h^\sig} (u) \to 0$ 
	$\mu$ a.e. Observing that the dependence in $h$ is monotone, we conclude that $\widehat{\Delta_h^\sig}(u)(x)$, 
	hence $\Delta_h^\sig(u)(x)$, do converge pointwise to $0$ at $\mu$ almost all $x\in \Rd$.
	It remains to remove the restriction on $|\sig|$. For $t>0$ one checks directly on \eqref{Delta-hsig*} that $\Delta_h^{t\sig}(u) = t \,\Delta_{h|t|}^{\sig}(u)$ with, so that, for every $\mu$-measurable $\theta \ge 0$ which is finite $\mu$ a.e.,
	$$\limsup_{h\to 0} |\Delta_h^{\theta\sig} (u)|(x)\ =\ \theta(x) \ \limsup_{h\to 0} \left|\Delta_h^{\sig} (u)(x)\right| ,$$
	whose right-hand side vanishes by the first part.
	Let now $\sig\in\T_\mu$ be arbitrary and put $B_k:=\{|\sig|>\frac1k\}$, so that $\mu(B_k)\le k\int|\sig|\,d\mu<+\infty$. The field $\sig_k := \frac{\sig}{|\sig|}\one_{B_k}$ is a bounded $\mu$-measurable selection of $T_\mu$ supported on a set of finite measure, hence belongs to $\T_\mu$ by \eqref{Tmu=}, and $|\sig_k| = \one_{B_k}$; moreover $\sig = \theta_k\,\sig_k$ on $B_k$ with $\theta_k := |\sig|$ finite $\mu$ a.e. Therefore $\Delta_h^\sig(u)\to0$ $\mu$ a.e.\ on $B_k$ for every $k$, hence $\mu$ a.e.\ on $\{\sig\neq0\}=\bigcup_k B_k$, and trivially on $\{\sig=0\}$. Finally the convergence also holds in $L^1_\mu$, by dominated convergence, since $|\Delta^\sig_h(u)|\le 2|\sig|\in L^1_\mu$.
	The assertion \ref{MKi:diff} of Theorem \ref{MK-infinitesimal} is proved. 
\end{proof}

\begin{proof}[Proof of the assertion \ref{MKi:full} of Theorem \ref{MK-infinitesimal}]
	We prove that differentiability along countably many tangent fields upgrades to differentiability along the whole subspace by an equicontinuity argument on the unit sphere. Since $\T_\mu$ is a closed subspace of the separable space $(L^1_\mu)^d$, pick a countable dense family $(\sig_k)\subset\T_\mu$; by the description of the essential union in Subsection~\ref{prel:multif}, $T_\mu(x)=\ov{\set{\sig_k(x)}{k\in\NN}}$ holds $\mu$ a.e. For $h>0$ and $v \in T_\mu(x)$ set
	\[
		g_h(v) \ :=\ \frac{u(x+h v)-u(x)}{h} - \pairing{\nabla_\mu u(x),v},
		\qquad\text{so that}\qquad g_h(t v) = t\,g_{t h}(v) \quad (t>0).
	\]
	What precedes says that $g_h(\sig_k(x))\to0$ as $h\to0^+$ for $\mu$ a.e.\ $x$. Discarding the countably many exceptional sets, fix $x$ at which this holds for every $k$ and $T_\mu(x)\neq\{0\}$, and let $S$ be the unit sphere of $T_\mu(x)$, which is compact. On $S$ the $g_h$ are Lipschitz with constant $2\Lip(u)$, independent of $h$, and by the homogeneity above they tend to $0$ on the dense subset $\set{\sig_k(x)/|\sig_k(x)|}{\sig_k(x)\neq0}$ of $S$. An equicontinuous family on a compact set which converges pointwise on a dense subset converges uniformly, so that $\sup_S|g_h|\to0$ as $h \to 0^+$. Finally, for every
	$v\in T_\mu(x)\setminus\{0\}$,
	\[
		\frac{u(x+v)-u(x)-\pairing{\nabla_\mu u(x),v}}{|v|}\ =\ g_{|v|}\Big(\frac{v}{|v|}\Big)
		\ \longrightarrow\ 0 \qquad \text{as } v\to 0,
	\]
	which is \eqref{eq:MKifull}.
\end{proof}

\section{Relations with other tangent bundles}\label{sec:otherbundles}

\subsection{About some variants}

Different tangent bundles to a measure $\mu$ can be constructed by choosing suitable subspaces $\W \subset (L^1_\mu(\Rd))^d$ to which one can associate the $\mu$-measurable multi-function $W_\mu(x):= \mu-\mathrm{ess} \; \bigcup \{\sigma ;\ \sigma\in \W \}$ detailed in Section~\ref{prel:multif}. By construction we have $\W \subset \set{ \sigma\in (L^1_\mu(\Rd))^d}{ \sigma(x)\in W_\mu(x) \ \text{$\mu$ a.e.}}$.
However considering $W_\mu(x)$ is merely useful when the latter inclusion is an equality. 
In virtue of Lemma \ref{local}, this is true if and only if $\W$ is closed and 
satisfies the stability condition \eqref{local}. This is the case when taking $\W$ to be $\T_\mu$ as in \eqref{def:Tmu}.
Other options are possible taking $\W$ to be one of the following subspaces:

\begin{enumerate}[left=0pt]
	\item $\S_\mu:= \left\{ \sigma \in (L^1_\mu(\Rd))^d : \sigma\, \mu \in \Nd \right\}$; 
	\item $\V_\mu:= \left\{ \frac{d \lambda}{d\mu} : \lambda \in \Nd\right\}$;
	\item $\T_\mu^{(p)}:= \left \{ \sigma \in (L^{p'}_\mu(\Rd))^d : \ -\div(\sigma\, \mu) \in L^{p'}_\mu(\Rd) \right\}$, with $1 < p < \infty$ and $p' = \frac{p}{p-1}$.
\end{enumerate}

Despite the fact that all theses spaces satisfy \eqref{local}, there are not closed subspaces of $(L^1_\mu(\Rd))^d$ as it is the case for $\T_\mu$.
However we may apply Proposition \ref{local} to their closure to obtain the equality 
$$\ov \W\ =\ \set{ \sigma\in (L^1_\mu(\Rd))^d}{ \sigma(x)\in W_\mu(x) \ \text{$\mu$ a.e.}}.$$
Accordingly we can consider the associated tangent bundles, in particular:
\begin{enumerate}[left=0pt]
	\item $S_\mu(x)$ (appearing in \cite{FraMant}) will be called the \emph{Smirnov} bundle; 
	\item $V_\mu(x)$ coincides with the \emph{Alberti-Marchese} decomposability bundle, this as a consequence of \cite[Corollary 6.5]{AM};
	\item $T_\mu^{(p)}$ is, for $p\in (1,\infty)$, the tangent bundle introduced in \cite{BBS}, with the motivation of defining the Sobolev space
	$H^p_\mu(\Rd)$.
\end{enumerate}

These bundles as related as follows.
\begin{proposition}\label{AM=MK} 
	Let $\mu\in \M_+(\Rd)$ be a Radon measure. All statements below are intended for $\mu$ almost $x\in \Rd$.
	\begin{enumerate}
		\item\label{item:VT} $V_\mu(x)= T_\mu(x)$;
		\item\label{ST} $S_\mu (x) \subset T_\mu(x)$ with possibly a strict inclusion;
		\item\label{Tp} If $\mu$ is finite with compact support, then $T_\mu^{(p)}(x)\subset S_\mu(x)$ while $T_\mu^{(p)}(x)$ is non decreasing with respect to $p\in (1,+\infty)$.
	\end{enumerate}
\end{proposition} 

\begin{proof} 
	By Lemma \ref{local}, the equality \ref{item:VT} is true if we show the following equality between subspaces of $(L^1_\mu(\Rd))^d$:
	\begin{equation}\label{Vmu=Tmu}
		\ov {\V_\mu} = \T_\mu \ .
	\end{equation}
	Let $\tau \in \V_\mu$, i.e. there exists $\lambda \in \M^d$ such that $\div \lambda \in \Mod$ and $\frac{d\lambda}{d\mu} =\tau$.
	Then we write $\lambda$ in the form $\lambda= \chi \, m$ where
	$$ m= \mu + | \lambda^s|\quad, \quad \chi = \begin{cases} \tau & \mu-{\rm a.e.} \\
		\frac {d\lambda^s}{d|\lambda^s|} & \lambda^s-{\rm a.e.} \end{cases},$$ 
	where $\lambda^s$ denotes the singular part of $\lambda$ with respect to $\mu$. 
	Since $\div(\chi \, m)$ belongs to $\Mod \subset \AEd$, we have $\chi(x)\in T_m(x)\ m$\ a.e..
	By applying \eqref{disjunction} to the decomposition $m = \mu +|\lambda^s|$, we infer that $T_m(x)= T_\mu(x)$ holds 
	$\mu$ a.e.. Combined with the equality $\chi=\tau$ holding $\mu$ a.e., we conclude that $\tau\in \T_\mu$, hence 
	$\ov{\V_\mu} \subset \T_\mu$ since the latter set is closed. 
	
	Conversely, let $\tau\in \T_\mu$. As $\lambda := \tau\, \mu$ belongs to $\Fd$, by the assertion \ref{fundamental:clos} of Proposition~\ref{fundamental}, there exists a sequence $(\lambda_n)$ in $\Nd$ such that $\Vert \lambda_n- \lambda\Vert_{\M^d}\to 0$.
	This means that the Lebesgue-Nikodym decomposition $\lambda_n= \tau_n \, \mu + \lambda_n^s$ satisfies
	\begin{equation}\label{approxN}
		\tau_n \to \tau \quad \text{in } L^1_\mu \quad, \quad \Vert \lambda_n^s \Vert_{\M^d}\to 0 .
	\end{equation}
	In other words, every tangent field is an $L^1_\mu$-limit of densities of normal $1$-currents. 
	By the definition of $\V_\mu$, we have $\tau_n \in \V_\mu$ whence $\tau \in \overline{\V_\mu}$ and $T_\mu \subset \overline{V_\mu}$.
	The desired equality \eqref{Vmu=Tmu} follows and the assertion \ref{item:VT} is established.
	
	The assertion \ref{ST} is straightforward from the inclusion $\S_\mu \subset \T_\mu$. The strict inclusion is shown in the example \ref{fat} below.
	
	For the assertion \ref{Tp}, note first that $\T_\mu^{(p)}\subset \T_\mu^{(q)}$ whenever $\mu(\Rd) < +\infty$ and $1< p \le q <\infty$, since $L^{p'}_\mu\subset L^{q'}_\mu$; this gives the monotonicity in $p$. Next, let $\sig\in\T_\mu^{(p)}$, so that $-\div(\sig\mu) = g\,\mu$ with $g\in L^{p'}_\mu\subset L^1_\mu$. Testing against $\f_n := \f(\cdot/n)$, where $\f\in\D(\Rd)$ equals $1$ on $B_1$, and letting $n\to\infty$ yields $\int g\,d\mu = \lim_n\int\pairing{\sig,\nabla\f_n}\,d\mu = 0$ by dominated convergence. Hence $g\mu\in\Mod$ and $\sig\mu\in\Nd$, i.e. $\sig\in\S_\mu$. This shows $\T^{(p)}_\mu\subset\S_\mu$, whence $T^{(p)}_\mu(x)\subset S_\mu(x)$ $\mu$-a.e. by Proposition~\ref{local}.
\end{proof}

\begin{remark}\label{nonlocal} 
	Amongst all tangent bundles appearing in Proposition~\ref{AM=MK}, only $T_\mu(x) = V_{\mu}(x)$ satisfies the crucial locality property \eqref{locality}, as checked in the 1d example \ref{fat} below (see also \cite{ARGB,FraMant}). This observation makes it suitable for comparison with the Preiss tangent space and ultimately for deriving rectifiability properties. 
\end{remark}

\begin{example}\label{fat} 
	Take $\mu$ to be the Lebesgue measure on $\R$ and $\nu= \mu\res C$ where $C$ is a fat compact Cantor subset ($\mu(C) >0$ and $C$ has empty interior). Then, as the property \eqref{locality} is true for $T_\mu$, $T_\nu(x)=\R$ holds $\nu$ a.e.. Take now $\sigma\in L^1_\mu$ such that $\div(\sigma \nu) = (\sigma \, \one_C)'$ is a finite measure. Then $u = \sigma \, \one_C$ can be identified to a compactly supported function in $BV(\R)$. Since $u$ vanishes on the complement of $C$, an open dense subset of $\R$, we deduce that $u=0$, hence $\sigma=0$ a.e. on $C$. Therefore, according to the definition of the Smirnov bundle $\S_\nu$ given above, $S_\nu(x)=\{0\}$ holds $\nu$ a.e.. The same degeneracy occurs for the bundles $T_\nu^{(p)}$ which are contained in $S_\nu$.
\end{example}

\subsection{The differentiability bundle}\label{nodiffsec}

Here we shortly explain why $V_{\mu}$ introduced before is also a differentiability bundle. Let us say that a function $u:\Rd\to\R$ is \emph{differentiable at $x$ with respect to} a linear subspace $W\subset\Rd$ if, for some linear form $L$ on $W$,
\begin{equation}\label{Wdiff}
	u(x+v)-u(x) \ =\ L(v) + o(|v|) \qquad \text{as } W\ni v\to 0 .
\end{equation}
In the remarkable paper \cite{AM}, Alberti and Marchese proved that every Lipschitz function is differentiable $\mu$ a.e. with respect to $V_\mu(x)$ and, in addition, that $V_\mu$ is the largest such bundle. This means that any bundle $W(x)$ such that \eqref{Wdiff} holds $\mu$-a.e. for every Lipschitz $u$ satisfies the inclusion $W(x) \subset V_{\mu}(x)$ for $\mu$-a.e. $x$. We recover the differentiability along $T_\mu$ in assertion \ref{MKi:full} of Theorem~\ref{MK-infinitesimal}, but its maximality so far only follows through the deep theory of \cite{AM} and Proposition~\ref{AM=MK}. In order to produce a self-contained proof of this maximality, we construct along any prescribed direction field which is not tangent to $T_\mu$, a Lipschitz function which is not differentiable on a set of arbitrarily large $\mu$-mass. The construction is short and uses only the approximation property \eqref{normal-approx} of $\N_\mu$; the result is weaker than \cite[Prop.~4.4 and 4.5]{AM}, which shows that the set of such functions is large in the sense of Baire.

\begin{proposition}\label{nodiff}
	Let $\mu \in \M_+(\Rd)$ be a finite measure and let $w\in(L^\infty_\mu)^d$ be a unit vector field such that $w(x)\notin T_\mu(x)$ for $\mu$-a.e. $x$. Then, for every $\eps>0$, there exist a compact set $K$ with $\mu(\Rd\setminus K)\le\eps$ and a Lipschitz function $f:\Rd\to\R$ such that the directional derivative $\partial_{w(x)}f(x)$ fails to exist at $\mu$ a.e. $x\in K$.
\end{proposition}
\begin{proof}
	We write $w_T:=P_\mu w$ and $w_N:=P_\mu^\perp w$ for the tangential and normal parts of $w$ and define 
	\begin{enumerate}
		\item the tangential displacement $t_h \coloneqq \mathrm{id}+hw_T$,
		\item the full displacement $z_h \coloneqq \mathrm{id}+hw = t_h + hw_N$,
		\item the normal increment
		\[
			N_hu:=\frac{u(z_h) -u(t_h)}{h}, \qquad\text{so that}\qquad |N_hu| \le \Lip(u)\,|w_N| .
		\]
	\end{enumerate}
	
	\medskip\noindent\textbf{Reduction to the normal fields.} Since $\mu$ is finite and $w\in(L^\infty_\mu)^d$, the field $w_T$ belongs to $\T_\mu$ by \eqref{Tmu=}. Splitting
	the difference quotient of $u$ along $w$ as $\Delta^w_h u  - N_hu + \Delta^{w_T}_{h} u$, we get from \ref{MKi:diff} of
	Theorem~\ref{MK-infinitesimal} that for every $u\in\Lip(\Rd)$, one gets $ \Delta^w_h  u  - N_hu  \to \pairing{\nabla_\mu u,w_T}$ in $L^1_\mu$. 
	We conclude that, at $\mu$ a.e.\ $x$,
	\begin{equation}\label{reduction}
	  \partial_{w(x)}u(x)\ \text{exists}\qquad\Longleftrightarrow\qquad \lim_{h\to0^+}N_hu(x)\ \text{exists}.
	\end{equation}
	It is therefore enough to construct $u\in\Lip(\Rd)$ whose normal increment oscillates on a compact
	set $K$ as in the statement.	
	
	\medskip\noindent\textbf{Set-up.}
	By assumption $w_N \neq 0$ $\mu$ a.e., so the unit normal field $\sigma := w_N/|w_N|$ is well defined. Fix $\delta \in (0,1]$ such that $\mu(\{|w_N|<\delta\})<\eps/2$, and introduce the smallness constant
	\[
		\eps_0\ :=\ \frac{\delta}{8}\ \le\ \tfrac18 .
	\]
	By Proposition~\ref{Nmu} there is a sequence $(u_k)_k$ of smooth $1$-Lipschitz functions converging weakly-$*$ to zero in $\Lip(\Rd)$, hence uniformly to zero on compact sets, and such that $\nabla u_k\to\sigma$ $\mu$ a.e.. By Egorov's and Lusin's theorems together with inner regularity there is a compact set $K\subset\{|w_N|\ge\delta\}$ with $\mu(\Rd\setminus K)\le\eps$ on which $\sigma$ is continuous and $\nabla u_k\to\sigma$ uniformly. By Tietze's theorem, followed by a projection onto the unit ball, we extend $\sigma$ to a continuous field on $\Rd$ (still denoted
	$\sigma$) with $|\sigma|\le1$. Writing $K^s:=\{\dist(\cdot,K)\le s\}$, we let $\omega$ be a modulus of
	continuity for $\sigma$ on the compact set $K^1$. Throughout, $\|\cdot\|_\infty$ denotes the supremum
	norm on $K^1$, so that $\|u_k\|_\infty\to0$.
	
	\medskip\noindent\textbf{Construction of a candidate non-differentiable function.}
	Set $s_0=h_0=1$ and define $s_r,h_r>0$ inductively: given $s_{r-1},h_{r-1}$, choose successively
	\begin{enumerate}[label={(\alph*)}]
	 \item\label{nd:v} a function $v_r:=u_k$ with $k$ so large that
	\[
		\|v_r\|_\infty\le 2^{-r}\min\{s_{r-1},h_{r-1}\}
		\qquad \text{and} \qquad
		\sup_K|\nabla v_r-\sigma| \le\eps_0 2^{-(r+1)}.
	\]
	This makes $v_r$ small and $\nabla v_r$ close to $\sigma$ on $K$, with an error decreasing geometrically.
	\item\label{nd:s} A scale $s_r\le s_{r-1}/4$ small enough so that the previous estimate still holds on $s_r$-enlargement of $K$ up to a factor 2, i.e.
	\[
		|\nabla v_r-\sigma|\le\eps_0 2^{-r}\quad\text{on } K^{s_r}.
	\]
	This is made possible by continuity of $\nabla v_r-\sigma$ and compactness of $K$.
	 \item\label{nd:h} A length scale $h_r$, at which the difference quotients will be read, such that 
	\[
		h_r\ \le\ \min \left\{ h_{r-1}\ , \ \frac{s_r}{{\sqrt{2}}} \right\}.
	\]
	\end{enumerate}
	We localise each $v_r$ at its own scale. Let $\chi_r$ be a smooth cut-off of $K^{s_{r-1}/2}$ supported in $K^{s_{r-1}}$, whose transition layer costs a gradient $\|\nabla\chi_r\|_\infty\le3/s_{r-1}$. By \ref{nd:v}, the amplitude of $v_r$ compensates this cost geometrically:
	\begin{equation}\label{nd:comm}
		 \|v_r\,\nabla\chi_r\|_\infty\ \le\ 3\cdot2^{-r}.
	\end{equation}
	Define the candidate $u$ and its associated partial sums
	\[
		u \coloneqq \sum_{r\ge1} \ (-1)^{r+1}\,\chi_r v_r, \qquad u_N:=\sum_{r\le N}(-1)^{r+1}\chi_rv_r. 
	\]
	Using \ref{nd:v} once more, the series above converge uniformly, and $u_N$ is $C^1$ for each $N$. 
	
	\medskip\noindent\textbf{Lipschitzianity.}
	We first estimate the gradient of $u_N$. To derive the bound on $K$, let $x\in K$. In particular $\chi_r\equiv1$ near $x$ for every $r \le N$, so that
	\[
		\nabla u_N(x)= \sigma(x) \mathbf 1_{\{N \; \text{odd}\}} + \sum_{r\le N}(-1)^{r+1} (\nabla v_r(x) - \sigma(x)) .
	\]
	Since $K\subset K^{s_r}$ for every $r$, property \ref{nd:s} gives $|\nabla u_N(x)|\le|\sigma(x)|+\eps_0 \sum_{r = 1}^{N} 2^{-r} \le 2$.
	Let us now compute the gradient outside of $K$. For this, let us choose $x \in \Rd \setminus K$ with $\dist(x,K) \le 1$. Let $M \in \mathbb{N}$ be such that $x \in K^{s_{M-1}} \setminus K^{s_{M}}$. In particular, we have $\chi_r\equiv1$ near $x$ for $r\le M-1$ and $\chi_r\equiv0$ near $x$ for $r\ge M+1$. Thus, near $x$ we get
	\begin{equation}\label{nd:loc}
	  u_N \ = \  (-1)^{M+1}\chi_{M} v_{M}\,\mathbf 1_{\{N\ge M\}}  \ + \ \sum_{r\, \le \, \min\{M-1, N\}}(-1)^{r+1}v_r.
	\end{equation}
	As $x\in K^{s_r}$ for $r \leqslant M-1$, the choice \ref{nd:s} yields
	\[
		\big|\sum_{r\le M-1}(-1)^{r+1}\nabla v_r(x)\big|\le|\sigma(x)|+\eps_0\le2.
	\]
	Adding $|\chi_M\nabla v_M(x)|\le1$ and \eqref{nd:comm}, we obtain $|\nabla u_N(x)|\le5$. Hence $\Lip(u_N)\le5$ for every $N$, and since $u_N\to u$ uniformly we conclude $\Lip(u)\le 5$. Finally, \eqref{nd:loc} with $N=\infty$ expresses $u$ near any $x\notin K$ as a finite sum of $C^1$
	functions, so that $u$ is of class $C^1$ on $\Rd\setminus K$.
	
	\medskip\noindent\textbf{Oscillation of the normal slope on $K$.}
	Fix $x\in K$ and $\ell\ge1$. Parametrizing the segment $[t_{{h_\ell}}(x),z_{{h_\ell}}(x)]$ by $p_{\tau} := t_{{h_\ell}}(x) + \tau\,{h_\ell}\,w_N(x)$ for $\tau \in [0,1]$, property \ref{nd:h} gives
	\begin{align*}
		|p_{\tau} - x|\ \le\ {h_\ell}(|w_T(x)|+|w_N(x)|)\ \le\ \sqrt{2} \, h_\ell\ \le\ s_\ell.
	\end{align*}
	Consequently, for every $r\le\ell$ the whole segment lies in $K^{s_r} \subset \{\chi_r\equiv1\}$, so that $N_h(\chi_rv_r)=N_hv_r$ there. By the fundamental theorem of calculus 
	\[
		\sum_{r\le \ell}(-1)^{r+1}N_{h}v_r(x)\;=\;\int_0^1\pairing{G_{\ell}(p_{\tau}),\,w_N(x)}\,d\tau,
		\qquad G_{\ell}:=\sum_{r\le \ell}(-1)^{r+1}\nabla v_r .
	\]
	Once more, by \ref{nd:s} we have $|G_\ell(p_{\tau}) -\mathbf 1_{\{\ell\ \mathrm{odd}\}}\sigma(p_{\tau})|\le\eps_0$, while $|\sigma(p_{\tau})-\sigma(x)|\le \omega(\sqrt{2}\, h_\ell)$ and $\pairing{\sigma(x),w_N(x)}=|w_N(x)|$ by the definition of $\sigma$. Therefore, the integral differs from $\mathbf 1_{\{\ell\ \mathrm{odd}\}}|w_N(x)|$ by at most $[\eps_0+\omega(\sqrt{2}\, h_\ell)]|w_N(x)|$. For the remaining indices $r>\ell$ we use $\|v_r\|_\infty \le 2^{-r} h_{r-1} \le 2^{-r} h_\ell$, which gives the bound $|N_{h_\ell}(\chi_rv_r)|\le2\|v_r\|_\infty/h_\ell\le2^{-r+1}$. Altogether, recalling
	$|w_N|\le 1$ and $\eps_0=\delta/8$,
	\[
	  |\,N_{h_\ell}u(x)-\mathbf 1_{\{\ell\ \mathrm{odd}\}}\,|w_N(x)|\,|
	  \ \le\ \frac{\delta}{8}+ \,\omega({\sqrt{2}}\, h_\ell)+2^{-\ell+1}
	  \ \le\ \frac{\delta}{4}
	  \qquad\text{for all }x\in K,
	\]
	the last inequality holding as soon as $\ell$ is large enough, since $h_\ell\to0$. As $|w_N|\ge\delta$ on $K$, we conclude that for every such $\ell$ and every $x\in K$
	\[
	N_{h_\ell} u(x)\ \begin{cases}
	\ \ge\ \tfrac34\delta & \text{if $\ell$ is odd},\\[2pt]
	\ \le\ \tfrac14\delta & \text{if $\ell$ is even}.
	\end{cases}
	\]
	Thus $\lim_{h\to0^+}N_hu(x)$ exists at no point of $K$, and by \eqref{reduction} the directional derivative $\partial_{w(x)}u(x)$ fails to exist $\mu$ a.e. $x\in K$.
\end{proof}

\begin{theorem}[$T_\mu$ is the largest differentiability bundle]\label{maximal}
	Let $\mu\in\M_+(\Rd)$ be a Radon measure and let $W:\Rd\to\Grass$ be a $\mu$-measurable map such that every Lipschitz function $u:\Rd\to\R$ is differentiable at $\mu$-a.e. $x$ with respect to $W(x)$. Then $W(x)\subset T_\mu(x)$ for $\mu$-a.e. $x$. 
\end{theorem}

\begin{proof}
	Let $(e_i)_i$ be the canonical basis of $\Rd$ and $P_W(x)$ the orthogonal projection onto $W(x)$. Since $W(x)$ is spanned by the vectors $P_W(x)e_i$, the set $\{x : W(x)\not\subset T_\mu(x)\}$ is the union of the $\mu$-measurable sets $E_i:=\{x : P_W(x)e_i\notin T_\mu(x)\}$. Assume by contradiction that $\mu(E_i)>0$ for some $i$ and pick a bounded Borel set $E\subset E_i$ with $\mu(E)>0$. On $E$, the unit field $w:=P_We_i/|P_We_i|$ satisfies $w(x)\in W(x)\setminus T_\mu(x)$, and $T_{\mu\res E}=T_\mu$ $\mu$-a.e. on $E$ by \eqref{restriction}. Proposition~\ref{nodiff}, applied to the finite measure $\mu\res E$ and to $w$, provides a Lipschitz function $f$ and a compact set $K$ with $\mu(K\cap E)\ge\mu(E)-\eps>0$ such that $D_{w(x)}f(x)$ does not exist for $x\in K$. Since $w(x)\in W(x)$ on $E$, $f$ is not differentiable with respect to $W(x)$ at any $x\in K\cap E$, a contradiction.
\end{proof}

In view of assertion \ref{MKi:full} of Theorem~\ref{MK-infinitesimal}, $T_\mu$ is therefore the largest bundle along which every Lipschitz function is $\mu$-a.e. differentiable.

\subsection{Relation with Preiss tangent measures}

In the literature, blow-up techniques were initially developed to define tangent cones for rectifiable structures such as currents and varifolds (see, e.g., H. Federer \cite{Fed}, W. Allard \cite{Allard72}, and the recollection by L. Simon \cite{Simon}). The notion of tangent measures for a general Radon measure was introduced by D. Preiss \cite{Preiss87} (see also \cite{FraMant} for a comparison between Preiss' one and the Smirnov bundle):

\begin{definition}\label{tangent-preiss}
	Let $\mu \in \M_+(\Rd)$ be a Radon measure and $x_0\in\Rd$. The Preiss tangent space to $\mu$ at $x_0$, denoted $\text{Tan}(\mu,x_0)$, consists of all non zero Radon measures $\nu \in \M_+(\Rd)$ such that, for suitable sequences of reals $c_j >0$ and $r_j \to 0$, we have 
	\begin{equation}\label{def:tangentPreiss}
		\lim_{j\to \infty} c_j \int \f\left(\frac{x-x_0}{r_j}\right)\,d\mu = \int_{\Rd} \f(y) d\nu(y) 
		\quad 
		\text{ for every } \f \in C^0_c(\Rd).
	\end{equation} 
\end{definition}

\begin{remark}\label{ball-control} 
	By testing the convergence \eqref{def:tangentPreiss} against a continuous function $\f\in C^0(\Rd; [0,1])$ such that $\f=1$ on $B_R$ and $\spt(\f) \subset B_{R+\delta}$, we infer that the reals $c_j, r_j$ involved (they depend on $x_0$) are such that, for every $R>0$ and $\delta>0$:
	\begin{equation*}
		\nu(B_R) 
		\le \liminf_{j\to\infty}\, c_j \, \mu \left( B_{r_j \, R}\right) 
		\le \ \limsup_{j\to\infty}\, c_j \, \mu \left( B_{r_j \, R}\right) 
		\le \nu(B_{R+\delta})
		<+\infty.
	\end{equation*}
	Therefore, by sending $\delta\to 0$, it follows that for every $R$ in the complement of a countable subset of $(0,+\infty)$, there exists a constant $\kappa(x_0,R) >0$ such that 
	\begin{equation}\label{ball-esti}
		c_j \, \mu \left( B_{r_j \, R}\right) \ \to \kappa(x_0,R) \quad \text{as } j\to\infty.
	\end{equation}
	Note that the situation is much simpler in the case where $\mu$ satisfies the doubling condition around $x_0$, namely:
	$$ \limsup_{s\to 0_+} \frac{ \mu(B(x_0,2 s))}{ \mu(B(x_0,s))} < +\infty.$$ 
	In this case, it is not restrictive to impose that $c_j \mu(B_{r_j})= 1$.
\end{remark}

The main feature of $\text{Tan}(\mu,x_0)$ is that it is non empty positive cone of $\M_+(\Rd)$ $\mu$-a.e., closed under vague convergence and dilation invariant (see \cite[Thm 2.5]{Preiss87}). For those properties, no regularity assumptions on $\mu$ are required. In contrast, when some rectifiability property of $\mu$ is expected, the more restrictive notion of tangent space to $\mu$ can be used (see \cite{Simon, Mattila}).

\begin{definition}\label{tangentplane}
	The $k$-dimensional subspace $V$ of $\Rd$ is said to be tangent to $\mu$ at $x_0$ if there exists a constant $\theta>0$ (called the \emph{multiplicity}) such that
	\begin{equation}\label{def:tanplan}
		\lim_{h\to 0_+} \frac1{h^k}\int\f\left(\frac{x-x_0}{h}\right)\, d\mu = \theta\, \int_{V} \f(y) d\H^k (y) 
		\quad \text{ for every } \f\in C^0_c(\Rd).
	\end{equation} 
\end{definition}

\begin{remark}\label{localV} 
	The property \eqref{def:tanplan} implies that \eqref{def:tangentPreiss} holds with $\nu = \theta \, \H^k\res V$ being $c_j= r_j^{-k}$
	and $r_j$ any vanishing sequence. This means that the possibly large set $\text{Tan}(\mu,x_0)$ contains at least one flat measure (in general not unique). 
	As originally established by Preiss \cite{Preiss87} (see \cite[Thm 14.12]{Mattila} for an accessible exposition), if $\text{Tan}(\mu,x_0)$ consists only of flat measures, then these measures are concentrated on the same subspace $V$. Thus, if $\text{dim} \, V=k$, one has $\text{Tan}(\mu,x_0)= \{ \theta \, \H^k\res V\ :\ \theta \in \R_+\}$.
\end{remark}

A revealing connection between ${\rm Tan}(\mu,x_0)$ defined in \eqref{def:tangentPreiss} and $T_\mu(x_0)$ is summarized in the following Proposition which extends \cite{FraMant} where the Smirnov tangent bundle $S_\mu(x_0)$ is considered (see also \cite{Adolfo2020}); recall that, by Proposition~\ref{AM=MK}, $S_\mu(x_0)$ can be strictly contained in the MK-bundle $T_\mu(x_0) $. Beforehand we give a related result of convergence, whose proof relies on the density of normal currents in flat chains (assertion \ref{fundamental:clos} of Proposition~\ref{fundamental}).

\begin{lemma}\label{Conv-AEd}
	Let $\mu \in \M_+(\Rd)$ be a Radon measure and $\sig$ an element of $\T_\mu$. Then, for $\mu$ a.e. $x_0\in\Rd$, we have the following convergence as $h\to 0_+$:
	\begin{equation}\label{blow-div}
		f_h^{x_0} := \frac{h}{\mu(B(x_0,h))}\ T_h^{x_0}( -\div (\sig\, \mu)) \ \to 0 \quad \text{in $\D'(B_1)$} 
	\end{equation}
	where $T_h^{x_0}: x \mapsto \frac{x-x_0}{h}$.
\end{lemma}

\begin{proof}
	Since $\sig\, \mu$ belongs to $\Fd$, there exists $\la_n= w_n\, \mu + \la_n^s\in \Nd$ such that
	$$ \int |\sig-w_n|\, d\mu + \int |\la_n^s| \le \frac1{n} \quad,\quad |\la_n^s| \perp \mu.$$
	Up to extracting a subsequence, we can assume that $ w_n\to \sig$ $\mu$ a.e. as $n\to \infty$.
	Then setting $f_h^{x_0} = g_{h,n}^{x_0} + r_{h,n}^{x_0}$ where $g_{h,n}^{x_0} := \frac{h}{\mu(B(x_0,h))}\ T_h^{x_0}( -\div \la_n)$, we have
	\begin{align*}
		\pairing {r_{h,n}^{x_0}, \f} 
		&= \fint_{B(x_0,h)} \pairing{\sig-w_n, \nabla \f\left(\frac{x-x_0}{h}\right)}\, d\mu
		+ \frac{1}{\mu(B(x_0,h))} \pairing{\la_n^s,\nabla \f\left(\frac{x-x_0}{h}\right)} \\
		& \le \ \Lip(\f) \, \left(\fint_{B(x_0,h)}|\sig-w_n| \, d\mu + \,\frac{ |\la_n^s| (B(x_0,h))}{\mu(B(x_0,h))} \right)
	\end{align*} 
	for every $\f\in\D(B_1)$. On the other hand, since $\div \la_n$ is a measure:
	\begin{align*}
		\pairing {g_{h,n}^{x_0}, \f} &\le h\, \sup |\f| \, \frac{|\div \la_n|(B(x_0,h))}{\mu(B(x_0,h))}. 
	\end{align*} 
	Next we choose a set $B \in \B(\Rd)$ with $\mu(B)=0$ and such that, for every $x_0\in B^c$ and $n\in \NN^*$:
	\begin{gather*}
		\limsup_{h\to 0} \frac{|\div \la_n|(B(x_0,h))}{\mu(B(x_0,h))} < +\infty, 
		\qquad\quad
		\limsup_{h\to 0} \frac{|\la_n^s|(B(x_0,h))}{\mu(B(x_0,h))} = 0, \\ 
		\text{and} \qquad
		\lim_{h\to 0} \fint_{B(x_0,h)}|\sig-w_n| \, d\mu = |\sig(x_0)- w_n(x_0)|.
	\end{gather*} 
	Then gathering all, we infer that, for every $n$ and $x_0\in B^c$, it holds
	\begin{align*}
		\limsup_{h\to 0} \left|\pairing{f_h^{x_0}, \f}\right| \ \le\ \Lip(\f) \, |\sig-w_n|(x_0) \quad \forall \f\in \D(B_1).
	\end{align*}
	The convergence \eqref{blow-div} follows by the pointwise convergence $w_n\to \sig$ as $n\to\infty$.
\end{proof}

\begin{proposition}\label{MK-Preiss} 
	Let $\mu\in \M_+(\Rd)$ be a Radon measure. Then for $\mu$ a.e. $x_0\in \Rd $, any measure $\nu \in \text{Tan}(\mu,x_0)$ satisfies the equality 
	\begin{equation}\label{slice}
		\nu \ \ = \H^{k}\res T_\mu(x_0) \otimes \eta \qquad \text{in } \Rd,
	\end{equation}
	where $k=k(x_0)$ denotes the dimension of the $T_\mu(x_0)$ and $\eta$ is a Radon measure on $T_\mu^\perp(x_0)$. In particular, if $\mu$ admits a $k$-dimensional tangent space $V$ at $x_0$ in the sense of~\eqref{def:tanplan}, then 
	\begin{equation}\label{P<T}
		T_\mu(x_0) \ = \ V,
	\end{equation}
while in general~\eqref{slice} only yields $T_\mu(x_0) \subset V$.
\end{proposition}

\begin{proof}
	Replacing $\mu$ by $\mu\res B_m$ does not change $T_\mu$ nor $\text{Tan}(\mu,x_0)$ for $x_0 \in B_m$. Hence, we may assume that $\mu$ is finite. Let $P_\mu^1,\dots,P_\mu^d \in \mathcal T_\mu$ be the rows of the orthogonal projection $P_\mu$. We fix a $\mu$-negligible Borel set $N$ such that, for every
	$x_0\in\Rd\setminus N$, the set $\text{Tan}(\mu,x_0)$ is non empty, the convergence \eqref{blow-div} holds at $x_0$ for $\sig=P_\mu^l$ and every $l\in\{1,\dots,d\}$, and $x_0$ is a $\mu$-Lebesgue point of $P_\mu$.

	Let $\nu\in \text{Tan}(\mu, x_0)$. Then there exist positive reals $c_j$ and $r_j\to 0$ such that \eqref{def:tangentPreiss} holds for any $\f\in\D(\Rd)$.
	Fix such an element $\f$ and a finite radius $R>0$ so large that $\spt(\f)\subset B_R$. In virtue of
	\eqref{ball-esti}, $\kappa_j:=c_j\,\mu\left(B(x_0,r_j\,R)\right)$ remains bounded and strictly positive.
	Thus, possibly after extracting a subsequence, there exists a finite constant $\kappa>0$ (depending on
	$R$) such that
	$$ 0< \kappa:= \limsup_j \kappa_j < +\infty.$$
	We claim that
	\begin{equation}\label{claimMKP}
	\partial_z \nu = 0 \quad \text{in the sense of distributions, for all $z \in V \coloneqq T_\mu(x_0)$}.
	\end{equation}
	Fix $z \in V$. Applying \eqref{blow-div} with $\sig=P_\mu^l$ and $h=R\,r_j$, we get for every
	$l\in\{1,\dots,d\}$:
	\begin{align*} 
		0 &= \lim_{j\to \infty} \fint_{B(x_0,R r_j)} \pairing{P_\mu^l(x), \nabla \f\left(\frac{x-x_0}{R r_j}\right)} \, \mu(dx) \\
		&= \lim_{j\to \infty} \fint_{B(x_0,R r_j)} \pairing{P_\mu^l(x_0), \nabla \f\left(\frac{x-x_0}{R r_j}\right)}\, \mu(dx) \\
		&= \frac1{\kappa(x_0,R)}\, \pairing{P_\mu^l(x_0), \int \nabla \f(y/R)\, \nu(dy)},
	\end{align*}
	where the second equality follows from $x_0$ being a $\mu$-Lebesgue point of $P_\mu$, and the
	third one from \eqref{def:tangentPreiss} applied to $\f_R=\f(\frac{\cdot}{R})$, taking into account
	\eqref{ball-esti}. As $\f$ is arbitrary in $\D(\Rd)$, the equality above yields $-\div(P_\mu(x_0) \, \nu) = 0$, which is equivalent to \eqref{claimMKP}.

	The fact that $\partial_z \nu = 0$ for all $v \in V \coloneqq T_\mu(x_0)$ implies that $\nu$ is invariant under translations by elements of $V$. Write $\Rd = V\times V^\perp$, fix a half-open unit cube $Q\subset V$ and set $\eta \in \M(V^\perp)$ by $\eta(B):=\nu(Q\times B)$. The $2^{jk}$ half-open dyadic subcubes of $Q$ of side $2^{-j}$ are disjoint, tile $Q$, and are translates of one another; by translation invariance they carry the same mass, so that $\nu(A\times B) = \H^k(A)\,\eta(B)$ whenever $A$ is a translate of such a subcube. Since every open subset of $V$ is a countable disjoint union of half-open dyadic cubes, additivity and regularity extend this to all Borel sets $A\subset V$, which is \eqref{slice}.
\end{proof}

\begin{corollary}
	Let $\mu$ be a $k$-rectifiable measure, that is of the form $\mu =\theta \, \H^k\res S$ where 
	$S$ is a countably $\H^k$ rectifiable subset of $\R^d$ and $\theta$ a positive $\H^k$ locally integrable function on $S$.
	Then, denoting $T_S(x)$ the classical tangent space (defined at $\H^k$ a.e. $x\in S$), the following equalities are true $\mu$ a.e.: 
	$$ T_\mu(x) = T_S(x), \quad \text{Tan}(\mu,x) = \{ c \H^k \res T_S(x)\ :\ c>0\}.$$
\end{corollary}

\begin{proof}
	The second equality is classical and also the fact that $\mu$ admits $T_S(x)$ as a tangent space in the sense of \eqref{def:tanplan}. The first equality follows from the latter and~\eqref{P<T}.
\end{proof}

\begin{remark}\label{k=d} 
	By applying Proposition~\ref{MK-Preiss} in the case $k=d$, we deduce that the Preiss $\text{Tan}(\mu,x_0)$ reduces to $\set{ \theta \, \L^d}{\theta > 0}$ for $\mu$ almost all $x_0$ (the measure $\eta$ appearing in \eqref{slice} reduces to a scalar). Unfortunately it is not sufficient to recover that $\mu \ll \L^d$ as stated in Proposition \ref{topdim}. In fact, in \cite{Preiss87}, David Preiss constructed a singular measure $\mu\in \M_+(\R^d)$ with $\text{Tan}(\mu,x_0) = \set{\theta \mathcal L^d}{\theta > 0}$.
\end{remark}

\section{Further properties of MK bundles and tangential calculus}\label{sec:further}

\subsection{Decomposition into components of constant dimension}

Denote $P_{\mu}$ the tensor-valued function from $\Rd$ to symmetric tensors $\Sdd$ associated to the orthogonal projection on $T_\mu(x)$. If $\mu$ is multi-dimensional (as it is often the case in the analysis of mechanical structures \cite{BBS}), this integer-valued function allows to split $\mu$ into components of constant dimension:
\begin{equation}\label{splitmu}
	\mu = \sum_{k=0}^d \mu_k, \qquad \mu_k := \mu \, \res  \big\{\dim  T_\mu = k\big\} \ ,
\end{equation}

Note that the components $\mu_k$ that we obtain here are mutually singular. The top dimension $k=d$ is special since $\mu_d$ is absolutely continuous with respect to the Lebesgue measure $\Ld$. Indeed, by adapting to flat chains the structure theorem for $\A$-free measures of Philippis and Rindler \cite{PR2016}, we prove the following.

\begin{proposition}\label{topdim} 
	Let $\mu\in \M_+(\Rd)$ be a Radon measure. Then
	$$ \dim \, T_\mu(x) = d \;\; \text{ $\mu$ a.e.} \quad \iff \quad \mu\ll \L^d .$$
\end{proposition} 

\begin{proof} 
	Assume that $\mu\ll \L^d $. Then, by applying the property \eqref{restriction}, we infer that $T_\mu(x)= T_{\L^d} (x)$ holds $\mu$ a.e.. The latter set is easily seen to be equal to $\Rd$, since any smooth compactly supported vector field $\sigma$ belongs to $\T_{\L^d}$. Taking $\sigma(x) \equiv \overline{\sigma}$ in an open ball, we deduce that $\overline{\sigma} \in T_{\mu}(x)$ for any $\overline{\sigma} \in \Rd$ and $\mu$-a.e. $x \in \Rd$.
	
	Conversely, let $\Lambda \subset \R^{d \times d}$ be the cone of all singular matrices. The distance function $g = \dist(\cdot,\Lambda)$ is $1$-Lipschitz, positively $1$-homogeneous and vanishes exactly on $\Lambda$. Define 
	\[
		\Phi(\lambda) \coloneqq\ \int g\Bigl(\frac{d\lambda}{d|\lambda|}\Bigr)\,d|\lambda|^s , \qquad \lambda  \in \M(\R^d;\R^{d \times d}) .
	\]
	Here $|\lambda|^s$ is the part of $|\lambda|$ that is singular with respect to $\L^d$.
	We observe that $\Phi$ is $1$-Lipschitz: Indeed, for $\lambda_1,\lambda_2 \in \M(\R^d;\R^{d \times d})$, let $\pi = |\lambda_1| + |\lambda_2|$ and write $\lambda_1 = f_1 \pi$, $\lambda_2 = f_2 \pi$. Because $\Phi$ is positively $1$-homogeneous and $g$ is $1$-Lipschitz we get
	\[
		| \Phi(\lambda_1) - \Phi(\lambda_2) |
		\ \le\ \int|g(f_1)-g(f_2)|\,d\pi 
		\ \le\ \int|f_1 - f_2|\,d\pi^s 
		\ \le \ \|\lambda_1-\lambda_2\|_\M .
	\]
	The result of De Philippis and Rindler~\cite[Thm 1.1]{PR2016} is equivalent with the functional $\Phi$ vanishing on $\Nd^d$. Therefore, it suffices to show that $\Phi(P_\mu \, \mu) = 0$. The density of $\Nd$ in $\Fd$ applied row-wise to $P_\mu\,\mu$ yields a sequence $(\lambda_j)_j \in \Nd^d$ satisfying $\|P_\mu\, \mu - \lambda_j\|_\M \to 0$. In particular $\Phi(P_\mu\,\mu) = \lim_{j \to \infty} \Phi(\lambda_j) = 0$.  Hence $d\lambda/d|\lambda|$ is singular for any flat chain, and we deduce that $T_{\mu}$ is at most $(d-1)$-dimensional $\mu$-a.e..
\end{proof}

\begin{corollary}[De Philippis--Rindler for flat chains]\label{DPRF1}
	Let $T_1,\dots,T_d \in \Fd$ and $\mathbf T = (T_1,\dots,T_d) \in \mathcal M(\R^d; \R^{d \times d})$, then
	\begin{align*}
		\mathrm{rank} \left(\frac{d\mathbf T}{d|\mathbf T|}\right) < d \qquad |\mathbf T|^s\text{-almost everywhere}.
	\end{align*}
\end{corollary}

\begin{proof}
	Let $\mu = |\mathbf T|$ and write $T_j = \sigma_j |\mathbf T|$ so that $\sigma_j \in \mathcal T_\mu$ for all $j = 1,\dots,d$. 
	By the top-dimensional estimate it follows that $\mu_d \ll \mathcal L^d$, which is precisely the sought assertion.
\end{proof}

\begin{remark}[Dimensional estimates] Proposition~\ref{topdim} extends to the following relation between tangential dimension and dimensional estimates:
\begin{align*}
		 \dim \, T_\mu(x) \ge k \;\; \text{ $\mu$ a.e.} \quad \implies \quad \mu\ll \mathcal I^k \ll \mathcal H^k ,
\end{align*}
where $\mathcal{I}^k$ is the $k$-dimensional integral-geometric measure and $k \in \{1,\dots,d\}$. The proof follows verbatim the proof of sufficiency in the proposition above, by replacing \cite[Theorem 1.1]{PR2016} with its geometric version \cite[Theorem~1.3]{ADHR} and with $g = \dist(\cdot,\Lambda_k)$, where $\Lambda_k$ is the cone of matrices with rank less than $k$.
\end{remark}

\subsection{Generalized varifold and mean curvature}

Let us identify the Grassmannian manifold $\G_d$ of $\Rd$ with the subset of those symmetric tensors $P \in \Sdd$ which are projection matrices (i.e. those with eigenvalues in $\{0,1\}$). Following \cite{BBF1,BBF2}, we can associate every Radon measure $\mu\in \M_+(\Rd)$ with a generalized varifold. This measure is on the product $ \Rd \times \Sdd$ whose first marginal is $\mu$. This varifold $\mathbf{V}_\mu$ is defined by
\begin{align*}
	\pairing{\mathbf{V}_\mu, \f} \ :=\ \int \f(x,P_\mu(x)) \, d\mu 
	\qquad \text{ for every } \f \in C_0(\Rd \times \Sdd).
\end{align*}
Then \eqref{splitmu} induces the decomposition $\mathbf{V}_\mu = \sum_{k=1}^d \mathbf{V}_{\mu_k}$ where $\mathbf{V}_{\mu_k}$ fits with the classical notion of $k$-varifold in $\Rd$ when $\mu_k$ is a $k$-rectifiable measure (see \cite{Allard72,Simon}). Following \cite{BBF1,BBF2} (see also \cite{Buet} for related weak notions of curvature), the first variation of $\mathbf{V}_\mu$ allows to define a generalized mean curvature vector $H(\mu)$ by setting:
\begin{align*}
	H(\mu) := \div (P_\mu \mu)
\end{align*}
where the divergence is taken in the distributional sense in $\Rd$. Note that, since every row of the tensor $P_\mu$ belongs to $\T_\mu$, divergence above is component-wise in the Arens-Eells space. Therefore $H(\mu)$ can be identified as an element of $(\AEd)^d$. 

\subsection{Operations on normal and tangent fields}

In this section, we investigate the modification of tangent and normal spaces when the measure $\mu$ is pushed forward by a Lipschitz function. This generalizes the deformation of tangent spaces in manifolds. At first, we consider the dual operation. Let $f \in \Lip(\R^d;\R^m)$, $\mu$ be a Radon measure on $\R^d$, and $\nu \coloneqq f_{\#} \mu$. Define the (tangential) pullback of $\zeta \in L^{\infty}_{\nu}(\mathbb{R}^m; \mathbb{R}^m)$ by $f$ as
\begin{align*}
	(f^* \zeta)(x) \coloneqq (\nabla_\mu f(x))^t (\zeta \circ f)(x) \qquad \text{for } \mu\text{-a.e. } x \in \R^d.
\end{align*}

Here $\nabla_\mu f(x)$ identifies with the matrix whose $i^{\mathrm{th}}$ line is given by $\nabla_\mu f_i$ (see \eqref{def:nabla_mu}). Note that $f^*$ is a linear operator from $L^{\infty}_{\nu}(\mathbb{R}^m; \mathbb{R}^m)$ to $L^{\infty}_\mu(\Rd; \Rd)$ with norm bounded by $\Lip(f)$.

\begin{theorem}\label{res:pullback_stable}
	The pullback operator $f^*$ sends $\N_\nu$ on $\N_\mu$. As a consequence, 
	\begin{align}\label{pullback_stable:statement}
		\nabla_\mu f(x)(T_\mu) \subset T_\nu(f(x)) \qquad \text{for } \mu \text{-a.e. } x \in \R^d.
	\end{align}
\end{theorem}

\begin{proof}
	Let $\zeta\in\N_\nu$. By \eqref{orthocrit}, we can choose a sequence $(u_n)_n \subset \D(\R^m)$ such that $u_n\weakstar 0$ in $\Lip(\R^m)$ and $\nabla u_n\to\zeta$ holds $\nu$-a.e. Since $u_n\to 0$ pointwise and $\Lip(u_n\circ f)\le\Lip(f)\Lip(u_n)$ is bounded, the sequence pulls back to $u_n\circ f\weakstar 0$ in $\Lip(\Rd)$. By the weak-star continuity of the tangential gradient (Theorem \ref{weakgrad}), this implies $\nabla_\mu(u_n\circ f)\weakstar 0$ in $(L^\infty_\mu)^d$. 
	
	For an arbitrary tangent field $\sig\in\T_\mu$, we might use assertion \ref{MKi:diff} of Theorem \ref{MK-infinitesimal} on each coordinate of $f$ to write
	\begin{align*}
		0
		&= \lim_{n\to\infty} \int\pairing{\nabla_\mu(u_n\circ f),\sig}\,d\mu 
		= \lim_{n\to\infty} \lim_{h \to 0} \int \dfrac{u_n(f(x + h \sig(x))) - u_n(f(x))}{h} \, d\mu \\
		&= \lim_{n\to\infty} \lim_{h \to 0} \int \pairing{\nabla u_n(f(x)),\dfrac{f(x + h \sig(x)) - f(x)}{h}} \, d\mu \\
		&= \lim_{n\to\infty} \int\pairing{\nabla u_n\circ f,\ \nabla_\mu f \ \sig}\,d\mu \ =\ \int\pairing{\zeta\circ f, \nabla_\mu f \,\sig}\,d\mu = \int \pairing{f^*\zeta, \sig}\,d\mu,
	\end{align*}
	The passage to the limit in the last line is justified by the dominated convergence theorem: $\nabla u_n\circ f\to\zeta\circ f$ holds $\mu$-a.e. (since $\nu=f_\#\mu$), $|\nabla u_n|$ is uniformly bounded and the majorant $|\nabla_\mu f \ \sig| \le \Lip(f)|\sig|$ belongs to $L^1_\mu$. 
	Since $\sig$ is arbitrary in $\T_\mu$, we infer that $f^*\zeta$ is orthogonal to $\T_\mu$. Hence the claimed inclusion $f^* (\N_\nu) \subset \N_\mu$. 
	
	To prove \eqref{pullback_stable:statement}, we localize. Let $g : \R^d \to \R^d$ be a $\mu$-measurable selection of $T_\mu$ and $\zeta = P_\nu^\perp a$ for $a \in \mathbb{R}^m$. For any measurable set $B\subset \mathbb{R}^d$, it holds that $\one_B g \in \T_\mu$ and 
	\begin{align*}
		f^* \zeta = (\nabla_\mu f)^t P_\nu^\perp a \circ f \in \N_\mu.
	\end{align*}
	The composition $P_\nu^\perp \circ f$ is measurable. The set $A$ of $\mu$-Lebesgue points of $g$, $\nabla_\mu f$ and $P_\nu^\perp \circ f$ is of full measure for $\mu$, and for each $x \in A$ and $B = \mathscr{B}(x,r)$ with $r > 0$, there holds
	\begin{align*}
		0
		&= \int \left<(\nabla_\mu f)^t P_\nu^\perp a \circ f, \one_B g\right> d\mu
		= \int_B \left<P_\nu^\perp(f(y)) a, \nabla_\mu f(y) g(y)\right> d\mu(y) \\
		&= \left<a, \int_B P_\nu^\perp(f(y)) \nabla_\mu f(y) g(y) d\mu(y)\right>.
	\end{align*}
	As $a \in \mathbb{R}^m$ is arbitrary, the integral vanishes. Dividing by $\mu(B)$ and passing to the limit as $r \to 0$, we get that $P_\nu^\perp(f(x)) \nabla_\mu f(x) g(x) = 0$, hence $\nabla_\mu f(x) g(x) \in T_\nu(f(x))$ on $A$. Since $g$ was arbitrary in $\T_{\mu}$, we conclude that $\nabla_\mu f(T_\mu) \subset T_\nu \circ f$ at $\mu$-a.e. point.
\end{proof}

The pullback allows to define a pushforward by duality, and to mirror the stability result of Theorem~\ref{res:pullback_stable} on the tangent bundle. Precisely, given $\sigma \in L^1_\mu(\R^d; \R^d)$ and $f \in \Lip(\R^d; \R^m)$ as before, there exists a unique vector field $f_* \sigma \in L^1_\nu(\R^m;\R^m)$ satisfying 
\begin{equation}\label{eq:push_duality}
	\int_{\R^m} \langle \zeta, f_* \sigma \rangle \, d\nu = \int_{\Rd} \langle f^* \zeta, \sigma \rangle \, d\mu 
	\qquad \text{for every } \zeta \in L^\infty_\nu(\R^m; \R^m).
\end{equation}

Explicitly, $f_* \sigma$ is the Radon-Nikodym derivative of the vector-valued measure $f_{\#} \left(\nabla_\mu f \sigma \mu\right)$ with respect to $\nu = f_{\#} \mu$. The duality formulation allows to recover the following result, which is the pendant of the stability of flat chains by (current) pushforward stated in \cite[4.1.14]{Fed}.

\begin{corollary}\label{pushflat} 
	The pushforward operator $f_*$ of $f \in \Lip(\R^d;\R^m)$ sends $\T_\mu$ on $\T_\nu$.
\end{corollary}

\begin{proof}
	Let $\sigma \in \T_\mu$ and $\zeta \in \N_{\nu}$. By Theorem~\ref{res:pullback_stable}, $f^* \zeta \in \N_\mu$, so that
	\begin{align*}
		\int \pairing{\zeta, f_* \sigma} d\nu 
		= \int \pairing{f^* \zeta, \sigma} d\mu
		= 0.
	\end{align*}
	As $\zeta$ is arbitrary, we conclude that $f_* \sigma \in \T_\nu$.
\end{proof}

As a corollary, one obtains an inclusion of tangent bundles constructed from tensor products. In fact, in this case, one can even get the following identity.

\begin{proposition}[Tensor products]\label{app:product}
	Let $d = p+q$ and $\mu \in \M_+(\R^p)$, $\nu \in \M_+(\R^q)$ be finite Radon measures. The tangent bundle of the product measure is the product of the tangent bundles:
	\begin{align*}
		T_{\mu\otimes\nu}(x,y) \ = \ T_\mu(x)\times T_\nu(y) \qquad \mu \otimes \nu \text{-a.e. } (x,y) \in \Rd.
	\end{align*}
\end{proposition}

\begin{proof}
	\emph{Inclusion $\subset$.} Let $\pi_1(x,y) \coloneqq x$ and $\pi_2(x,y) \coloneqq y$ be the coordinate projections, so that $(\pi_1)_\#(\mu\otimes\nu) = \nu(\R^q)\,\mu$ and $(\pi_2)_\#(\mu\otimes\nu) = \mu(\R^p)\,\nu$. Since scaling a measure by a positive constant does not change its bundle ($\T_{c\,\mu} = \T_\mu$ for $c>0$), these two pushed-forward measures have bundles $T_\mu$ and $T_\nu$ respectively. Moreover, the tangential gradient of the linear maps $\pi_i$ can be computed using \eqref{def:Amu} as $\nabla_\mu \pi_i = P_\mu \nabla \pi_i$. Hence $\nabla_\mu \pi_i(T_\mu) = \nabla \pi_i(T_\mu) = \pi_i(T_\mu)$. Applying Theorem~\ref{res:pullback_stable}, we obtain
	\[
	\pi_1\big(T_{\mu\otimes\nu}(x,y)\big)\subset T_\mu(x),
	\qquad
	\pi_2\big(T_{\mu\otimes\nu}(x,y)\big)\subset T_\nu(y)
	\qquad \mu\otimes\nu\text{-a.e.}
	\]
	The inclusion $T_{\mu\otimes\nu}(x,y) \subset T_\mu(x)\times T_\nu(y)$ follows immediately.
	
	\noindent \emph{Inclusion $\supset$.} Conversely, let $\tau\in\T_\mu$ and set $\sig(x,y) \coloneqq (\tau(x),0)$. For any smooth test function $\f\in\D(\Rd)$, Fubini's theorem gives
	\[
	\pairing{-\div(\sig\,\mu\otimes\nu), \f} \ =\ \int_{\R^q} \pairing{-\div(\tau\mu),\ \f(\cdot,y)}\, d\nu(y).
	\]
	If $\f_n \weakstar 0$ in $\Lip(\Rd)$, then for each fixed $y$, $\f_n(\cdot,y)\weakstar 0$ in $\Lip(\R^p)$, meaning the inner bracket tends pointwise to $0$. It is moreover uniformly bounded by $\Vert\div(\tau\mu)\Vert_{\KR}\,\sup_n\Lip(\f_n)$ by Proposition~\ref{completion}. As $\nu$ is finite, dominated convergence yields $\pairing{-\div(\sig\,\mu\otimes\nu),\f_n}\to0$, and the continuity criterion \eqref{criterion} implies $\sig\in\T_{\mu\otimes\nu}$, i.e., $(\tau(x),0)\in T_{\mu\otimes\nu}(x,y)$ almost everywhere. Choosing a countable generating family $(\tau_k)\subset\T_\mu$ such that $T_\mu(x) = \mathrm{span}\set{\tau_k(x)}{k\in\NN}$ $\mu$-a.e. (as in Subsection~\ref{prel:multif}) and intersecting the exceptional sets, we obtain $T_\mu(x)\times\{0\}\subset T_{\mu\otimes\nu}(x,y)$ a.e.. The symmetric argument yields $\{0\}\times T_\nu(y)\subset T_{\mu\otimes\nu}(x,y)$, and the conclusion follows since $T_{\mu\otimes\nu}(x,y)$ is a vector subspace.
\end{proof}

\subsection{Tangential chain rule}

The computation in the proof of Theorem~\ref{res:pullback_stable} used a chain rule of the form $\pairing{\nabla_\mu(u_n\circ f),\sig} = \pairing{\nabla u_n\circ f,\ \nabla_\mu f \ \sig}$ for a smooth $u$ and $\sig \in \T_\mu$. The following result shows that the chain rule holds with tangential gradients.

\begin{theorem}[Tangential chain rule]\label{lipchain:new}
	Let $g \in \Lip(\R^m;\R^k)$. Then, for every $\sig\in\T_\mu$,
	\begin{equation}\label{lipchain:id}
		\nabla_\mu(g\circ f)(x)\,\sig(x)\ =\ \nabla_\nu g(f(x))\,\nabla_\mu f(x)\,\sig(x) \qquad \mu\text{-a.e.}.
	\end{equation}
\end{theorem}

\begin{proof}
	Arguing componentwise, we may assume $g$ is scalar. By Remark \ref{Lipdual}, we can choose $g_k\in\D(\R^m)$ such that $g_k\weakstar g$ in $\Lip(\R^m)$. As $g_k\circ f\weakstar g\circ f$ in $\Lip(\Rd)$, the weak-star continuity of the tangential gradient (Theorem \ref{weakgrad}) ensures that
	\begin{equation}\label{doubleweak}
		\nabla_\nu g_k \weakstar \nabla_\nu g \quad \text{in } L^\infty_\nu, \qquad \text{and} \qquad \nabla_\mu (g_k \circ f)\weakstar \nabla_\mu (g \circ f) \quad \text{in } L^\infty_\mu.
	\end{equation}
	Fix $\phi \in L^\infty_\mu$ and define the tangent field $\tau \coloneqq \phi \, \sig \in \T_\mu$. Then, arguing as in Theorem~\ref{res:pullback_stable},
	\begin{equation}\label{eq:chain_proof_1}
		\int \pairing{\nabla_\mu(g_k\circ f),\tau}\,d\mu 
		\ = \ \int \pairing{\nabla g_k \circ f,\nabla_\mu f \, \tau}\,d\mu
		\ = \ \int \pairing{\nabla g_k, f_*\tau} d\nu.
	\end{equation}
	We now pass to the limit as $k \to \infty$. On the target space, since $f_*\tau \in \T_\nu \subset (L^1_\nu)^m$ by Corollary \ref{pushflat}, the weak-star convergence of $\nabla g_k$ yields:
	\[
	\lim_{k\to\infty} \int \pairing{\nabla g_k, f_*\tau}\,d\nu \ =\ \int \pairing{\nabla_\nu g, f_*\tau}\,d\nu \ =\ \int \pairing{\nabla_\nu g \circ f, \nabla_\mu f \, \tau}\,d\mu.
	\]
	On the source space, the weak-star convergence of $\nabla_\mu(g_k \circ f)$ implies the left-hand side of \eqref{eq:chain_proof_1} converges to $\int \pairing{\nabla_\mu(g\circ f), \tau} d\mu$. Equating the limits and recalling that $\tau = \phi \sig$, we get
	\[
		\int \pairing{\nabla_\mu(g\circ f), \sig} \phi \,d\mu \ =\ \int \pairing{\nabla_\nu g \circ f, \nabla_\mu f \ \sig} \phi \, d\mu.
	\]
	Since $\phi \in L^\infty_\mu$ is arbitrary, the identity holds $\mu$-a.e..
\end{proof}

\subsection{The case of bi-Lipschitz maps}

By the inclusion \eqref{pullback_stable:statement}, there always holds
\begin{equation}\label{app:rank}
	\dim T_{f_\#\mu}(f(x)) \ \ge\ \dim(\nabla_\mu f(T_\mu(x))) = \mathrm{rank}\big(\nabla_\mu f(x)\big) \qquad \mu\text{-a.e..}
\end{equation}
We show that if $f$ is bi-Lipschitz on its image, then equality holds. The Example~\ref{app:counterex} below illustrates that injectivity is not sufficient. 

\begin{proposition}\label{app:bilip}
	Let $\mu\in\M_+(\Rd)$ be finite and let $f \in \Lip(\Rd;\R^m)$ with $m \geqslant d$. Assume that $f$ is injective, and that its left inverse $f^{-1} : f(\Rd) \to \Rd$ is Lipschitz. Then, denoting $\nu \coloneqq f_\#\mu$, we have
	\begin{align}\label{bilip:statement}
		\nabla_\mu f(x)(T_\mu(x)) = T_\nu(f(x))
		\qquad \text{for } \mu\text{-a.e. } x \in \mathbb{R}^d.
	\end{align}
\end{proposition}

\begin{proof}
	Let $g \in\Lip(\R^m;\Rd)$ be a Lipschitz extension of $f^{-1}$, which exists by Kirszbraun's theorem. Then $f \circ g = \mathrm{id}$ at $\nu$ almost any point. Passing to the $\nu$-tangential gradient in the latter equality thanks to \eqref{localgrad:statementu}, and using the chain rule of Theorem~\ref{lipchain:new}, we get
	\begin{align*}
		P_\nu(y)
		= \nabla_\nu (\mathrm{id})(y)
		= \nabla_\nu (f \circ g)(y)
		= \nabla_\mu f(g(y)) \nabla_\nu g(y).
	\end{align*}
	Since $P_\nu(y)(w) = w$ for any $w \in T_\nu(y)$, we deduce that $\nabla_\mu f(g(y))$ is surjective on $T_\nu(y)$. Hence, using that $\nabla_\mu f(g(y))(N_\mu) = \{0\}$, there holds 
	\begin{align*}
		\dim(\nabla_\mu f(g(y))(\Rd))
		= \dim(\nabla_\mu f(g(y))(T_\mu)) 
		\geqslant \dim(T_\nu(y))
		\qquad \text{at } \nu \text{ a.e. } y.
	\end{align*}
	As $\mu = g_{\#} \nu$ and $\nu = f_{\#} \mu$, integrating against $\nu$ yields 
	\begin{align*}
		0
		\leqslant \int \left[\dim(\nabla_\mu f(g(y))(T_\mu)) - \dim(T_\nu(y))\right] d\nu 
		= \int \left[\dim(\nabla_\mu f(T_\mu)) - \dim(T_\nu(f(x)))\right] d\mu.
	\end{align*}
	However, in virtue of Theorem~\ref{res:pullback_stable}, we know that $\nabla_\mu f(x)(T_\mu(x)) \subset T_\nu(f(x))$ holds  $\mu$ almost everywhere. The wished equality follows.
\end{proof}

To conclude, we construct an injective Lipschitz function which is not bi-Lipschitz, and for which the equality \eqref{bilip:statement} does not hold.

\begin{example}\label{app:counterex}
	Let $E\subset[0,1]$ be a compact set with empty interior and $\Ld(E)>0$ (a \emph{fat} Cantor set), and define
	\[
		f(t) \ \coloneqq\ \int_0^t \one_{E^c}(s)\,ds, \qquad t\in\R.
	\]
	Then $f$ is $1$-Lipschitz and, since $E$ has empty interior, $f(b)-f(a) = \Ld([a,b]\setminus E)>0$ for $a<b$. Thus, the map $f$ is strictly increasing, hence injective. Take $\mu \coloneqq \Ld\res[0,1]$, so that $T_\mu(x) = \R$ for every $x$ by Proposition~\ref{topdim}. Denote $\nu \coloneqq f_\#\mu$. 
	
	Since $f' = \one_{E^c}$ a.e., the area formula gives $\Ld(f(E)) \le \int_E |f'| = 0$, while injectivity gives $\nu(f(E)) = \mu(E) = \Ld(E)>0$. Thus $\nu\res f(E)$ is a non-zero measure singular with respect to $\Ld$. If the set $A \coloneqq \set{y}{\dim T_\nu(y) = 1}$ met $f(E)$ in positive $\nu$-measure, then by \eqref{locality} and Proposition~\ref{topdim} the measure $\nu\res(A\cap f(E))$ would be both absolutely continuous with respect to $\Ld$ and concentrated on the $\Ld$-null set $f(E)$, hence zero. Therefore
	\[
		\dim T_\nu(f(x)) = 0 < 1 = \dim T_\mu(x) \qquad \text{for } \mu\text{-a.e. } x\in E.
	\]
	Consistently with \eqref{app:rank}, the tangential differential degenerates here: $\nabla_\mu f = f' = 0$ holds $\Ld$-a.e. on $E$.
\end{example}

\subsection{Weaver derivation}\label{sec:weaver}

We point a connection with the analysis on metric spaces already noticed in \cite{lucicCharacterisationUpperGradients2021} for the tangent space therein. Let $\Lip_b$ stand for Lipschitz and bounded functions. The tangential gradient map $\nabla_\mu : \Lip_b(\R^d) \to (L^\infty_\mu)^d$ is a derivation in the sense of Weaver, i.e. satisfies
\begin{enumerate}
	\item \emph{Weak-star continuity:} $\nabla_\mu u_n \weakstar \nabla_\mu u$ whenever $u_n \weakstar u$ in $\Lip_b(\R^d)$;
	\item \emph{Leibniz rule:} $\nabla_\mu(uv) = u \nabla_\mu v + v \nabla_\mu u$ for all $u,v \in \Lip_b(\R^d)$.
\end{enumerate}

\begin{proof} 
	The weak-star continuity is a direct consequence of Lemma~\ref{closable} and the linearity of $\nabla_\mu$.
	The Leibniz rule follows by approximation with test functions and the smooth Leibniz rule: if $u_n \weakstar u, v_n \weakstar v$ in $\Lip_b(\R^d)$, then by dominated convergence $u_n\nabla_\mu v_n \weakstar u\nabla_\mu v$ and $v_n\nabla_\mu u_n \weakstar v\nabla_\mu u$ in $(L^\infty_\mu)^d$. On the other hand, $\nabla_\mu(u_nv_n) \weakstar \nabla_\mu(uv)$, and
	\[
		\nabla_\mu(u_nv_n) = P_\mu (v_n \nabla u_n + u_n \nabla v_n) = v_n \nabla_\mu u_n + u_n \nabla_\mu v_n \weakstar v\nabla_\mu u + u \nabla_\mu v
		\]
	in $(L^\infty_\mu)^d$. Uniqueness of the weak-star limit yields the assertion.
\end{proof}

\end{document}